\documentclass[leqno,11pt]{article}
\usepackage{amsfonts}
\usepackage[utf8]{inputenc}
\usepackage[T1]{fontenc}
\usepackage{lmodern}
\usepackage{amsthm}
\usepackage{amsmath,amssymb}
\usepackage{graphicx,floatrow,subfig}
\usepackage{multirow,makecell}
\usepackage{pifont}
\usepackage{enumerate}
\usepackage{mathtools}
\usepackage{mathrsfs}
\usepackage{xcolor}
\usepackage{geometry}
\usepackage{hyperref}
\usepackage{color}
\usepackage{bm}
\usepackage{multirow,makecell}
\usepackage{mathtools}
\allowdisplaybreaks
\allowdisplaybreaks[4]
\newtheorem{thm}{Theorem}
\newtheorem{lm}{Lemma}
\newtheorem{pro}{Proposition}

\newtheorem{defi}{Definition}

\newtheorem{rem}{Remark}
\newtheorem{cor}{Corollary}
\title{\bf Cyclicity of lips and the center--focus problem at infinity}
\author{{\small Hebai Chen$^1$,\ \ \, Dehong Dai$^1$, \ \ \, Vadim Kaloshin$^{2}$, \ \ \, Zhijie Li$^{1}$}\\
		\footnote{Email: chen\_hebai@csu.edu.cn (H. Chen), dai\_dehong@csu.edu.cn (D. Dai), vadim.kaloshin@gmail.com (V. Kaloshin,
			corresponding author), li\_zhijie@csu.edu.cn (Z. Li)}	
	{\footnotesize{\em{$^1$School of Mathematics and Statistics, HNP-LAMA, Central South University, Changsha 410083, China}}}
	\\
	{\footnotesize {\em{$^2$
Institute of Science and Technology Austria (ISTA), Klosterneuburg 3400, Austria
}}}}

\date{}
\begin{document}
    \maketitle

    %%%%%%%%%%%%%%%%%
    \begin{abstract}
	In this work, we discover an intriguing Li\'{e}nard system that is integrable near infinity and near zero and has Ilyashenko--Kotova lips in between. Moreover, we give a necessary and sufficient condition for a polynomial Li\'{e}nard system to have a center at infinity. In a different direction, we improve a lower bound for the Hilbert number of Li\'{e}nard systems. The improvement is due to the development of Brudnyi's method for calculating the cyclicity of centers near infinity and near zero.
    \end{abstract}
	
\tableofcontents	
	
	\setlength{\baselineskip}{13.6pt}
	\setlength{\parskip}{0pt}

	%%%%%%%%%%%%%%%%%%%%%%%%%%%%%%%%%%%%%%%%%%%%%%%%%%%%%%%%%%%%%%%%%%%%%%%%%%%%%%%%%%%%%%%%%%%%%%%%%
		\section{Introduction and main results}
{
The pioneering works of Poincar\'{e} \cite{P} and Dulac \cite{Dulac}
introduced the concept of a center.  Since then, distinguishing a center from
a focus has remained a central problem in the qualitative theory of planar
differential systems, commonly referred to as the center--focus problem
\cite{BPY,TY}.  The corresponding problem at infinity has also received
considerable attention \cite{BR,BRY}.  Throughout this paper, centers and foci
at infinity are understood in the Poincar\'{e} disc; the compactification
charts and the precise definitions used here are recalled in
Appendix~\ref{app:poincare}.
}

A widely studied system with both theoretical and practical significance is the generalized polynomial Li\'{e}nard system \cite{LMP, Smale91, Smale98, ZDHD}:
\begin{equation} 
\dot x=y,\qquad
\dot y=-g(x)-f(x)y,
\label{ls}
\end{equation}
where 
$$g(x)=\sum_{i=k}^{m}a_{i}x^{i},~~~~~~  f(x)=\sum_{i=s}^{n}b_{i}x^{i}$$
with $m, n, k, s\in\mathbb{N}$ and $a_{k}a_{m}b_{s}b_{n}\neq0$, and 
the dot represents the derivative of the independent variable $t$.

In this paper, we prove two types of results concerning the above system:
\begin{itemize}
\item ({\it center at infinity}) We give a necessary and sufficient condition for \eqref{ls}
to have a center at infinity (Theorem \ref{cf-thm1}). For $\deg f=1$ and $\deg g=3,5$ 
we give a complete topological classification of all possible phase portraits 
(see Figs. \ref{fig:m3n1} and \ref{fig:m5n1}).

An interesting feature of one phase portrait is a lips configuration (Fig.~\ref{lip-Ham}). 
This phase portrait was discovered for abstract planar vector fields by Ilyashenko--Kotova 
and has surprisingly large semilocal cyclicity (see Section~\ref{sec:lips}).  

\item ({\it improved estimate for the Hilbert number}) We improve known bounds for the so-called 
Hilbert number (cyclicity of \eqref{ls}) for $\deg g=5$ and large $\deg f$ (Theorem \ref{maint}). 

The improvement is due to the application of two new methods: the Brudnyi cyclicity
method near infinity and the Xiong--Han method near two cuspidal separatrices.
We also apply Brudnyi cyclicity near zero (Theorem \ref{t-hn5}), but this does not improve the lower bound on the Hilbert number. 
\end{itemize}

\subsection{Centers and foci at infinity}

	 The center--focus problem at infinity for polynomial Li\'{e}nard systems was 
	 initially proposed by Dumortier and Herssens in \cite{DH}, but has not been fully resolved. 		   				 
 They provided a classification of the behavior near infinity using Poincar\'{e} and Poincar\'{e}-Lyapunov compactifications, showing that the dynamical behavior depends critically on $a_m$ and $b_n$, except in the case where 
\begin{center}
   $m, n$ are odd, $m>2n+1$, $a_m>0$ or $m=2n+1$, $n$ is odd and $4(n+1)a_mb_{n}^{-2}>1$.
   \end{center}
   In such cases, the dynamical behavior also depends on the values of the other coefficients $a_i$ and $b_j$.

 Center singularities in planar differential systems can be categorized into distinct types based on their linear approximations and nonlinear characteristics. See the definitions of centers of linear and nilpotent types in Appendix~\ref{app:poincare}.

Recently, Llibre and Valls \cite{LV} presented a necessary and sufficient condition for 
the generalized polynomial Li\'{e}nard system \eqref{ls} to have a global center of linear type.
Following their work, the first author and co-workers \cite{CLZ} provided a necessary and 
sufficient condition for the generalized polynomial Li\'{e}nard system \eqref{ls} to have a global 
center of either linear or nilpotent type. They successfully solved the center--focus problem 
at infinity for the generalized polynomial Li\'{e}nard system \eqref{ls} with a single equilibrium.
However, these methods do not extend to systems with multiple equilibria,
since the function $\int_{0}^{x}g(x){\rm d}x$ does not exhibit the property of being monotonically 
decreasing on $(-\infty,0)$ and monotonically increasing on $(0,\infty)$ in such cases.

\subsection{Center--focus problem}
Following \cite{DH}, we address the following fundamental question in this paper:

{\it Can we give necessary and sufficient conditions for a generalized polynomial 
	Li\'{e}nard system to have a center at infinity, thereby solving 
	the center--focus problem at infinity?}
	
{
	Every finite equilibrium of system \eqref{ls} has the form $(x_*,0)$ with
	$g(x_*)=0$.  Its index belongs to $\{-1,0,1\}$, and the sum of the indices of
	all finite equilibria is $1$ under the hypotheses considered below.  We shall
	therefore distinguish equilibria in the right half-plane $x>0$ and in the
	left half-plane $x<0$.  The corresponding index configurations are derived in
	Section~\ref{sec:proof}.
}

Building upon this classification, we now present the main result, which answers the center--focus problem at infinity affirmatively. 

{
	For system~\eqref{ls}, every finite equilibrium is of the form
$
	E_{x_*}=(x_*,0)
$ and
	$g(x_*)=0.
	$
	For an isolated equilibrium \(E\), we denote its Poincar\'e index by
	\(\operatorname{ind}(E)\).
	
	\begin{thm}\label{cf-thm1}
		Assume that \(g(0)=0\), that all finite equilibria of system~\eqref{ls}
		are isolated, and that
	$
		\operatorname{ind}(O)\in\{-1,1\}$ and
	$O=(0,0).
	$
		Then \(P_\infty\) is a center at infinity of system~\eqref{ls} if and only
		if the following conditions hold:
		\begin{description}
			
			\item[(i)] The integers \(m\) and \(n\) are odd, and either
			$
			m>2n+1,$ $a_m>0,
			$
			or
			$
			m=2n+1,$ $
			4(n+1)a_m b_n^{-2}>1.
			$
			
			\item[(ii)] One has
			$
			F(x_1)=F(x_2)
			$
			whenever
			$
			G(x_1)=G(x_2)$ for
			$ x_1<0<x_2,
			$
			where
			$
			F(x)=\int_0^x f(\xi)\,{\rm d}\xi,
			$
			$G(x)=\operatorname{sgn}(x)
			\int_0^x |g(\xi)|\,{\rm d}\xi.
			$
			
		\end{description}
	\end{thm}
	
	{
		\begin{rem}\label{rem:conditions-not-automatic}
			The conditions in Theorem~\ref{cf-thm1} are not automatically satisfied by
			every generalized Li\'enard system.
			
			For example, consider
		$
			f(x)=x+1,
		$ and
			$g(x)=x(x^2+1)^2.
		$
			The origin is the unique finite equilibrium.  Moreover, the Jacobian at the origin satisfies
			\[
			D_o
			=
			\begin{pmatrix}
				0&1\\
				-1&-1
			\end{pmatrix},
			\qquad
			\det D_o=1,
			\]
			and hence
			\(
			\operatorname{ind}(O)=1.
			\)
			Thus all the hypotheses preceding conditions~{\rm(i)} and~{\rm(ii)} in
			Theorem~\ref{cf-thm1} are satisfied.
			
			In this example,
			$m=5$, $n=1$, $a_5=b_1=1$.
			Therefore,
			$
			m>2n+1$, $a_5>0,
			$
			and condition~{\rm(i)} is satisfied.  On the other hand,
			\[
			F(x)
			=
			\int_0^x(\xi+1)\,{\rm d}\xi
			=
			\frac{x^2}{2}+x,
			\]
			whereas
			\[
			G(x)
			=
			\operatorname{sgn}(x)
			\int_0^x |\xi|(\xi^2+1)^2\,{\rm d}\xi
			=
			\frac{(x^2+1)^3-1}{6}.
			\]
			Hence, for every \(x>0\),
		$
			G(-x)=G(x),
			$
			but
			$
			F(-x)
			\neq
			F(x).
			$
			Thus condition~{\rm(ii)} is not satisfied, and infinity is not a center. This example shows that the conditions in
			Theorem~\ref{cf-thm1} are not automatically
			satisfied by every generalized Li\'enard system.
		\end{rem}
}}

If $g(x)$ is an odd function, we obtain a simplified criterion
\begin{cor}\label{cf-cor1}
Assume that $g(x)$ is odd. Then
$P_{\infty}$ is a center at infinity  of  system \eqref{ls} if and only if the condition {\bf(i)} in 
{Theorem \ref{cf-thm1} holds and $f(x)$ is odd.} 
\end{cor}
When $n=1$, we have the following theorem.
\begin{thm}\label{cf-thm2}
When $n=1$, $P_{\infty}$ is a center at infinity  of  system \eqref{ls} if and only if 
the system can be rescaled to the form:
\begin{equation*}
\dot{x}=y,\qquad
\dot{y}=-g(x)-bxy,
\end{equation*}
where $g(x)$ is an odd function with $m\geq 3$ and $a_m$, $b$ satisfy $a_m>b^2/8>0$ when $m=3$ and $a_m>0$ when $m>3$.
\end{thm}

We also provide explicit global phase portraits for specific cases in the Poincar\'{e} disc:
\begin{cor}\label{cf-cor2}
When $n=1, m=3$, $P_{\infty}$ is a center at infinity of  system \eqref{ls} if and only if the variables $x$, $y$, 
and $t$ can be rescaled such that the system takes the form of 
\begin{equation}
	\dot{x}=y,\qquad
	\dot{y}=-(ax+x^3)-bxy,
\label{m3n1}
\end{equation}
where $a\in\mathbb{R}$ and $0<b<2\sqrt{2}$.
Moreover, there are only two classes of global phase portraits in the Poincar\'{e} disc, as shown in {\rm Fig. \ref{fig:m3n1}}.
\end{cor}
 \begin{figure}[!htb]
		\centering
		\subfloat[$a\geq0$]
		{\includegraphics[scale=0.7]{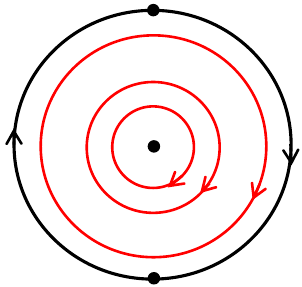}}\hspace{25pt}
		\subfloat[$a<0$]
		{\includegraphics[scale=0.7]{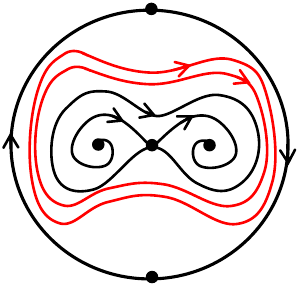}}\hspace{25pt}
		\caption{{\footnotesize Global phase portraits in the Poincar\'{e} disc of system \eqref{m3n1}.}}
		\label{fig:m3n1}
	\end{figure}
\begin{cor}\label{cf-cor3}
When $m=5$, $P_{\infty}$ is a center at infinity of  system \eqref{ls} if and only if the variables $x$, $y$, and $t$ can be rescaled such that the system takes the form of 
\begin{equation}
\dot{x}=y,\qquad
\dot{y}=-(ax+cx^3+x^5)-bxy,
\label{m5n1}
\end{equation}
where $a, c\in\mathbb{R}$ and $b>0$.
Moreover, the six parameter regions below give five topologically distinct global phase portraits in the Poincar\'e disc, as shown in {\rm Fig. \ref{fig:m5n1}}, and the bifurcation diagram as shown in {\rm Fig. \ref{bi-fig}}, 
where 
\begin{equation*}
        \begin{aligned}
        \mathcal{G}_1&:=\left\{(a, c, b)\in\mathbb{R}^2\times\mathbb{R}^{+}: c>0,a>0,~{\rm or}~c>\frac{b^2}{8}, a=0,~{\rm or}~c=0,a\geq0,~{\rm or}~c<0,a>\frac{c^2}{4}\right\},
        \\
        \mathcal{G}_2&:=\left\{(a, c, b)\in\mathbb{R}^2\times\mathbb{R}^{+}: 0<c\leq\frac{b^2}{8}, a=0\right\},
        \\
        \mathcal{G}_3&:=\left\{(a, c, b)\in\mathbb{R}^2\times\mathbb{R}^{+}: a<0\right\},
        \\
        \mathcal{G}_4&:=\left\{(a, c, b)\in\mathbb{R}^2\times\mathbb{R}^{+}: c<0,a=0\right\},
        \\
         \mathcal{G}_5&:=\left\{(a, c, b)\in\mathbb{R}^2\times\mathbb{R}^{+}: c<0, a=\frac{c^2}{4} \right\},
         \\
         \mathcal{G}_6&:=\left\{(a, c, b)\in\mathbb{R}^2\times\mathbb{R}^{+}: c<0,0<a<\frac{c^2}{4}\right\}.
   \end{aligned}
	\end{equation*}
\end{cor}
 \begin{figure}[!htbp]
		\centering
		\subfloat[$\mathcal{G}_1$]
		{\includegraphics[scale=0.7]{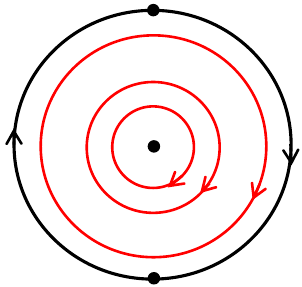}}\hspace{25pt}
		\subfloat[$\mathcal{G}_2$]
		{\includegraphics[scale=0.7]{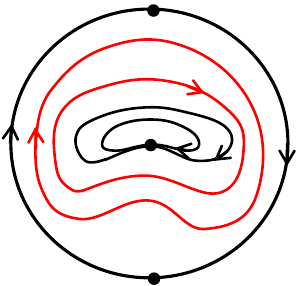}}\hspace{25pt}
		\subfloat[$\mathcal{G}_3\cup\mathcal{G}_4$]
		{\includegraphics[scale=0.7]{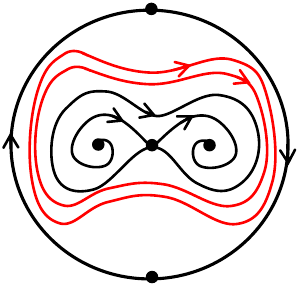}}\hspace{25pt}\\
		\subfloat[$\mathcal{G}_5$]
		{\includegraphics[scale=0.7]{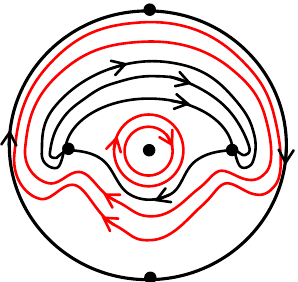}}\hspace{25pt}
		\subfloat[$\mathcal{G}_6$]
		{\includegraphics[scale=0.7]{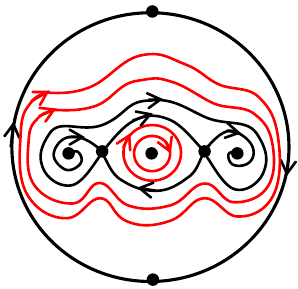}}\hspace{25pt}
		\caption{{\footnotesize Global phase portraits in the Poincar\'{e} disc of system \eqref{m5n1}.}}
		\label{fig:m5n1}
	\end{figure}

 \begin{figure}[!htbp]
	\centering
	{\includegraphics[scale=0.22]{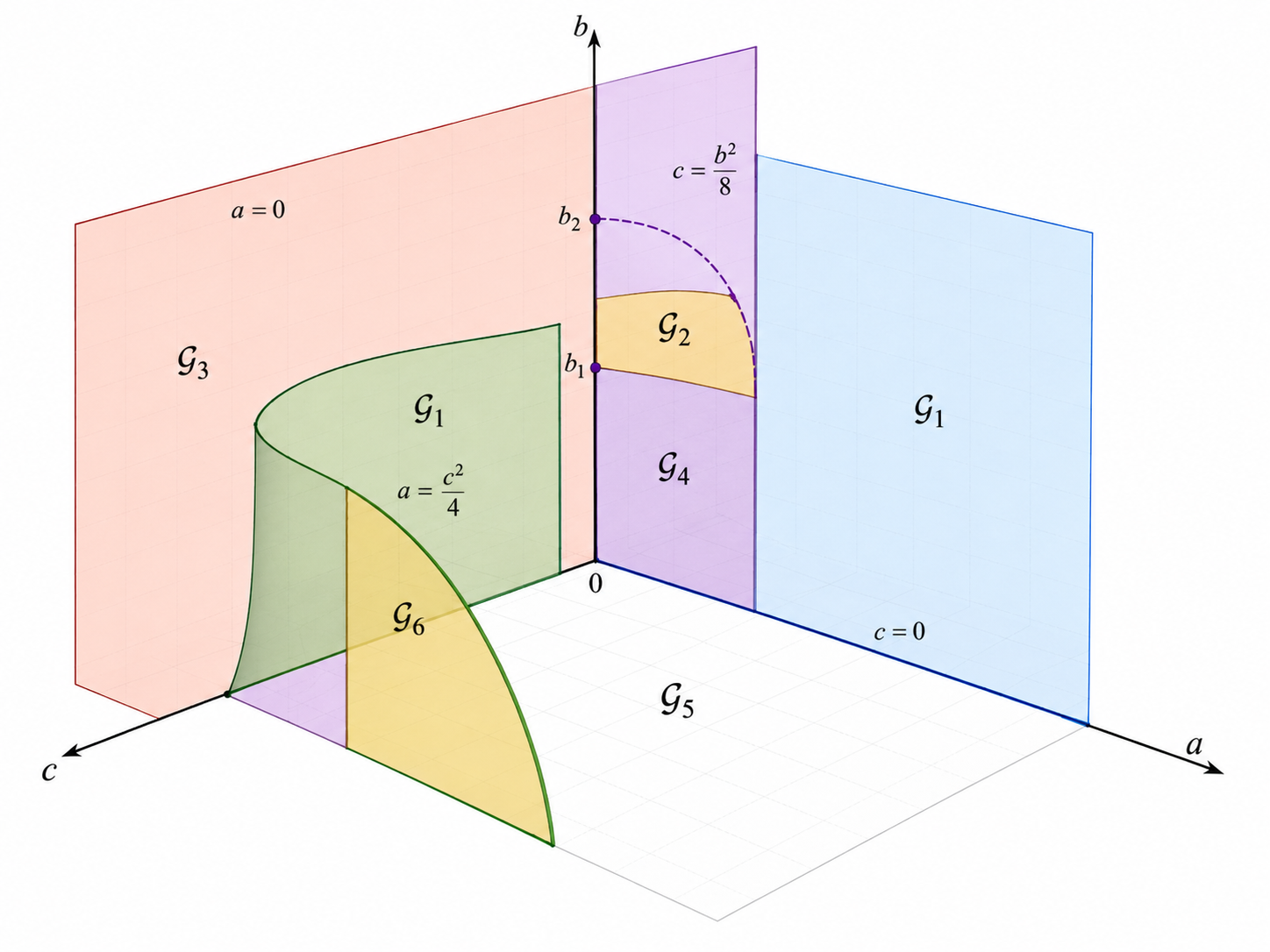}}\hspace{2pt}
	\caption{Bifurcation region in $(a,b,c)$-space for $b>0$.}
	\label{bi-fig}
\end{figure}

\subsection{Ilyashenko--Kotova lips and Hilbert numbers}

We first introduce an intriguing structure known as the Ilyashenko-Kotova lips.
\begin{defi}
				The Ilyashenko-Kotova lips refer to the ensemble of a codimension $3$ polycycle consisting of two saddle-nodes connected by a separatrix
connection.
	\end{defi}

{

The bifurcation of polycycles in families with a small number of parameters
is an important topic in planar bifurcation theory and is closely related to
the Hilbert--Arnold problem for vector fields; see \cite{IL,I}.  In particular,
Fig.~\ref{fig:m5n1}(d) contains an Ilyashenko--Kotova lips configuration,
which is shown separately in Fig.~\ref{lip-origin}.

An Ilyashenko--Kotova lips configuration is a codimension-three polycycle
formed by two saddle-node singularities and two heteroclinic separatrix
connections joining them.  Consequently, its generic unfolding requires
three independent parameters.

Let
\(
\mathcal X=\{X_\mu\}_{\mu\in\mathbb R^3}
\)
be a three-parameter family of smooth planar vector fields, and suppose that
\(X_{\mu_*}\) contains a lips polycycle \(\Gamma\) at a critical parameter
value \(\mu_*\).  The cyclicity
\(
\operatorname{Cycl}(\Gamma,\mathcal X)
\)
is the maximal number of distinct limit cycles that can simultaneously
bifurcate from \(\Gamma\) for a parameter value arbitrarily close to
\(\mu_*\).

Following \cite{I,IL}, let \(C(3)\) denote the supremum of the total number of
limit cycles that can bifurcate from all the polycycles of a vector field
corresponding to a critical parameter value in a typical three-parameter
family.  The following result shows that the unboundedness of \(C(3)\) is
already realized by the lips configuration.

\begin{pro}[\cite{IY}]
For every \(N\in\mathbb N\), there exists a typical three-parameter family
\[
\mathcal X=\{X_\mu\}_{\mu\in\mathbb R^3},
\]
a critical parameter value \(\mu_*\), and a lips polycycle
\(\Gamma\) of \(X_{\mu_*}\) such that
\[
\operatorname{Cycl}(\Gamma,\mathcal X)>N.
\]
\end{pro}

Consequently,
\(
C(3)=+\infty.
\)
This means that the cyclicity of a lips polycycle occurring in a typical
three-parameter family admits no uniform finite upper bound. 

The preceding result concerns general three-parameter families of smooth
planar vector fields.  It does not imply unbounded cyclicity within a
polynomial Li\'enard family of fixed degrees, nor does it imply that a
fixed-degree Hilbert number is infinite.  In the present paper, we instead
consider generalized polynomial Li\'enard systems of fixed degrees.  For
\(n\geq2\), we take
$\deg f=2n+1$, $\deg g=5,$
and construct an explicit lower bound for \(H(2n+1,5)\).

More precisely, let \(\mathcal N(f,g)\) denote the number of limit cycles of
system \eqref{ls}.  The corresponding Hilbert number is defined by
\[
H(n,m)
:=
\sup\left\{
\mathcal N(f,g):
\deg f\leq n,\quad \deg g\leq m
\right\}.
\]
The previously known lower bounds used for comparison are summarized below.
}
\begin{figure}[!htbp]
	\centering
	{\includegraphics[scale=0.3]{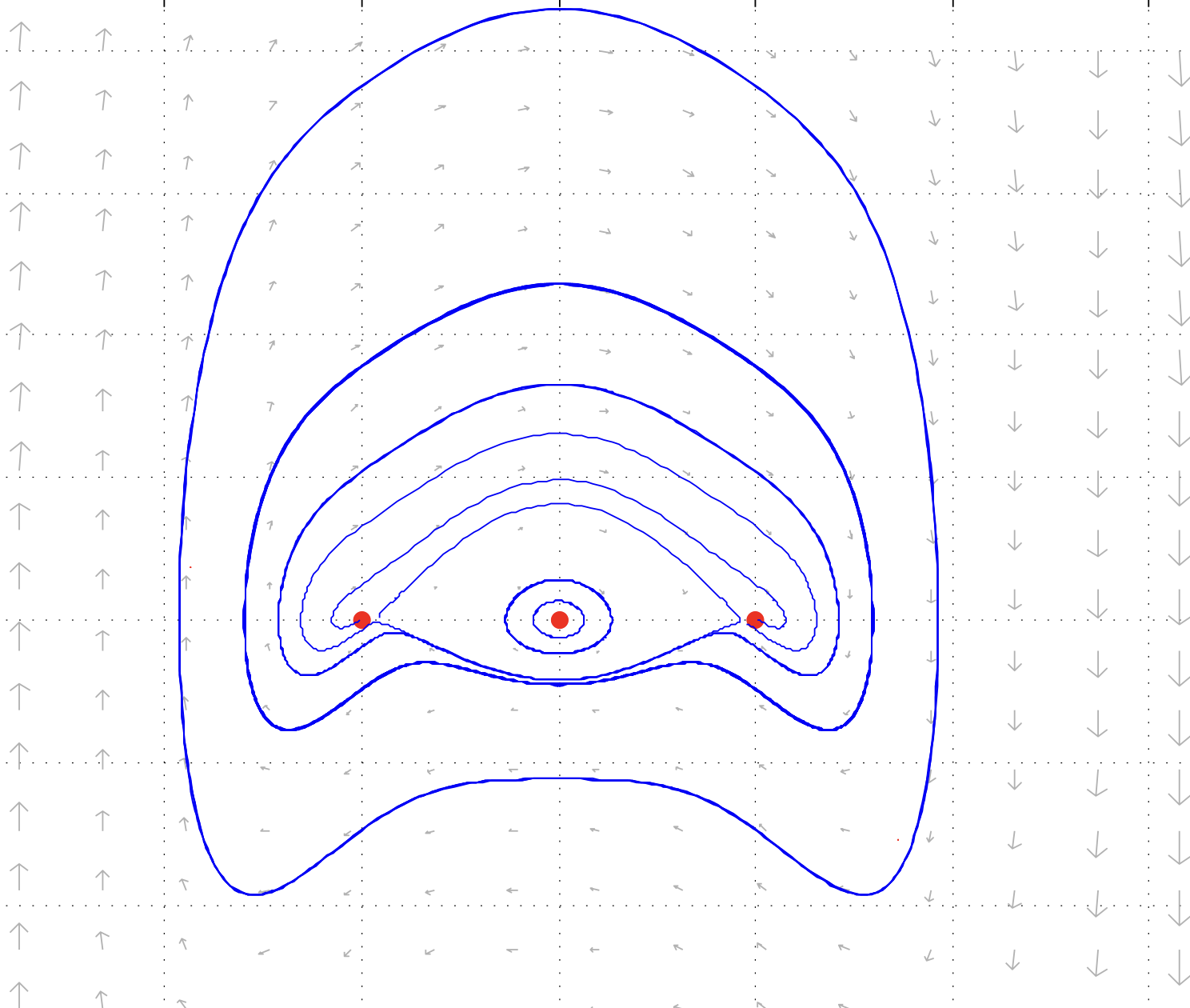}}\hspace{8pt}
	\caption{{Diagram of lips of system \eqref{m5n1} in $\mathcal{G}_5$.}}
	\label{lip-origin}
\end{figure}
\begin{enumerate}[(i)]
					\item Blows and Lloyd \cite{BL} proved
					$
					H(n,1)\geq\left[\frac{n}{2}\right]~~{\text{for}}~~n\ge1.
					$
					\item Han \cite{H1} proved
					$
					H(n,2)\geq\left[\frac{2n+1}{3}\right]~~{\text{for}}~~n\ge2.
					$
					\item Xiong \cite{X2} proved
					$
					H(n,3)\geq n+\left[\frac{n}{4}\right]~~{\text{for}}~~n\ge1.
					$
					\item For $m\geq4$, Han and Romanovski \cite{HR} proved general conclusions
					$$
					H(n,4)\geq H(n,3)\geq 2\left[\frac{n-1}{4}\right]+\left[\frac{n-1}{2}\right]~~{\text{for}}~~n\ge3,$$
					$$H(n,6)\geq H(n,5)\geq 2\left[\frac{n-1}{3}\right]+\left[\frac{n-1}{2}\right]~~{\text{for}}~~n\ge5.$$
					\item For fixed $n$, Xiong \cite{X1} obtained $H(n,5)\geq 2n+1$ for $n=2,3,4$ and $H(n, 5)\geq5 + n$ for $n = 5, 6, 7, 8$, Li and Yang \cite{LY} obtained $H(n,5)\geq n+7-\left[\frac{n+1}{6}\right]$ for $n=10,11,12$ and
					 $H(n,5)\geq n+8-\left[\frac{n+1}{6}\right]$ for $13\leq n\leq 20$.
					\item For $m=5,6$, Xiong and Han \cite{XH} proved that
					$$
					H(n,5)\geq2\left[\frac{n-1}{3}\right]+\left[\frac{n}{2}\right]+2,\quad 
					H(n,6)\geq 2\left[\frac{n-1}{3}\right]+2\left[\frac{n-1}{2}\right]+3.$$
				\end{enumerate}

      %%%%%%%%%%%%%%%%%%%%%%%%%%%%%%%%%%%%          
\subsection{Cyclicity of lips}\label{sec:lips}

{
In the remainder of this subsection, the cyclicity of the lips means the
local cyclicity of the corresponding two-cuspidal polycycle: we count only
limit cycles of the perturbed system that lie in a sufficiently small
neighborhood of the polycycle and converge to it as the perturbation
parameters tend to zero. 
} In addition, we also count the cyclicity near zero and near infinity. 

To obtain an explicit lower bound for the local cyclicity of the lips
within a polynomial Li\'enard family, we first identify the Hamiltonian
two-cuspidal configuration arising from system~\eqref{m5n1}.
Recall that the degenerate parameter surface associated with
$\mathcal G_5$ is characterized by
$c<0,$ $a=\frac{c^2}{4}$.
On the Hamiltonian boundary $b=0$, system~\eqref{m5n1} becomes
\begin{equation}
    \dot{x}=y,\qquad
    \dot{y}
    =-x\left(x^2+\frac{c}{2}\right)^2,
    \qquad c<0.
    \label{general-lips-Ham}
\end{equation}
For every $c<0$, let
$
r=\sqrt{-\frac{c}{2}}.
$
Under the scaling
\begin{equation*}
    \left(x, y, t\right)\rightarrow \left(\sqrt{\frac{-2}{c}}x,\quad \sqrt{\frac{8}{-c^3}}y,\quad \frac{2}{-c}t\right),
\end{equation*}
system~\eqref{general-lips-Ham} is transformed into
\begin{equation}
   \dot x=y,\qquad
   \dot y=-x(x^2-1)^2.
    \label{211}
\end{equation}
Thus all systems satisfying
$
c<0,$ $a=\frac{c^2}{4}$ and $b=0$
are equivalent, up to a linear scaling of the phase variables and a
positive rescaling of time, to the normalized case
$a=1$, $c=-2$, $b=0$.

System~\eqref{211} possesses two cusps at $(\pm1,0)$, whose
separatrix connections form the two-cuspidal polycycle shown in
Fig.~\ref{lip-Ham}. 
{This Hamiltonian system possesses two cusps at $(\pm1,0)$, whose separatrix connections form the lips configuration shown in Fig.~\ref{lip-Ham}.}
\begin{figure}[!htbp]
	\centering
	{\includegraphics[scale=0.3]{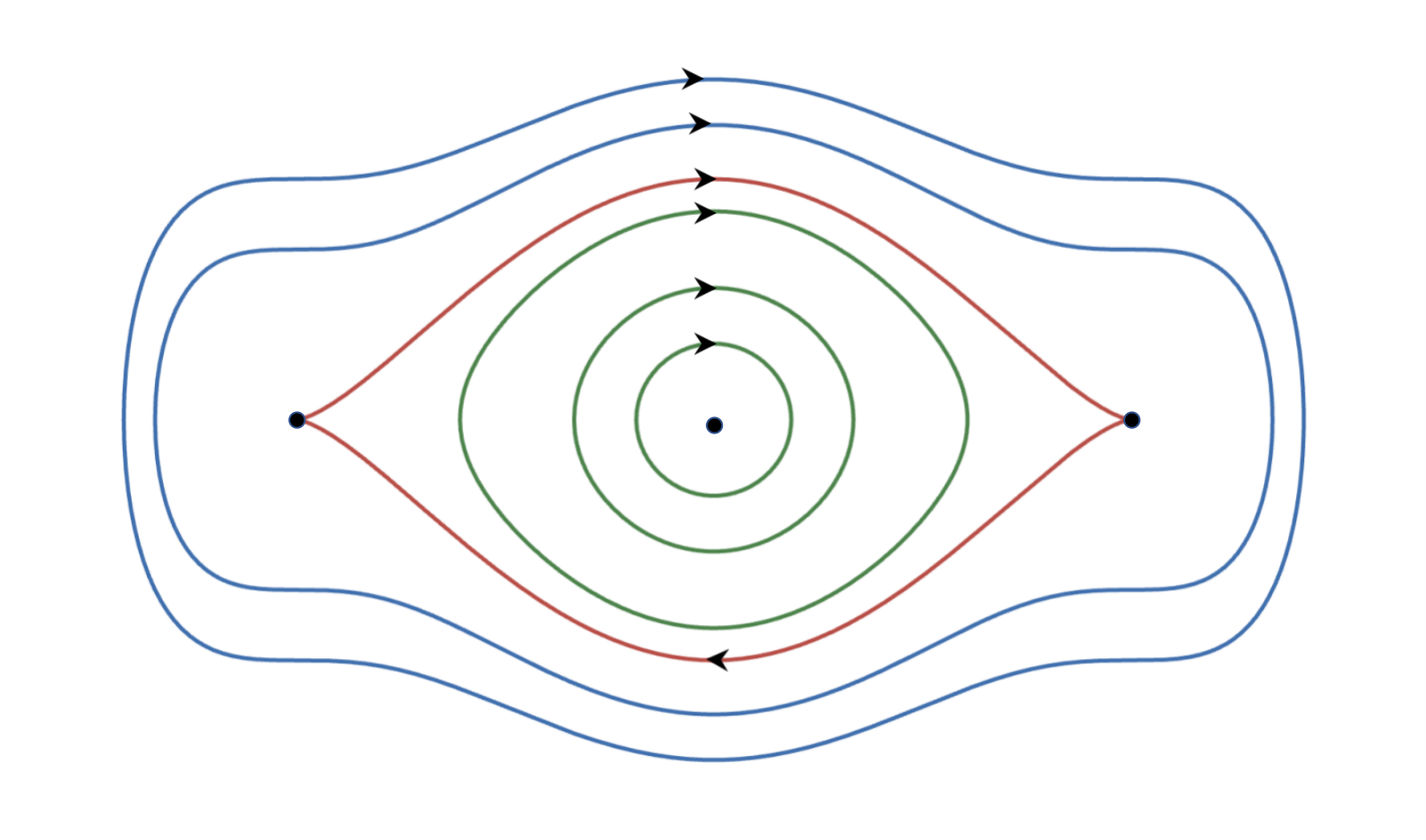}}\hspace{12pt}
	\caption{{Diagram of lips of system \eqref{211}.}}
	\label{lip-Ham}
\end{figure}

{
We next perturb this Hamiltonian system within a fixed-degree polynomial
Li\'{e}nard family.  Our aim is to construct limit cycles that bifurcate
locally from the lips and to combine them with limit cycles arising near
infinity.  This gives an explicit lower bound for the corresponding Hilbert
number.  We establish the following result.
}

{
\begin{thm}
\label{maint}
For every $n\geq2$,
\[
 H(2n+1,5)\geq B(n),
\]
where
\[
 B(n)=2n+\left[\frac n3\right]
 +\left[\frac{n+1}{3}\right]-2.
\]
\end{thm}
}

We now compare the lower bound obtained in Theorem \ref{maint} with some
previously known estimates. Compared with the result of \cite{HR}, we recall that
\[
H(N,5)\geq
2\left[\frac{N-1}{3}\right]
+
\left[\frac{N-1}{2}\right].
\]
Taking \(N=2n+1\), this gives
\[
H(2n+1,5)
\geq
2\left[\frac{2n}{3}\right]+n.
\]
Thus we define
\[
\Delta_1(n)
=
B(n)
-
\left(
2\left[\frac{2n}{3}\right]+n
\right).
\]
{
A direct computation gives \(\Delta_1(n)>0\) for every \(n\geq7\); hence the present bound is weaker for \(n=2,3\), coincides with the estimate of \cite{HR} for \(n=4,5,6\), and is strictly better for every \(n\geq7\).
}

For comparison, at \(N=2n+1\) the estimates in \cite{HR} and \cite{XH}
read, respectively,
\[
H(2n+1,5)\geq2\left[\frac{2n}{3}\right]+n,\qquad
H(2n+1,5)\geq2\left[\frac{2n}{3}\right]+n+2.
\]
The bound \(B(n)\) is weaker than the first for \(n=2,3\), equal for
\(n=4,5,6\), and strictly stronger for \(n\geq7\). It equals the
second for \(n=10,11,12\) and is strictly stronger for \(n\geq13\).
Since \(H(N,5)\geq H(N,3)\), the estimate
\(H(N,3)\geq N+[N/4]\) in \cite{X2} also applies; \(B(n)\) exceeds
\(2n+1+[(2n+1)/4]\) for \(n=21\) and all \(n\geq23\).

The fixed-degree estimates in \cite{X1,LY} remain stronger in their
overlapping cases:
\[
\begin{array}{c|ccccccc}
N&5&7&11&13&15&17&19\\ \hline
B((N-1)/2)&3&6&11&14&16&19&22\\
\text{bounds in \cite{X1,LY}}&10&12&16&19&21&22&24
\end{array}
\]
Theorem~\ref{maint}, however, provides a uniform formula for every
\(n\geq2\).

We now summarize the mechanism behind Theorem~\ref{maint}.  The lower bound is
obtained by combining two local families of limit cycles: one bifurcating from
infinity and the other from the two-cuspidal loop.  We do not use the family
near the origin in the unconditional statement, because the finite-dimensional
compatibility between the origin condition and the lips condition requires
additional information on the transition matrix between the corresponding
\(a\)-parameter coordinates.

\begin{enumerate}[a)]

\item \textbf{Near infinity.}
Near infinity, a weighted compactification transforms the problem into a
regular radial equation on a fixed outer annulus.  The weighted displacement
has selected coefficients whose leading parts form a triangular
nondegenerate system in the effective \(b\)-parameters.

The compactification and the exact \(R\)-equation are given in
Section~\ref{subsec:limit-cycles-infinity}.  The weighted
Picard--Brudnyi expansion is established in
Lemma~\ref{lm:weighted-picard-brudnyi}, and the triangular dependence of the
selected coefficients is proved in
Lemma~\ref{lm:low-coefficients-infty}.
The recursive realization of the selected infinity polynomial is given in
Lemma~\ref{lm:C-infty-nonempty}.
Finally, Lemma~\ref{lm:infinity-cyclicity} proves that the exact weighted
displacement has at least \(n-2\) simple positive zeros.
Consequently, the cyclicity contributed by infinity is at least \(n-2\).
For \(n=2\), this contribution is understood to be zero.

\item \textbf{Near the two-cuspidal loop.}
Near the Ilyashenko--Kotova lips, the leading contribution is the first
Melnikov function.  The Melnikov functions \(M^\pm(h)\) admit 
expansions on the two sides of the two-cuspidal loop.

{
The detailed construction is given in
Section~\ref{subsec:limit-cycles-lips}.  In particular,
Lemma~\ref{lm:Ln-plus-nondegenerate} proves that the selected plus-side
coefficient map
$
a=(a_0,\ldots,a_n)\longmapsto
\bigl(\Lambda_0^+(a),\ldots,\Lambda_n^+(a)\bigr)
$
is a linear isomorphism.
Therefore, the plus-side coefficients can be chosen with alternating signs
and strong dominance, producing at least \(n\) simple zeros of \(M^+(h)\),
and hence plus-side cyclicity at least \(n\).
The induced sign relations for the minus-side coefficients and the resulting
two-sided cyclicity are proved in Lemma~\ref{hn5l}.  More precisely, the total
lips contribution to cyclicity is
$
\eta_{\rm lip}(n)
=
n+\left[\frac n3\right]
+\left[\frac{n+1}{3}\right].
$
}

\end{enumerate}

{
The parameter choices for these two regions are compatible.  Indeed, choose
\(a\in\mathcal C_{\rm lip}^a\) with \(a_n\neq0\).
For this fixed \(a\), the effective \(b\)-parameters can be chosen by the
triangular construction in Lemma
\ref{lm:infinity-cyclicity}.  Hence
$
\mathbb P_2\cap\mathbb P_3\neq\emptyset.
$
}

Combining the two local contributions gives
\[
(n-2)+\eta_{\rm lip}(n)
=
2n+\left[\frac n3\right]
+\left[\frac{n+1}{3}\right]-2
=
B(n).
\]
This proves the lower bound stated in Theorem~\ref{maint}.

 %%%%%%%%%%%%%%%%%%%%%%%%%%%%%%%%%%%%%%%%%%%%%
\subsection{Cyclicity of a near-Hamiltonian system with lips}
\label{subsec:local-conditions}
Consider the Li\'{e}nard system \eqref{ls} as follows
\begin{equation}
\dot{x} = y, \qquad
\dot{y} = -x(x^2 - 1)^2 - \varepsilon_a \sum_{i=0}^{n} a_i x^{2i} y - \varepsilon_b \sum_{i=0}^{n} b_i x^{2i+1} y,
	\label{or}
\end{equation}
{where $0<\varepsilon_b^2\ll\varepsilon_a\ll\varepsilon_b\ll1$, $a_i,b_i\in\mathbb R$, $a_n\neq0$, and $n\in\mathbb Z_+$.} 

System \eqref{or} represents a perturbation of a Hamiltonian vector field with Hamiltonian
\begin{equation}
	H(x,y) = \frac{y^2}{2} + \frac{3x^2 - 3x^4 + x^6}{6}.
	\label{H}
\end{equation}
{
Let $H_0(2n+1,5)$ denote the maximum number of limit cycles realized within
the family \eqref{or}.  Since this family is contained in the class defining
the Hilbert number, one has
\[
 H(2n+1,5)\geq H_0(2n+1,5).
\]
We now state the local and compatibility results used below.
}

We specify the three local parameter conditions. Their derivations and
zero-counting arguments are given in
Sections~\ref{subsec:origin-analysis}, \ref{subsec:limit-cycles-infinity},
and~\ref{subsec:limit-cycles-lips}.

\medskip\noindent\textbf{1. The origin coefficients.}
Write \(P_0(r)-r=\sum_{i\geq0}C_i^0r^{i+1}\) and set
\(\delta=\max_i\{|\varepsilon_a a_i|,|\varepsilon_b b_i|\}\).
The Picard--Brudnyi calculation in
Section~\ref{subsec:origin-analysis} gives
\begin{equation}
C_{2j}^0=\varepsilon_a
\bigl(\gamma_j a_j+\mathcal L_j(a_0,\ldots,a_{j-1})\bigr)
+O(\delta^2),\qquad \gamma_j\neq0,\quad 0\leq j\leq n.
\label{C-even-origin}
\end{equation}
The odd coefficients are \(O(\delta^2)\). Define
\[
\mathcal C_{\rm org}
=\left\{(a,b,\varepsilon_a,\varepsilon_b):
\begin{array}{l}
0<|C_0^0|\ll|C_2^0|\ll\cdots\ll|C_{2n}^0|\ll1,\\
C_0^0C_2^0>0,\quad
C_{2j}^0C_{2j+2}^0<0,\quad 1\leq j<n,\\
\text{the odd terms are dominated by the selected even terms}
\end{array}\right\}.
\]
Lemma~\ref{ln5o} gives the corresponding origin cycles.

\medskip\noindent\textbf{2. The infinity coefficients.}
Use the compactifying variable \(R\) and the return map \(P_\infty\)
from Section~\ref{subsec:limit-cycles-infinity}. With \(M=6n-3\), put
\[
\Delta_\infty(R)=\frac{P_\infty(R)^M-R^M}{M}
\sim\sum_{j\geq0}C_j^\infty R^j.
\label{Cinf-intro}
\]
Its positive zeros are precisely the positive fixed points of \(P_\infty\).
The expansion is asymptotic on the outer annulus, not a Taylor expansion
at \(R=0\) for fixed nonzero parameters
(Lemma~\ref{lm:weighted-picard-brudnyi}).
Set \(\beta_i=\varepsilon_b b_i\) and
\[
Q_\infty^*(z,a,\beta)=\sum_{q=0}^{n-2}C_{2q+1}^*(a,\beta)z^q.
\]
The selected leading contribution to the displacement is
\(\varepsilon_a RQ_\infty^*(R^2,a,\beta)\).
By Lemma~\ref{lm:low-coefficients-infty},
\[
\begin{aligned}
C_1^*&=\Gamma_0^\infty a_n\beta_n^2,\\
C_{2q+1}^*&=\Gamma_q^\infty a_n\beta_n\beta_{n-q}
+\mathcal L_q^\infty(a,\beta_n,\ldots,\beta_{n-q+1}),
\qquad 1\leq q\leq n-2,
\end{aligned}
\]
where every \(\Gamma_q^\infty\neq0\). For fixed \(a_n\beta_n\neq0\),
this dependence is triangular in \(\beta_{n-1},\ldots,\beta_2\);
Lemma~\ref{lm:C-infty-nonempty} realizes the required root pattern.
Define
\[
\mathcal C_{\rm inf}
=\left\{(a,b,\varepsilon_a,\varepsilon_b):
\begin{array}{l}
a_n\neq0,\quad\beta_i=\varepsilon_b b_i,\quad 0\leq i\leq n,\\
\text{there exist }0<z_1^-<z_1^+<\cdots<z_{n-2}^-<z_{n-2}^+\ll1,\\
Q_\infty^*(z_j^-,a,\beta)Q_\infty^*(z_j^+,a,\beta)<0,\\
\operatorname{sgn}\Delta_\infty(\sqrt{z_j^\pm})
=\operatorname{sgn}Q_\infty^*(z_j^\pm,a,\beta),\quad 1\leq j\leq n-2
\end{array}\right\}.
\]

\medskip\noindent\textbf{3. The lips coefficients.}
Put \(h_0=1/6\), \(\tau_-=h-h_0>0\) on \(L_h^-\), and
\(\tau_+=h_0-h>0\) on \(L_h^+\). The expansions used in
Section~\ref{subsec:limit-cycles-lips} are
\[
M^\pm(h)=\sum_{j\geq0}C_j^\pm\tau_\pm^j
+\sum_{j\geq0}D_j^\pm\tau_\pm^{j+5/6}
+\sum_{j\geq0}E_j^\pm\tau_\pm^{j+7/6}.
\]
For related cuspidal-loop expansions, see \cite{X2,XH}.
Order the selected coefficients by \(\Lambda_0^\pm=C_0^\pm\) and
\[
\Lambda_{3j+1}^\pm=D_j^\pm,\qquad
\Lambda_{3j+2}^\pm=C_{j+1}^\pm,\qquad
\Lambda_{3j+3}^\pm=E_j^\pm,\qquad j\geq0.
\]
Lemma~\ref{lm:Ln-plus-nondegenerate} shows that
\(a\mapsto(\Lambda_0^+(a),\ldots,\Lambda_n^+(a))\) is a linear
isomorphism. Hence the cone
\[
\mathcal C_{\rm lip}^a
=\left\{a\in\mathbb R^{n+1}:
\begin{array}{l}
\Lambda_j^+(a)\Lambda_{j+1}^+(a)<0,\quad 0\leq j<n,\\
0<|\Lambda_0^+(a)|\ll|\Lambda_1^+(a)|\ll\cdots\ll|\Lambda_n^+(a)|
\end{array}\right\}
\]
is nonempty. Under this hierarchy, the two-sided sign relations are
\[
\begin{aligned}
\operatorname{sgn}C_j^-&=(-1)^j\operatorname{sgn}C_j^+,\\
\operatorname{sgn}D_j^-&=(-1)^{j+1}\operatorname{sgn}D_j^+,\qquad
\operatorname{sgn}E_j^-=(-1)^{j+1}\operatorname{sgn}E_j^+.
\end{aligned}
\]
Lemma~\ref{hn5l} therefore yields at least
\(\eta_{\rm lip}(n)=n+[n/3]+[(n+1)/3]\) lips cycles. In particular,
\(B(n)=(n-2)+\eta_{\rm lip}(n)\).

\begin{thm}
\label{t-hn5}
Let
$
a=(a_0,a_1,\ldots,a_n),$
$b=(b_0,b_1,\ldots,b_n),
$
and assume \(n\ge2\).  {
Define
\[
 \mathbb P_1:=\mathcal C_{\rm org},\qquad
 \mathbb P_2:=\mathcal C_{\rm inf},
\]
and
\[
 \mathbb P_3
 :=\left\{(a,b,\varepsilon_a,\varepsilon_b):
 a\in\mathcal C_{\rm lip}^a,
 \ 0<\varepsilon_b^2\ll\varepsilon_a\ll\varepsilon_b\ll1
 \right\}.
\]
}
Then the following statements hold.

\begin{enumerate}[(i)]
\item If
$
(a,b,\varepsilon_a,\varepsilon_b)\in\mathbb P_1,
$
then system \eqref{or} has at least \(n-1\) limit cycles near the origin.

\item If
$
(a,b,\varepsilon_a,\varepsilon_b)\in\mathbb P_2,
$
then system \eqref{or} has at least \(n-2\) limit cycles near infinity.  

\item If
$
(a,b,\varepsilon_a,\varepsilon_b)\in\mathbb P_3,
$
then system \eqref{or} has at least \(\eta_{\rm lip}(n)\) limit cycles near
the two-cuspidal loop.
\end{enumerate}
\end{thm}

\begin{thm}
    We have
\(
\mathbb P_1\cap\mathbb P_2\neq\emptyset.
\)
Consequently, 
\[
H_0(2n+1,5)
\ge
(n-1)+(n-2)
=
2n-3.
\]
\label{thm-p12}
\end{thm}

\begin{thm}
     We have
\(
\mathbb P_2\cap\mathbb P_3\neq\emptyset.
\)
Consequently,
{
\[
H_0(2n+1,5)
\ge
(n-2)+\eta_{\rm lip}(n)
=
B(n).
\]
Therefore,
\[
 H(2n+1,5)\ge B(n).
\]
}
\label{thm-hn}
\end{thm}

\begin{pro}
\label{prop:P3-not-origin-max}
Assume that the maximal lips hierarchy defining \(\mathbb P_3\) and the
maximal origin hierarchy defining \(\mathcal C_{\rm org}\) are interpreted
in the usual asymptotic sense of \(\ll\). Then these two maximal hierarchies
cannot be realized simultaneously with arbitrarily strong dominance.
Equivalently, there is no sequence of normalized parameters for which both
sets of successive dominance ratios tend to zero.
\end{pro}

{
\begin{rem}
Proposition~\ref{prop:P3-not-origin-max} excludes only the simultaneous
maximal constructions.  Whether a nonmaximal lips hierarchy can be combined
with additional cycles near the origin is left open.
\end{rem}
}

{
Since $H(2n+1,5)\ge H_0(2n+1,5)$, Theorem~\ref{maint} follows immediately
from Theorem~\ref{thm-hn}.
}

The system also exhibits independent dynamics in the middle energy region. We study the local cyclicity within the period
annulus
\(
\{H=h:\ h\in(1/6,\infty)\}
\)
in Appendix.

\begin{rem}
	The Appendix gives a separate construction of at least (n) simple zeros of
	the first-order Melnikov function on the regular energy interval
	\((1/6,\infty)\). We do not claim that the corresponding interior limit
	cycles can be realized simultaneously with the maximal local constructions
	near the lips and infinity.
\end{rem}

The paper is organized as follows.  Section~2 recalls the compactification and
index tools used in the analysis at infinity.  Section~3 proves the center
criterion and classifies the corresponding global phase portraits.  Section~4
studies the cyclicity near the origin, the two-cuspidal loop, and infinity,
proves the compatibility results, and derives lower bounds for Hilbert
numbers.  The Appendix recalls some background definitions and gives a
separate realization result for limit cycles bifurcating from the regular
period annulus with \(h\in(1/6,\infty)\).

\section{Preliminaries}	
\subsection{Analysis of center at infinity}\label{sec:pre}
This subsection establishes the theoretical foundation for analyzing the center conditions at infinity for generalized polynomial Li\'{e}nard systems.
We begin by applying the  Li\'{e}nard transformation
\(
(x,y)\to\left(x,y+F(x)\right).
\)
The Li\'{e}nard transformation converts system  \eqref{ls} into the equivalent form
\begin{equation}
			\dot{x}=y-F(x),\qquad
\dot{y}=-g(x),
	\label{lst}
\end{equation}
where $F(x): =\int_{0}^{x}f(\xi){\rm d}\xi$. The Li\'{e}nard shear
$
\mathcal S(x,y)=(x,y-F(x))
$
is a polynomial homeomorphism of $\mathbb R^2$ with inverse
$\mathcal S^{-1}(x,Y)=(x,Y+F(x))$. Moreover, it is proper, since the
$x$-coordinate is unchanged. Consequently it maps complements of compact
sets to complements of compact sets and carries periodic orbits of
\eqref{ls} to periodic orbits of \eqref{lst}. Thus the exterior property
used below---the existence of a neighborhood of infinity filled with closed
orbits---is preserved by this transformation. We do not claim that
$\mathcal S$ extends to a homeomorphism of the boundary of the Poincar\'{e}
disc; when $\deg F>1$ such an extension need not preserve the boundary
directions.
To facilitate our analysis, fix an integer $k\geq0$ and define
\begin{equation}
	w(x)=((k+1)G(x))^{\frac{1}{k+1}},
	\label{w}
\end{equation}
where $G(x): =\text{sgn}(x)\int_{0}^{x}|g(\xi)|{\rm d}\xi$.
By $G^{'}(x)=\text{sgn}(x)|g(x)|$ and $G(0)=0$,  $G(x)$ is decreasing for $x<0$ and increasing for $x>0$ and $G(x)\geq0$.
Furthermore, it follows from ${\rm d}w/{\rm d}x=\text{sgn}(x)w^{-k}|g(x)|$ that  $w(x)$ is decreasing for $x<0$ and increasing for $x>0$. Thus, we can conclude that $w(x)\geq0$.
We then define $x_1(w)$ (resp., $x_2(w)$) to be the branch of the inverse of $w(x)$ for $x\leq0$ (resp., $x\geq0$).
Using the transformation $w=w(x)$, we can transform system \eqref{lst} into the following form in the zone $x\leq0$:
\begin{equation}
	\begin{aligned}
		\frac{{\rm d}w}{{\rm d}y}&=\frac{-w^{-k}|g(x)|{\rm d}x}{{\rm d}y}=\frac{F(x)-y}{w^k}\cdot\frac{|g(x)|}{-g(x)}=\frac{F(x_1(w))-y}{w^k}\cdot(-\text{sgn}(g(x_1(w))))
		\\
		&=: \frac{F_1(w)-y}{w^k}\cdot(-\text{sgn}(g_1(w))),
		\label{f1}
	\end{aligned}
\end{equation}
where $w\geq0$, $F_1(w)=F(x_1(w))$ and $g_1(w)=g(x_1(w))$.
Similarly, by the transformation $w=w(x)$, system \eqref{lst} can be expressed in the zone $x\geq0$ as follows:
\begin{equation}
	\begin{aligned}
		\frac{{\rm d}w}{{\rm d}y}&=\frac{w^{-k}|g(x)|{\rm d}x}{{\rm d}y}=\frac{F(x)-y}{w^k}\cdot\frac{|g(x)|}{g(x)}=: \frac{F_2(w)-y}{w^k}\cdot\text{sgn}(g_2(w)),
		\label{f2}
	\end{aligned}
\end{equation}
where $w\geq0$, $F_2(w)=F(x_2(w))$ and $g_2(w)=g(x_2(w))$.

{Let \(B\) be the intersection of the orbit with the \(y\)-axis having
positive \(y\)-coordinate.  Following the orbit forward and backward from
\(B\), let \(A\) and \(C\) denote its next intersections with the negative
part of the \(y\)-axis.  We denote the resulting oriented orbit arc by
\(\widehat{ABC}\).}
Consider an orbit arc $\widehat{ABC}$ that encloses all equilibria of system \eqref{lst} in $\mathbb{R}^2$, as shown in Fig. \ref{abc}(a).
Let $\mathcal{G}$ denote the set of all such orbit arcs.
Since $g$ is a nonzero polynomial, $|g|$ cannot vanish identically on any
nontrivial interval. Hence $G$, and therefore $w$, is strictly decreasing on
$(-\infty,0]$ and strictly increasing on $[0,+\infty)$. Thus the
restrictions of $x\mapsto w(x)$ to the two half-lines are continuous
bijections onto $[0,+\infty)$, with inverse branches $x_1(w)$ and $x_2(w)$.
The map $x\mapsto w(x)$ is therefore a folding map: for every $w>0$ there
are exactly two preimages $x_1(w)<0<x_2(w)$, while $w=0$ corresponds to
$x=0$. No global one-to-one identification of the two phase planes is used.
The intersections of the orbit arc $\widehat{ABC}$ with the regions
$x\leq0$ and $x\geq0$ correspond to integral curves $\gamma_1$ and
$\gamma_2$ on the two folded sheets in the $w$-$y$ plane, satisfying
\eqref{f1} and \eqref{f2}, respectively, as shown in Fig. \ref{abc}(b).
For the sake of simplicity, we will refer to the corresponding points of $A$, $B$, and $C$ in the $x$-$y$ plane as $A'$, $B'$, and $C^{'}$, respectively, in the $w$-$y$ plane.
It can be verified that $\gamma_1$ (resp., $\gamma_2$) corresponds to the integral curve $\widehat{B'C'}$ (resp., $\widehat{B'A'}$) of equation \eqref{f1} (resp., \eqref{f2}) starting from $B'$ and ending at $C'$ (resp., starting from $A'$ and ending at $B'$) on the $y$-axis of the $w$-$y$ plane.
In other words, the orbit arc $\widehat{ABC}$ is closed if and only if $A'$ and $C'$ coincide in the $w$-$y$ plane.
\begin{figure}[!htb]
	\centering
	\subfloat[An orbit arc $\widehat{ABC}$ encircling all the equilibria]
	{\includegraphics[scale=0.6]{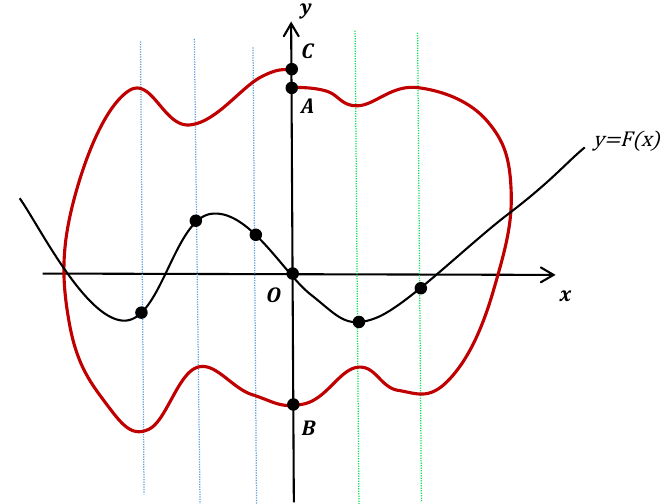}}\hspace{10pt}
	\subfloat[The orbit arcs starting from $B^{'}$ for equations \eqref{f1} and \eqref{f2}]
	{\includegraphics[scale=0.6]{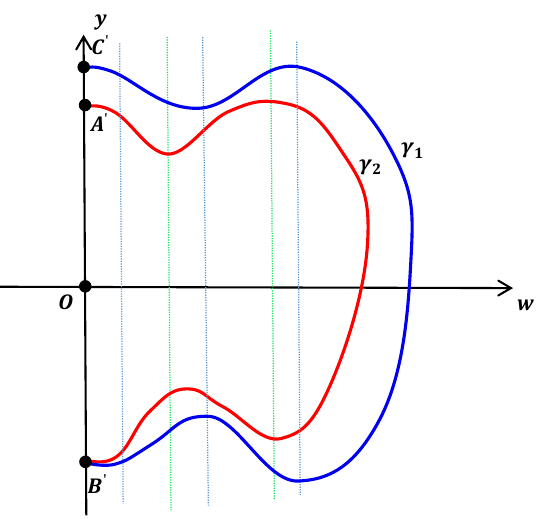}}\hspace{10pt}
	\caption{{\footnotesize The orbit arcs for showing the change of transformation \eqref{w}.}}
	\label{abc}
\end{figure}
We next provide criteria for classifying linear type and nilpotent type center here 
\begin{cor}[Corollary 3 in \cite{CLZ}]
\label{CLZ}
	Assume that $g(x)$ is an odd function.
	System \eqref{ls} has a linear type global center at the origin  if and only if
	following conditions 
	\begin{itemize}
		\item[\bf (\romannumeral1)]$xg(x)>0$ for all $x\ne0$;
		\item[\bf (\romannumeral2)] $k=1$, $s\geq 1$, $a_k>0$; 
		\item[\bf (\romannumeral3)] $m$ is odd,  $m>2n+1$, $a_m>0$, or  $m=2n+1$ and $4(n+1)a_mb_n^{-2}>1$;
	\end{itemize}
	hold and
	$f(x)$ is odd;
	polynomial Li\'{e}nard system \eqref{ls} has a nilpotent type global center at the origin  if and only if
	following conditions 
	\begin{itemize}
		\item[\bf (\romannumeral1)]$xg(x)>0$ for all $x\ne0$;
		\item[\bf (\romannumeral2*)]$k$ is odd,   $2<k<2s+1$, $a_k>0$, or $k=2s+1\geq3$ and $b_s^2-2(k+1)a_k<0$;
		\item[\bf (\romannumeral3)] $m$ is odd,  $m>2n+1$, $a_m>0$; or  $m=2n+1$ and $4(n+1)a_mb_n^{-2}>1$;
	\end{itemize}
	hold and
	$f(x)$ is odd.
\end{cor}

\subsection{Strategy for the cyclicity of the lips}
\label{subsec:lips-strategy}

We first introduce the geometric object and the notation used throughout this
subsection.  Consider the near-Hamiltonian system
\begin{equation}
 \dot{x}=H_y+\varepsilon p(x,y),\qquad
 \dot{y}=-H_x+\varepsilon q(x,y),
 \label{i1}
\end{equation}
where \(H,p,q\) are polynomials and
\[
 p(x,y)=\sum_{i,j\in\mathbb N}a_{ij}x^iy^j,\qquad
 q(x,y)=\sum_{i,j\in\mathbb N}b_{ij}x^iy^j.
\]
For \(\varepsilon=0\), system \eqref{i1} reduces to the Hamiltonian system
\begin{equation}
 \dot{x}=H_y,\qquad
 \dot{y}=-H_x.
 \label{i2}
\end{equation}

Let \(S\) be a cusp of \eqref{i2} on the critical level \(H=h_0\).
After translating \(S\) to the origin, assume that
\begin{equation}
 H(x,y)-h_0
 =
 \frac12y^2+\sum_{i+j\geq3}h_{ij}x^iy^j.
 \label{1.5}
\end{equation}
There is a unique analytic function
\[
 \varphi(x)=\sum_{j\geq2}e_jx^j
\]
such that \(H_y(x,\varphi(x))\equiv0\).  Write
\begin{equation}
 H(x,\varphi(x))-h_0
 =
 h_kx^k+O(x^{k+1}),
 \qquad h_k\neq0.
 \label{1.6}
\end{equation}
If \(k=2m+1\geq3\), then \(S\) is called a cusp of order \(m\).

\begin{defi}
A \emph{two-cuspidal loop} of \eqref{i2} is a heteroclinic polycycle
\[
 L_0=L_1\cup L_2\cup\{S_1,S_2\}
 \subset\{H=h_0\},
\]
where \(S_1\) and \(S_2\) are cusps, \(L_1\) and \(L_2\) are heteroclinic
orbits, and
\[
 \omega(L_1)=\alpha(L_2)=S_2,\qquad
 \omega(L_2)=\alpha(L_1)=S_1.
\]
\end{defi}

For the unperturbed system associated with \eqref{or}, one has
\[
 h_0=\frac16,\qquad
 S_1=(-1,0),\qquad
 S_2=(1,0),
\]
and both cusps are of order one.  The corresponding two-cuspidal loop is
shown in Fig.~\ref{tcl}.

\begin{figure}[!htbp]
 \centering
 \includegraphics[width=3.5in]{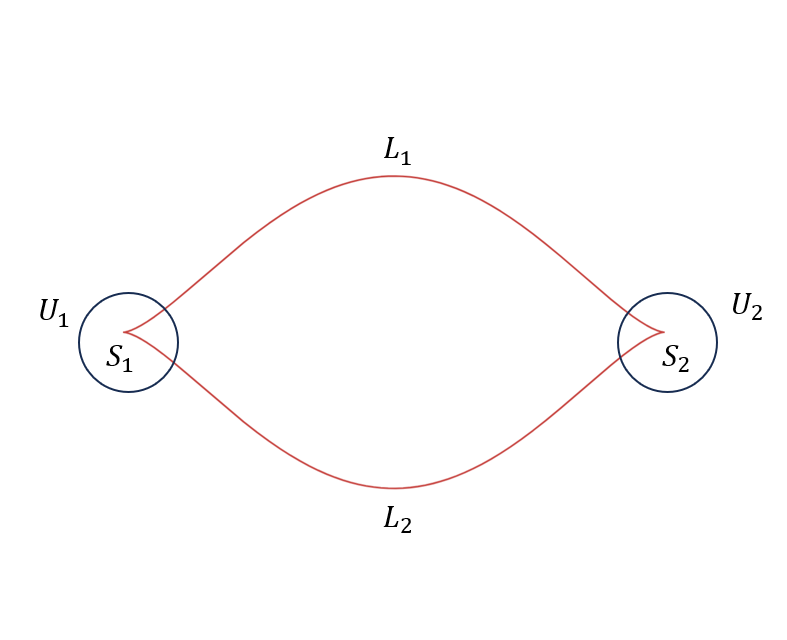}
 \caption{The two-cuspidal loop of system \eqref{i2}.}
 \label{tcl}
\end{figure}

Assume that two families of periodic orbits approach \(L_0\) from its two
sides.  We use the convention
\[
 L_h^-\subset\{H=h\},\qquad h>h_0,
\]
for the outer family, and
\[
 L_h^+\subset\{H=h\},\qquad h<h_0,
\]
for the inner family.  Set
\[
 \tau_+=h-h_0>0,\qquad
 \tau_-=h_0-h>0.
\]
The corresponding first-order Melnikov functions are
\begin{equation}
 M^\pm(h)
 =
 \oint_{L_h^\pm}\bigl(q(x,y)\,{\rm d}x-p(x,y)\,{\rm d}y\bigr).
 \label{1.3}
\end{equation}
This convention for the superscripts \(+\) and \(-\) will be used throughout
the remainder of the paper.

We next recall the local result for a homoclinic loop containing one cusp.
We now present a lemma on the derivative of  $M^{\pm}(h)$ for a homoclinic loop with a cusp. 
\begin{lm}[\cite{HZY}]
	Assume that system \eqref{i1} satisfies \eqref{1.5}. Let $U$ denote a disk of diameter $\varepsilon_0>0$ with center at the origin (cusp). Then
	$$
	\left.\lim _{h \rightarrow 0^{ \pm}} \int_{L_h^{ \pm} \cap U}\left(p_x+q_y-\sigma\right)\right|_{\varepsilon=0} \mathrm{d} t=\left.\int_{L_0 \cap U}\left(p_x+q_y-\sigma\right)\right|_{\varepsilon=0} \mathrm{d} t,
	$$	
	where $\sigma=\left.\left(p_x+q_y\right)\right|_{x=y=0}$, is called the divergence of the cusp.
	\label{lemX}
\end{lm}
\begin{lm}[\cite{HZY}]
	Assume that  \eqref{1.5} and \eqref{1.6} hold, and the equation $H(x, y)=0$ defines a homoclinic loop $L_0$ with a cusp with $h_k<0$ and $k\geq 3$ odd. Then, $M^{ \pm}(h)$ admits the asymptotic expansions
{	\footnotesize
	\begin{align*}
		M^{-}(h)= & c_0+\sum_{l=0}^{\frac{k-3}{2}} c_{l+1}|h|^{\frac{l+1}{k}+\frac{1}{2}}+\sum_{i \geq 0}\left(c_{\frac{k-1}{2}+1+k i} h^{i+1}+\sum_{l=\frac{k-1}{2}}^{k-2} c_{l+2+k i}|h|^{\frac{l+1}{k}+\frac{1}{2}+i}\right. \\
		& \left.+\sum_{l=0}^{\frac{k-3}{2}} c_{k+1+l+k i}|h|^{\frac{l+1}{k}+\frac{1}{2}+i+1}\right)~~{\text{for}}~~0<-h \ll 1, \\
		M^{+}(h)= & \bar{c}_0+\sum_{l=0}^{\frac{k-3}{2}} \bar{c}_{l+1} h^{\frac{l+1}{k}+\frac{1}{2}}+\sum_{i \geq 0}\left(\bar{c}_{\frac{k-1}{2}+1+k i} h^{i+1}+\sum_{l=\frac{k-1}{2}}^{k-2} \bar{c}_{l+2+k i} h^{\frac{l+1}{k}+\frac{1}{2}+i}\right. \\
		& \left.+\sum_{l=0}^{\frac{k-3}{2}} \bar{c}_{k+1+l+k i} h^{\frac{l+1}{k}+\frac{1}{2}+i+1}\right)~~{\text{for}}~~0<h \ll 1,
	\end{align*}
}
	where $c_0, \bar{c}_0, c_{\frac{k-1}{2}+1+k i}, \bar{c}_{\frac{k-1}{2}+1+k i}, i \geq 0$ are constants.
	\label{lemY}
\end{lm}

\begin{cor}[\cite{HZY}]\label{chzy}
	For $k = 3$, the Melnikov function expansions simplify to
\[
	\begin{aligned}
		M^{-}(h)= & c_0+c_1|h|^{\frac{5}{6}}+\sum_{i \geq 0}\left(c_{2+3 i} h^{i+1}+c_{3+3 i}|h|^{\frac{7}{6}+i}+c_{4+3 i}|h|^{\frac{11}{6}+i}\right)~~{\text{for}}~~0<-h \ll 1, \\
		M^{+}(h)=& \bar{c}_0+\bar{c}_1 h^{\frac{5}{6}}+\sum_{i \geq 0}\left(\bar{c}_{2+3 i} h^{i+1}+\bar{c}_{3+3 i} h^{\frac{7}{6}+i}+\bar{c}_{4+3 i} h^{\frac{11}{6}+i}\right)~~{\text{for}}~~0<h \ll 1,
	\end{aligned}
\]
	with coefficients given by
\[
	\begin{aligned}
		& c_0=\bar{c}_0=\oint_{L_0} q \mathrm{d} x-p \mathrm{d}y, \quad
		c_1=2 \sqrt{2} h_{30}^{-\frac{1}{3}} \sigma_0\left(a_{10}+b_{01}\right), \bar{c}_1=2 \sqrt{2} h_{30}^{-\frac{1}{3}} \kappa_0\left(a_{10}+b_{01}\right), \\
		& c_2=\oint_{L_0}\left(p_x+q_y-a_{10}-b_{01}\right) \mathrm{d}t+\mathcal{L}\left(a_{10}+b_{01}\right),\\
		& \bar{c}_2=\oint_{L_0}\left(p_x+q_y-a_{10}-b_{01}\right) \mathrm{d}t+\mathcal{L}\left(a_{10}+b_{01}\right), \\
		& c_{3+3 i}=(-1)^i \sigma_1 \sum_{j=0}^i r_{3(i-j)+1, j} \rho_{3(i-j)+1, j}, \quad \bar{c}_{3+3 i}=\kappa_1 \sum_{j=0}^i r_{3(i-j)+1, j} \rho_{3(i-j)+1, j}, \\
		& c_{4+3 i}=(-1)^{i+1} \sigma_0 \sum_{j=0}^{i+1} r_{3(i+1-j), j} \rho_{3(i+1-j), j}, \quad \bar{c}_{4+3 i}=\kappa_0 \sum_{j=0}^{i+1} r_{3(i+1-j), j} \rho_{3(i+1-j), j},
	\end{aligned}
\]
	and $\sigma_i$, $\kappa_i$, $r_i$ and $\rho_i$ are constants.
	\label{coro1}
\end{cor}

A fundamental question arises: {\it What is the relationship between the Melnikov function expansions for a one-cuspidal loop  (homoclinic loop) and a two-cuspidal loop (heteroclinic loop)?}

 The following two results establish that the asymptotic coefficients $\{c_i\}$ for a one-cuspidal loop (as characterized in Lemma \ref{lemY}) admit a natural generalization to the case of two-cuspidal loop if $S_1$ and $S_2$ have the same order (when $S_1$ and $S_2$ have different orders, see \cite{MY}). 
 
 \begin{lm}[\cite{YH}]\label{lmy}
Consider system \eqref{i1}. For each $i$, $i=1,2$, the order of $S_{i}\left(x_{i}, y_{i}\right)$ is $m_{i}$. Suppose $m_{1}\geq m_{2}\geq 1$. For the Melnikov functions given in \eqref{1.3}, we have
\begin{align*}
	M^{-}(h,\delta)=& c_{0}+\sum_{i=1}^{2}\sum_{l=0}^{l_{1}^{[i]}} B_{l, 0}^{[i]} c_{l+1}^{[i]}\left|h-\frac{1}{6}\right|^{\frac{k_{i}+2+2 l}{2 k_{i}}}
	+c_{m_{1}+1}\left(h-\frac{1}{6}\right)\\
	&+\sum_{i=1}^2\sum_{l=l_1^{[i]}+1}^{k_i-2} B_{l, 0}^{[i]} c_{l+2}^{[i]}\left|h-\frac{1}{6}\right|^{\frac{k_i+2+2 l}{2 k_i}}\\
	&-\sum_{i=1}^2\sum_{l=k_i}^{l_2^{[i]}}\frac{B_{l-k_i, 0}^{[i]} c_{l+1}^{[i]}}{k_i+2+2 l}\left|h-\frac{1}{6}\right|^{\frac{k_i+2+2 l}{2 k_i}}+O\left(\left(h-\frac{1}{6}\right)^2\right)
\end{align*}
for $0<1/6-h\ll 1$ , and
\begin{align*}
	M^+(h)=& c_{0}+\sum_{i=1}^{2}\sum_{l=0}^{l_{1}^{[i]}}\bar{B}_{l, 0}^{[i]} c_{l+1}^{[i]}\left|h-\frac{1}{6}\right|^{\frac{k_{i}+2+2 l}{2 k_{i}}}+c_{m_{1}+1}^{*}\left(h-\frac{1}{6}\right)\\
	&+\sum_{i=1}^{2}\sum_{l=l_{1}^{[i]}+1}^{k_{i}-2}\bar{B}_{l, 0}^{[i]} c_{l+2}^{[i]}\left|h-\frac{1}{6}\right|^{\frac{k_{i}+2+2 l}{2 k_{i}}}\\
	&+\sum_{i=1}^{2}\sum_{l=k_{i}}^{l_{2}^{[i]}}\frac{\bar{B}_{l-k_{i}, 0}^{[i]} c_{l+1}^{[i]}}{k_{i}+2+2 l}\left|h-\frac{1}{6}\right|^{\frac{k_{i}+2+2 l}{2 k_{i}}}
	+O\left(\left(h-\frac{1}{6}\right)^{2}\right)
\end{align*}
for $0<h-1/6\ll 1$, where
\begin{align*}
	& l_1^{[i]}=\frac{k_i-3}{2},\quad l_2^{[i]}=\frac{3(k_i-1)}{2},\quad k_i=2 m_i+1,\\
	& c_{m_1+1}=\bar{c}_{m_1+1}+\sum_{i=1}^2\sum_{l=0}^{l_1^{[i]}} O_1\left(c_{l+1}^{[i]}\right),\\
	& c_{m_1+1}^*=\bar{c}_{m_1+1}+\sum_{i=1}^2\sum_{l=0}^{l_1^{[i]}} O_1^*\left(c_{l+1}^{[i]}\right),
\end{align*}
where $O_{1}(\cdot)$, $O_{1}^{*}(\cdot)$ means $\cdot$ times some constant, and
\[
\bar{c}_{m_1+1}=\oint_{L_0}\left(p_x+q_y\right) {\rm d}t=\sum_{j=1}^2\int_{L_{j}}\left(p_x+q_y\right) {\rm d}t,
\]
if $c_{l+1}^{[i]}=0$, $i=1,2$ for $l=0,1,\ldots, l_1^{[i]}$.
 \end{lm}
 Combining Corollary \ref{chzy} and Lemma \ref{lmy} (see more details in \cite{HZY,YH}), we can give directly the generalized Melnikov function expansion for two-cuspidal  loop is established as follow:
 \begin{cor}
	Assume that there exist two families of periodic orbits near the two-cuspidal loop $L_0$,
	$
	L_h^{ \pm}: H(x, y)=h~~{\text{for}}~~0<|h| \ll 1, 
	$
	then the Melnikov functions are
	\begin{equation*}
		M^{\pm}(h) = \oint_{L_h^{\pm}} \left( q \mathrm{d}x - p \mathrm{d}y \right).
	\end{equation*}
	For $k = 3$, $M^{\pm}(h)$ can be expressed as
	\begin{align*}
		M^{-}(h) &= c_0 + c_1 |h|^{\frac{5}{6}} + \sum_{i \geq 0} \left( c_{2+3i} |h|^{i+1} + c_{3+3i} |h|^{\frac{7}{6}+i} + c_{4+3i} |h|^{\frac{11}{6}+i} \right) 
		 \quad \text{for }~~ 0 < \tfrac{1}{6} - h \ll 1, \\
		M^{+}(h) &= \bar{c}_0 + \bar{c}_1 h^{\frac{5}{6}} + \sum_{i \geq 0} \left( \bar{c}_{2+3i} h^{i+1} + \bar{c}_{3+3i} h^{\frac{7}{6}+i} + \bar{c}_{4+3i} h^{\frac{11}{6}+i} \right)
		\quad \text{for } ~~0 < h - \tfrac{1}{6} \ll 1,
	\end{align*}
	where the coefficients satisfy
	\begin{align*}
		c_0 &= \bar{c}_0 = \oint_{L_0} q \mathrm{d}x - p \mathrm{d}y, \quad
		c_2 = \bar{c}_2 = \oint_{L_0} \left( p_x + q_y - a_{10} - b_{01} \right) \mathrm{d}t + \mathcal{L}(a_{10} + b_{01}), \\
		c_i &= c_i(S_1) + c_i(S_2), \quad \bar{c}_i = \bar{c}_i(S_1) + \bar{c}_i(S_2), \quad\quad i \in \mathbb{Z}_+ \setminus \{1\},
	\end{align*}
	and for each cusp $S_j$, $j=1,2$,
	\begin{align*}
		c_1(S_j)&= 2\sqrt{2}\, h_{30}^{-1/3} \sigma_0 (a_{10} + b_{01}), 
		~~\bar{c}_1(S_j) = 2\sqrt{2}\, h_{30}^{-1/3} \kappa_0 (a_{10} + b_{01}), \\
		c_{3+3i}(S_j)&= (-1)^i \sigma_1 \sum_{\ell=0}^i r_{3(i-\ell)+1,\ell} \rho_{3(i-\ell)+1,\ell},~~
		\bar{c}_{3+3i}(S_j)= \kappa_1 \sum_{\ell=0}^i r_{3(i-\ell)+1,\ell} \rho_{3(i-\ell)+1,\ell}, \\
		c_{4+3i}(S_j)&= (-1)^{i+1} \sigma_0 \sum_{\ell=0}^{i+1} r_{3(i+1-\ell),\ell} \rho_{3(i+1-\ell),\ell},~~
		\bar{c}_{4+3i}(S_j)= \kappa_0 \sum_{\ell=0}^{i+1} r_{3(i+1-\ell),\ell} \rho_{3(i+1-\ell),\ell},
	\end{align*}
	with $\sigma$, $\kappa$, $r$, and $\rho_i$ being constants.
	\label{mel-lips}
\end{cor}

The preceding one-cuspidal expansion is local at the singular point.
For a two-cuspidal loop, the singular contributions from \(S_1\) and \(S_2\)
are added, while the regular parts of the two heteroclinic arcs contribute
to the integer-power coefficients.

System \eqref{or} is integrable for $\varepsilon_a=\varepsilon_b=0$. A fundamental approach to studying the cyclicity of limit cycles—both near the origin and near infinity—is through the analysis of the Poincar\'{e} map. The asymptotic expansion of this map provides the necessary framework for quantifying the birth and persistence of limit cycles under small perturbations. The following lemma establishes the asymptotic expression for the Poincar\'{e} map.
\begin{lm}[\cite{Br}]
	Consider the differential equation
$$
		\frac{{\rm d}r}{{\rm d}\theta} = \sum_{k=0}^{\infty} a_k(\theta) r^{k+1}.
$$
	For sufficiently small initial $r$, the first return map $P(a)$ is given by an absolutely convergent power series
$$
		P(a)(r) = r + \sum_{n=1}^{\infty} c_n(a) r^{n+1},
$$
	where $$c_n(a) = \sum_{i_1 + \cdots + i_k = n} c_{i_1, \ldots, i_k} I_{i_1, \ldots, i_k},$$
	\begin{align*}
		I_{i_1, \ldots, i_k}(\theta)=:\idotsint\limits_{0 \leq s_1 \leq \cdots \leq s_k \leq 2\pi} a_{i_1}(s_1) \cdots a_{i_k}(s_k)  {\rm d}s_1 \cdots {\rm d}s_k,
	\end{align*}
	and the coefficients satisfy
	\begin{equation}
		c_{i_1, \ldots, i_k} = (n - i_1 + 1)(n - i_1 - i_2 + 1) \cdots 1.
	\end{equation}
	\label{by}
\end{lm}
This result guarantees that the displacement function can be systematically analyzed through its series expansion. The coefficients ${c_n(a)}$ which encapsulate the influence of the perturbation terms, become the critical parameters for detecting fixed points of the Poincar\'{e} map, which correspond to limit cycles of the system.
To apply this analytical framework to system \eqref{or} for the distinct regimes as $H \to 0$(near the origin) and $H \to \infty$(near infinity), we introduce the following coordinate transformation
\begin{equation*}
	u = \operatorname{sgn}(x) \sqrt{\frac{(x^2 - 1)^3 + 1}{3}}, \quad \mathrm{d}\tau = \frac{x(x^2 - 1)^2}{|u| \operatorname{sgn}(u)}  \mathrm{d}t.
\end{equation*}
A direct calculation gives
\begin{equation}
\frac{{\rm d}u}{\mathrm{d}\tau}=y, \qquad
\frac{{\rm d}y}{\mathrm{d}\tau}
=-u-\frac{y}{(\omega-1)^2}
\left(
\varepsilon_a |u|\sum_{i=0}^{n}a_i\omega^{i-\frac12}
+\varepsilon_b u\sum_{i=0}^{n}b_i\omega^i
\right),
	\label{uy}
\end{equation}
where
\[
u^2=3^{-1}\left((\omega-1)^3+1\right).
\]
Indeed, $\omega=x^2$ and
$\mathrm d\tau/\mathrm dt=|x|(x^2-1)^2/|u|$; hence the even damping
part produces the factor $|u|$, whereas the odd damping part produces $u$.
Equation \eqref{uy} is the exact equation used in the local calculation
near the origin below. The separate analysis near infinity is given in
Subsection~\ref{subsec:limit-cycles-infinity}.

\section{Center--focus problem}\label{sec:proof}

\subsection[Proofs of the first center--focus theorem]{Proofs of Theorem \ref{cf-thm1} and Corollary \ref{cf-cor1}}
{This subsection gives the proofs of Theorem~\ref{cf-thm1} and Corollary~\ref{cf-cor1}.}

According to \cite{DH}, a center--focus problem arises in statement {\bf(i)}.
The following lemma describes the dynamics near infinity of system \eqref{ls} when statement {\bf(i)} holds.
\begin{lm}
	There is  no orbit of system \eqref{ls} connecting equilibria at infinity in the Poincar\'{e} disc if and only if {statement {\bf(i)} of Theorem~\ref{cf-thm1} holds.}
	System \eqref{ls} exhibits two equilibria $I_{A}^{\pm}$ located on the $y$-axis at infinity. 
	Moreover, the dynamics near infinity are shown in {\rm Fig. \ref{infty}}.
	\label{inf}
\end{lm}

\begin{proof}
	According to the result of \cite{DH}, we can obtain this conclusion and get the dynamics near infinity directly. Moreover, the detailed proof can be found in \cite{CLZ}.
\end{proof}
To facilitate the study of the problem, we analyze system \eqref{lst}. As noted above, the Li\'{e}nard shear is a proper homeomorphism of the finite plane and preserves the exterior closed-orbit property needed here; no homeomorphic extension to the boundary of the Poincar\'{e} disc is used. Assume that $E^{*}=(x^*, y^*)$ is an equilibrium  of system \eqref{lst}.
Using Taylor series to expand $(y-F(x),-g(x))$ around $E^{*}$, we have
\begin{equation*}
	\begin{aligned}
		y-F(x)&=y-y^*+p(x-x^*)+O(|x-x^*|^2),\\
		-g(x)&=q_r(x-x^*)^r+O(|x-x^*|^{r+1}),
	\end{aligned}
\end{equation*}
where $p,q_r\in\mathbb{R}$ and $q_r\neq0$.
Then, we have the following lemmas, which describe the qualitative properties of the equilibria of system \eqref{lst}.
\begin{figure}[!htb]
	\centering
	{\includegraphics[scale=0.6]{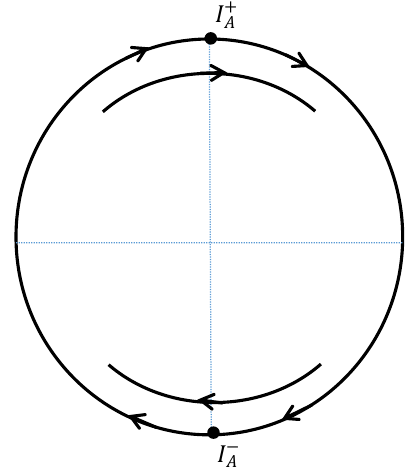}}\hspace{10pt}
	\caption{{Dynamics near infinity in the Poincar\'{e} disc of system \eqref{ls} when statement {\bf(i)} holds. }}
	\label{infty}
\end{figure}
\begin{lm}
	The index of the equilibrium of system \eqref{lst}  has only three possible values: $0$, $1$ or $-1$.
	Moreover, the index of $E^{*}$ is $0$ (resp., $1$; $-1$) if and only if $r$ is even (resp., $r$ is odd and $q_r<0$; $r$ is odd and $q_r>0$).
	\label{index}
\end{lm}
\begin{proof}
	According to \cite[Theorem 7.1 of Chapter 3.2]{ZDHD}, we can derive that $$J_{E^*}(X,Y)=J(y+px,q_rx^r).$$
	Furthermore, based on \cite[Theorem 4.1 of Chapter 3.4]{ZDHD} and the definition of the Cauchy index, we  can obtain that  
	\begin{equation*}
		\begin{aligned}
			J(y+px,q_rx^r)=
			\begin{dcases}
				0, ~~~~~~~~~~~~~~~~~~~~~~~~~~~~~~~~~~\text{when} ~1+r~ \text{is odd},\\
				-N(q_rt^r,1+pt)=
				\begin{dcases}
					1,~ &\text{when} ~1+r~ \text{is even},~ q_r<0,\\
					-1,~&\text{when} ~1+r~ \text{is even}, ~q_r>0.
				\end{dcases}
			\end{dcases}
		\end{aligned}
	\end{equation*}
	Therefore, the index of the equilibrium of system \eqref{lst} has only three possible values: $0$, $1$ or $-1$. Specifically, $r$ is odd and $q_r<0$ if and only if the index of $E^{*}$ is $1$, $r$ is odd and $q_r>0$ if and only if the index of $E^{*}$ is $-1$, and $r$ is even if and only if the index of $E^{*}$ is $0$.
\end{proof}
\begin{rem}
	According to \cite[Theorem 7.1 and Theorem 7.2 of Chapter 2]{ZDHD}, if the index of an equilibrium of system \eqref{lst} is -1, then the equilibrium is a saddle; if the index of an equilibrium of system \eqref{lst} is 1, then the equilibrium is an anti-saddle; if the index of the equilibrium is 0, then the equilibrium is either a saddle-node or a cusp, or it is a critical point with elliptic and hyperbolic sectors.
\end{rem}
\begin{lm}
	The sum of indices of  the finite equilibria of system \eqref{lst} is 0 (resp., 1; -1) when $m$ is even (resp., $m$ is odd and $a_m>0$; $m$ is odd and $a_m<0$).
	\label{index1}
\end{lm}
\begin{proof}
Factor $g$ over $\mathbb R$ in the standard form
\[
g(x)=a_m
\prod_{j=1}^{p}(x-r_j)^{m_j}
\prod_{\ell=1}^{q}
\bigl((x-\alpha_\ell)^2+\beta_\ell^2\bigr)^{\nu_\ell},
\qquad \beta_\ell>0,
\]
where $r_1<\cdots<r_p$ are the distinct real zeros. The irreducible
quadratic factors are strictly positive on $\mathbb R$ and hence do not
change the sign of $g$.

A real zero of even multiplicity produces an equilibrium of index $0$ by
Lemma~\ref{index}. Let
$\rho_1<\rho_2<\cdots<\rho_s$
be the real zeros of odd multiplicity. The sign of $g$ changes at each
$\rho_j$, so the signs of the leading coefficients of $-g$ at these zeros
alternate. By Lemma~\ref{index}, the corresponding indices alternate
between $1$ and $-1$. Their sum is therefore $0$ if $s$ is even. If $s$ is
odd, the last sign is determined by the sign of $g(x)$ for large positive
$x$, namely by $\operatorname{sgn}(a_m)$, and hence
\[
\sum_{g(r)=0}\operatorname{ind}(r,F(r))
=
\begin{cases}
0, & s\ \text{even},\\
1, & s\ \text{odd and }a_m>0,\\
-1,& s\ \text{odd and }a_m<0.
\end{cases}
\]
Finally,
\[
m=\sum_{j=1}^{p}m_j+2\sum_{\ell=1}^{q}\nu_\ell,
\]
so $m$ and the number $s$ of odd-multiplicity real zeros have the same
parity. The stated conclusion follows.
\end{proof}
\begin{defi}
	A closed-orbit set is a set in which every orbit contained within it is a closed orbit.
\end{defi}

For the folded equations, set
\[
\Sigma=\{w\geq0:g(x_1(w))g(x_2(w))=0\}.
\]
Since $g$ is a polynomial, $\Sigma$ is finite. On each component of
$[0,+\infty)\setminus\Sigma$, the functions $x_i(w)$ and $F_i(w)$ are real
analytic. If an outer orbit arc crosses a critical level
$\bar w\in\Sigma$, say $x_i(\bar w)=\xi$, then the crossing point is not the
finite equilibrium $(\xi,F(\xi))$, and therefore
\(
y-F(\xi)\neq0.
\)
Hence, on the adjacent regular intervals,
\[
\frac{{\rm d}y}{{\rm d}w}
=
\pm\frac{w^k}{F_i(w)-y}
\]
has a right-hand side locally Lipschitz in $y$. The usual uniqueness theorem,
together with continuity of the original orbit, therefore gives unique
continuation through $\bar w$. We shall use this fact below whenever a
comparison is continued through a critical level.

We also note that $F_1$ and $F_2$ are continuous semialgebraic functions.
Indeed, $G$ is piecewise polynomial and semialgebraic, the inverse branches
$x_i(w)$ are semialgebraic, and hence so are $F_i=F\circ x_i$. Consequently,
if
\(
D(w):=F_1(w)-F_2(w)
\)
is not identically zero, then $\{D\neq0\}$ is a finite union of open
intervals, on each of which $D$ has a fixed strict sign. In particular,
$\{D\neq0\}$ has a rightmost component.

Consider whether the  set $\mathcal{G}$ is a closed-orbit  set of  system \eqref{lst} when $m$ is odd and $a_m>0$.
To systematically analyze this problem, we classify the system into three essentially different cases based on the number of saddle points:
{\bf(C1)} there is no saddle;
{\bf(C2)} there are $2N-1$ ($N\geq1$) saddles;
{\bf(C3)} there are $2N$ ($N\geq1$) saddles.

Firstly, we consider {\bf (C1)}. 
According to the  Lemmas \ref{index} and \ref{index1}, we can obtain that  system \eqref{lst} has only one equilibrium  with index $1$, while  the indices of all other equilibria are $0$.
Even if the origin is not an equilibrium  of  system \eqref{ls}, we can always move the equilibrium with index $1$ to the origin through coordinate transformation.
Without loss of generality, we assume that the index of  the origin is $1$. Then,  due to the continuity of $g(x)$, we have $xg(x)\geq0$ meaning that $G(x)=\text{sgn}(x)\int_{0}^{x}|g(\xi)|{\rm d}\xi=\int_{0}^{x}g(\xi){\rm d}\xi$.
And it follows from  $xg(x)\geq0$ that $\text{sgn}(g_1(w))=-1$ when $g_1(w)\neq0$ and $\text{sgn}(g_2(w))=1$ when $g_2(w)\neq0$.

\begin{lm}
	In case {\bf (C1)}, the set $\mathcal{G}$ is a closed-orbit set if and only if 
	\(
	F_1(w)\equiv F_2(w)\) for \(0\leq w\leq +\infty.
	\)
	\label{C1}
\end{lm}
\begin{proof}
	The proof employs a comprehensive two-step approach
	
% 	​​{\bf Step I:} 
	{\bf Step I:} Sufficiency proof demonstrating that the identity condition guarantees closed orbits
	
%	{\bf ​​Step II​​:} 
	{\bf Step II:} Necessity proof by contradiction, divided into four technical sub-steps

	{\bf Step I: Sufficiency.}
	In case {\bf(C1)}, one has
$
-\operatorname{sgn}g_1(w)=\operatorname{sgn}g_2(w)=1$
whenever the two values are nonzero. Hence, if
$F_1(w)\equiv F_2(w)$, equations \eqref{f1} and \eqref{f2} are identical
on every regular component. Ordinary uniqueness gives coincidence of the two
folded traces there, and the continuation argument above propagates this
coincidence through the finitely many levels at which $g_i(w)=0$.
Consequently $A'=C'$, and the orbit arc $\widehat{ABC}$ is closed.
	
	{\bf Step II: Necessity.}  We will use a proof by contradiction to demonstrate its necessity.  Due to the complexity of the proof, we have divided it into four parts.
	\begin{description}
		\item[(i)] If $F_1\not\equiv F_2$, the semialgebraic observation above
		shows that $\{F_1-F_2\neq0\}$ has finitely many connected components.
		We use its rightmost component in the comparison below.
		\item[(ii)] On this component $F_1-F_2$ has a fixed strict sign.
		After the symmetry $(y,t)\mapsto(-y,-t)$ if necessary, it is enough
		to treat the case $F_1>F_2$ there. If this component has a finite right
		endpoint, the strict separation produced before that endpoint is
		continued through the subsequent common regular pieces by uniqueness
		and through critical levels by the continuation argument above.
		\item[(iii)] Assume that  there exists an orbit arc $\widehat{ABC}$ in  $\mathcal{G}$.
		Let $\widehat{B^{'}C^{'}}$ cross $B^{*}=(w^*, y_{B^{*}})$ and $C^{*}=(w^*, y_{C^{*}})$,  where $y_{B^{*}}$ is a number sufficiently small and tending to $-\infty$  and $y_{C^{*}}$ is a number  sufficiently large and tending to $+\infty$. 
		And the intersection of $\widehat{B^{'}C^{'}}$ and $F_1(w)$ is $P=(w_P,F_1(w_P))$.
		Clearly, there are three possible types of relationships between $y_{A^*}$ and $y_{C^{*}}$ where $A^{*}$ represents another intersection of $\widehat{B^{'}A^{'}}$ and the line $w=w^*$.
		Define $\gamma^{'}$ as the integral curve of equation \eqref{f2} starting from $A^{*}$. 
		Then, we will prove that $D$ cannot lie on $\gamma_2$ where    
		$D=(w_P,y_D)$ represents  the intersection of $\gamma^{'}$ and the line $w=w_{P}$
		in each of these three cases.
		\item[(iv)] 
		Last, we will prove that  $w_Q < w_P$ where 
		$\widehat{B^{*}P}$ is the integral curve of \eqref{f1} and $Q=(w_Q, F_2(w_Q)$)  is located on the curve $y = F_2(w)$, implying that $A'$ and $C'$ do not coincide.
	\end{description}
	
\noindent{\bf Step II(i).}
Set $D(w)=F_1(w)-F_2(w)$. By the semialgebraic observation above, if
$D\not\equiv0$ then $\{D\neq0\}$ is a finite union of open intervals. Let
$I_*=(\alpha,\beta)$ be its rightmost component, where
$0\leq\alpha<\beta\leq+\infty$, and choose a fixed
$w^*\in I_*$. Then $D$ has one strict sign on $I_*$. If
$\beta<+\infty$, one has $D(w)=0$ for every $w>\beta$.

\noindent{\bf Step II(ii).}
After applying $(y,t)\mapsto(-y,-t)$ if necessary, we may assume
$
F_1(w)>F_2(w),$   $\qquad w\in I_*.$
If $\beta<+\infty$, choose an outer orbit whose rightmost point lies to
the right of $\beta$. The scalar comparison theorem on $I_*$ produces a
strict separation of the two folded traces at $w=\beta$. In case
{\bf(C1)} the sign factors satisfy
$-\operatorname{sgn}g_1=\operatorname{sgn}g_2$ on every regular component,
so for $w>\beta$ the two branch equations are identical. Ordinary
uniqueness, together with the continuation argument above, shows that two
traces which are distinct at $\beta$ cannot meet again; hence their
endpoints at $w=0$ cannot coincide and the orbit is not closed. Therefore
the bounded case is already excluded.

It remains only to consider $\beta=+\infty$. In this case
$
F_1(w)>F_2(w)$, $w>w^*$,
after increasing $w^*$ inside $I_*$ if necessary. The comparison in
Steps~II(iii)--II(iv) is then exactly the one used below. Thus no existence
of a largest zero of $F_1-F_2$ is required.
	\begin{figure}[!htb]
		\centering
		\subfloat[$y_{A^{*}}>y_{C^{*}}$]
		{\includegraphics[scale=0.55]{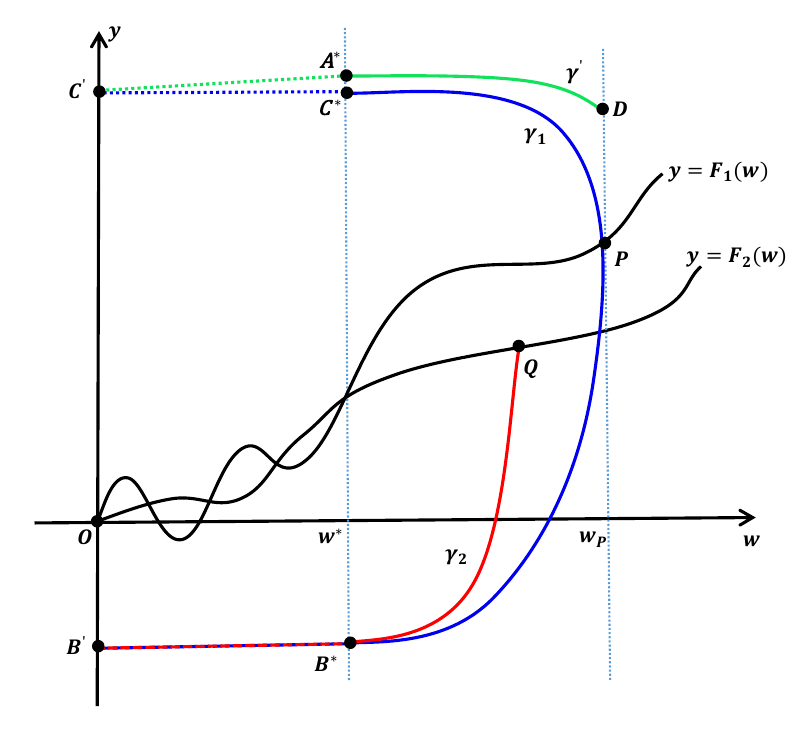}}\hspace{10pt}
		\subfloat[$y_{A^{*}}=y_{C^{*}}$]
		{\includegraphics[scale=0.55]{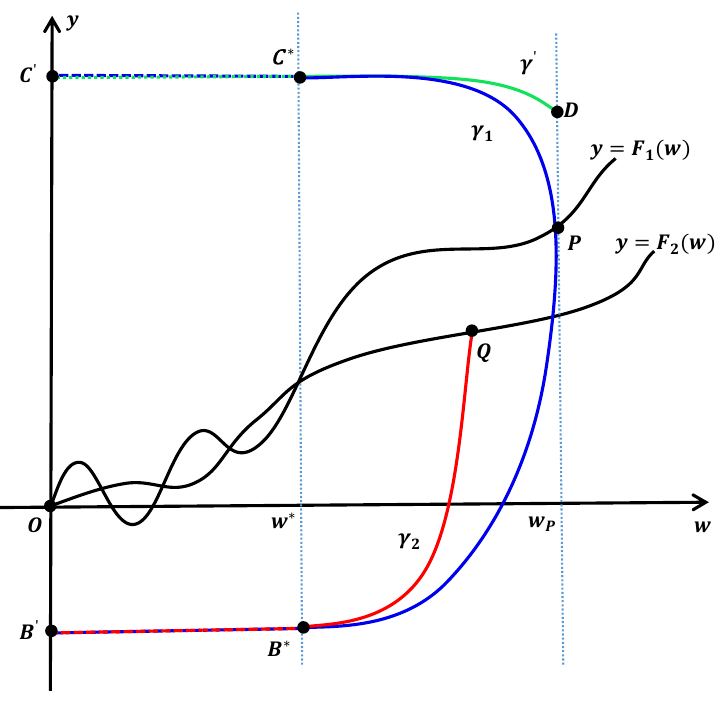}}\hspace{10pt}
		\subfloat[$y_{A^{*}}<y_{C^{*}}$ and $\bar{w}=w_P$]
		{\includegraphics[scale=0.55]{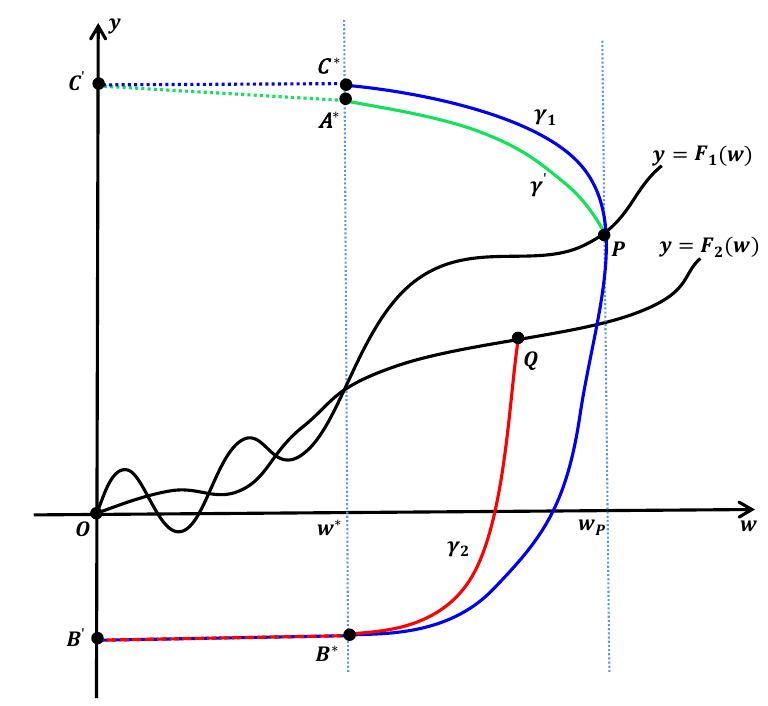}}\hspace{10pt}
		\subfloat[$y_{A^{*}}<y_{C^{*}}$ and $\bar{w}<w_P$]
		{\includegraphics[scale=0.55]{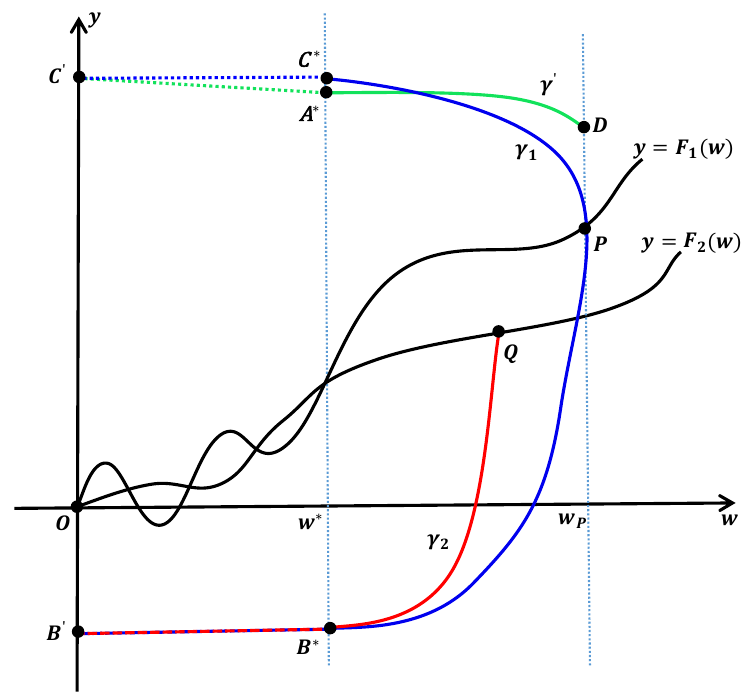}}\hspace{10pt}	
		\caption{{\footnotesize $\widehat{ABC}$ is not a closed orbit in case {\bf(C1)}.}}
		\label{c1}
	\end{figure}
	
	{\bf Step II(iii).} 
	We demonstrate that no orbit arc $\widehat{ABC}\in\mathcal{G}$ can be closed when $F_1(w) \not\equiv F_2(w)$.
	Let $\widehat{B^{'}C^{'}}$ cross $B^{*}=(w^*, y_{B^{*}})$ and $C^{*}=(w^*, y_{C^{*}})$,  where $y_{B^{*}}$ is sufficiently small and tending to $-\infty$  and $y_{C^{*}}$ is a number  sufficiently large and tending to $+\infty$.
	And the intersection of $\widehat{B^{'}C^{'}}$ and $F_1(w)$ is $P=(w_P,F_1(w_P))$.
	Since $w^*$ is fixed, put
\[
M_*=\max_{i=1,2}\max_{0\leq w\leq w^*}|F_i(w)|.
\]
If $Y=|y(w^*)|$ is sufficiently large, then $|y(w)|\geq Y/2$ on
$[0,w^*]$, and on every regular subinterval
\[
\left|\frac{{\rm d}y}{{\rm d}w}\right|
\leq
\frac{w^k}{|y|-M_*}
\leq
\frac{2(w^*)^k}{Y-2M_*}.
\]
After integration and continuation through the finitely many critical
levels, this gives
\[
|y_{C^{'}}-y_{C^{*}}|=O(|y_{C^{*}}|^{-1}),
\qquad
|y_{A^{'}}-y_{A^{*}}|=O(|y_{A^{*}}|^{-1})
\]
as the outer orbit tends to infinity, where $A^{*}$ is the other
intersection of $\widehat{B^{'}A^{'}}$ with $w=w^*$.
When $A'$ and $C'$ coincide,
	it is evident that there are the following three types of relationships:
	
	{\begin{center}
			(a) $y_{A^{*}}>y_{C^{*}}$;~~
			(b) $y_{A^{*}}=y_{C^{*}}$;~~
			(c) $y_{A^{*}}<y_{C^{*}}$.~~
		\end{center}
	}
	
	For each case, we define an auxiliary integral curve $\gamma'$ of equation \eqref{f2} starting from $A^*$ and demonstrate that it must intersect the line $w = w_P$ at a point $D=(w_P, y_D)$.
	When  (a) holds, applying the comparison theorem to equations \eqref{f1} and \eqref{f2}, we find that for $w\in(w^*,w_P)$,  the curve $\gamma^{'}$
	lies strictly to the right of the integral curve $\widehat{C^*P}$.
	Consequently, $\gamma^{'}$
	cannot meet the line $w=w_P$ at point $P$, and must instead intersect it at a distinct point $D$ located above $P$, as illustrated in Fig. \ref{c1}(a).
	
	When (b) holds,  a similar application of the comparison theorem shows that the integral curve $\gamma^{'}$ for $w\in(w^*,w_P)$ again lies to the right of $\widehat{C^{*}P}$.
Therefore, $\gamma^{'}$ must intersect the line $w=w_{P}$ at $D=(w_P,y_D)$ with $y_D>F_1(w_P)$, see  Fig. \ref{c1}(b).

	When (c) holds, this case requires a more detailed analysis. Let $G=(w_G, y_G)$ denote the first point where $\gamma^{'}$ intersects the vertical isocline $y=F_2(w)$. The points on 
	$\widehat{C^*P}\cap\{(w,y)|w^*\leq w<\min\{w_G,w_P\}\}$ and $\widehat{A^*G}\cap\{(w,y)|w^*\leq w<\min\{w_G,w_P\}\}$ can be presented by $(w,y_1(w))$ and  $(w,y_2(w))$ respectively. Let
	$z(w): =y_2(w)-y_1(w)$ be the vertical distance between $y_2(w)$ and $y_1(w)$, with initial condition $z(w^*)=y_{A^*}-y_{C^*}$.
From \eqref{f1} and \eqref{f2}, we express $z(w)$ relative to its value
at the reference point $w^{*}$. Since
\(
z(w^*)=y_{A^*}-y_{C^*},
\)
we have
\[
\begin{aligned}
	z(w)
	={}&z(w)-z(w^*)+(y_{A^*}-y_{C^*})
	\\
	={}&
	\left.
	\left\{y_2(\tau)-y_1(\tau)\right\}
	\right|_{w^*}^{w}
	+(y_{A^*}-y_{C^*})
	\\
	={}&
	\int_{w^*}^{w}
	\left\{
	\frac{\tau^k}{F_2(\tau)-y_2(\tau)}
	-
	\frac{\tau^k}{F_1(\tau)-y_1(\tau)}
	\right\}
	{\rm d}\tau
	+(y_{A^*}-y_{C^*})
	\\
	={}&
	M_1(w)
	+
	\int_{w^*}^{w}
	\left\{y_2(\tau)-y_1(\tau)\right\}
	M_2(\tau)\,{\rm d}\tau
	\\
	={}&
	M_1(w)
	+
	\int_{w^*}^{w}
	z(\tau)M_2(\tau)\,{\rm d}\tau,
\end{aligned}
\]
where the functions $M_1(w)$ and $M_2(w)$ are defined as
\[
M_1(w)
=
\int_{w^*}^{w}
\frac{
	\tau^k\bigl(F_1(\tau)-F_2(\tau)\bigr)
}{
	\bigl(F_2(\tau)-y_2(\tau)\bigr)
	\bigl(F_1(\tau)-y_1(\tau)\bigr)
}
{\rm d}\tau
+
(y_{A^*}-y_{C^*}),
\]
and
\[
M_2(w)
=
\frac{
	w^k
}{
	\bigl(F_2(w)-y_2(w)\bigr)
	\bigl(F_1(w)-y_1(w)\bigr)
}.
\]
From the integrated form, we obtain the relation
	\begin{eqnarray}
		M_2(w)z(w)=M_1(w)M_2(w)+M_2(w)\int_{w^*}^{w}z(s)M_{2}(s){\rm d}s.
		\label{M2}
	\end{eqnarray}
To solve this integral equation, we introduce an auxiliary function 
\begin{eqnarray*}
	M_{3}(w)=\int_{w^*}^{w} z(s) M_{2}(s){\rm d}s.
\end{eqnarray*}
	Differentiating $M_3(w)$ with respect to $w$ yields
	\begin{eqnarray}
		\frac{{\rm d}M_3(w)}{{\rm d}w}-M_2(w)M_3(w)=M_1(w)M_2(w),
		\label{M3}
	\end{eqnarray}
We multiply through by the exponential factor
$\exp\left\{-\int_{w^{*}}^{w} M_{2}(s) {\rm d}s  \right\},$ for  equation \eqref{M3} leading to
	\begin{eqnarray*}
	\frac{{\rm d}}{{\rm d}w}\left\{M_3(w)\exp\left\{-\int_{w^*}^{w} M_{2}(s) {\rm d}s \right\}\right\}
	=M_1(w)M_2(w)\exp\left\{-\int_{w^*}^{w} M_{2}(s) {\rm d}s  \right\},
\end{eqnarray*}
Integrating both sides from $w^{*}$ to $w$, we solve for $M_3(w)$
\begin{eqnarray*}
	M_{3}(w)=\int_{w^{*}}^{w} M_{1}(s) M_{2}(s)
	\exp \left\{\int_{s}^{w} M_{2}(\tau) {\rm d}\tau\right\} {\rm d}s.
\end{eqnarray*}
Substituting this expression back into the original equation for $z(w)$, we obtain the form
	\begin{eqnarray*}
		\begin{aligned}
			z(w) &=M_{1}(w)+\int_{w^*}^{w}M_{1}(s)M_{2}(s)\exp\left\{\int_{s}^{w}M_{2}(\tau){\rm d}\tau\right\}ds
			\\
			&=M_{1}(w^*)\exp\left\{\int_{w^*}^{w} M_{2}(\tau){\rm d}\tau\right\}+\int_{w^*}^{w}M_{1}^{\prime}(s)\exp \left\{\int_{s}^{w}M_{2}(\tau){\rm d}\tau\right\}ds \\
			&=(M_{1}(w^*)+\int_{w^*}^{w}M_{1}^{\prime}(s)\exp \left\{\int_{s}^{w^*}M_{2}(\tau){\rm d}\tau\right\}ds)\exp\left\{\int_{w^*}^{w} M_{2}(\tau){\rm d}\tau\right\}\\
			&: =M(w)\exp\left\{\int_{w^*}^{w} M_{2}(\tau){\rm d}\tau\right\}.
		\end{aligned}
		\label{z}
	\end{eqnarray*}
	\begin{figure}[!htb]
		\centering
		\subfloat[$w_G<w_P$]
		{\includegraphics[scale=0.55]{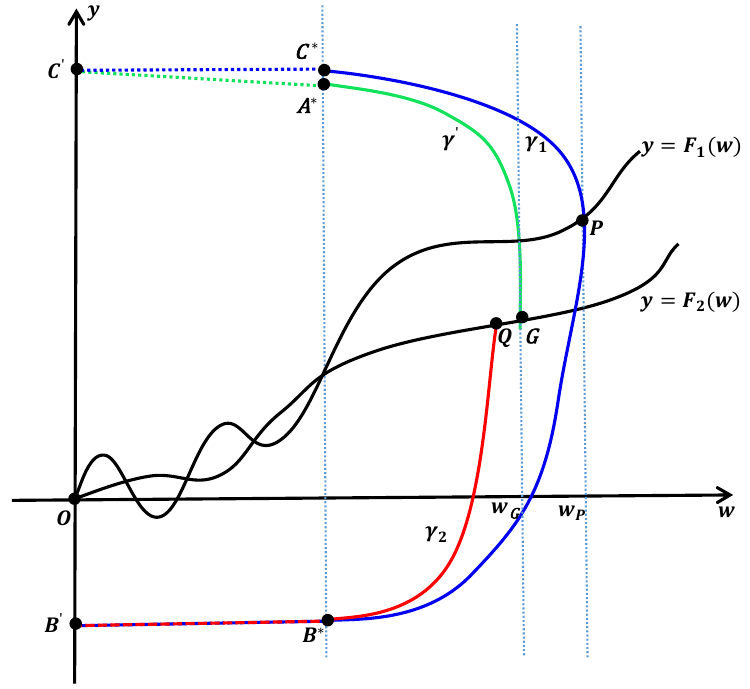}}\hspace{10pt}
		\subfloat[$w_G\geq w_P$]
		{\includegraphics[scale=0.55]{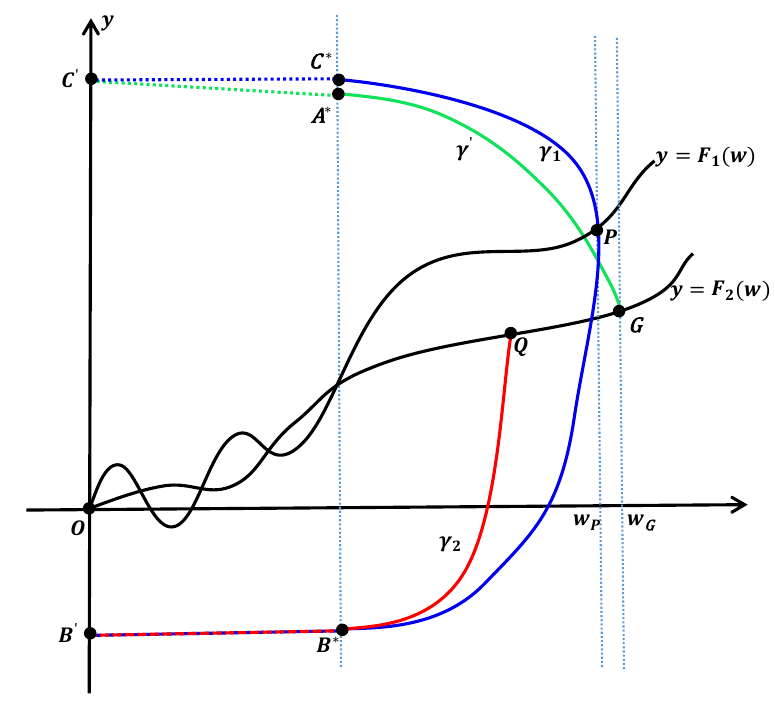}}\hspace{10pt}
		\caption{{\footnotesize $z(w)<0$ when $y_{A^{*}}<y_{C^{*}}$ in case {\bf(C1)}.}}
		\label{c1d}
	\end{figure}
	It can be easily verified that $M(w)$ is  monotonically increasing since ${\rm d}M(w)/{\rm d}w>0$.
	Assume that $z(w)<0$ for all $w\in\{w|w^*\leq w<\min\{w_G,w_P\}\}$, as shown in Fig. \ref{c1d}. 	 
	Then, we can obtain that $M(w)<0$ for all $w\in\{w|w^*\leq w <\min\{w_G,w_P\}\}$. 
	However, when $w_G<w_P$, as shown in Fig. \ref{c1d}(a), it becomes evident that the relationship $y_2(w)=F_2(w)$ holds  at $w=w_G$, which implies that $M(w)$ tends to $+\infty$. This contradicts the  fact that $M(w)<0$.
	Similarly, when $w_G\geq w_P$, as shown in Fig. \ref{c1d}(b), it becomes apparent that  the equation $y_1(w)=F_1(w)$ is satisfied at $w=w_P$, resulting in $M(w)$  tending to $+\infty$. This contradicts the fact that $M(w)<0$.
	Furthermore, it can be concluded that  there exists a $\bar{w}\in\{w|w^*\leq w \leq \min\{w_G,w_P\}\}$ satisfying  $z(\bar{w})=0$.
	
	If $\bar{w}=w_P$,  $\gamma^{'}$ must intersect the line $w=w_{P}$ at the point $P$, see Fig. \ref{c1}(c).
	If $w^*<\bar{w}<w_P$, we can deduce that the integral curve $\gamma^{'}\cap\{(w,y)|w\in(\bar{w},w_P)\}$ of \eqref{f2} lies on the right-hand side of the integral curve $\widehat{C^{*}P}$ of \eqref{f1}  by  applying the comparison theorem to equations \eqref{f1} and \eqref{f2} again. 
	Then, $\gamma^{'}$ must intersect the line $w=w_{P}$ at the point $D=(w_P,y_D)$, see  Fig. \ref{c1}(d).

	{\bf Step II(iv).} 
	Let $\widehat{B'A'}$ cross $\bar B^{*}=(w^*, y_{B^*})$. It is evident that there are the following three types of relationships:
	(a) $y_{\bar B^{*}}>y_{B^{*}}$;~~
(b) $y_{\bar B^{*}}=y_{B^{*}}$;~~
(c) $y_{\bar B^{*}}<y_{B^{*}}$.

	By a similar analysis of  {\bf Step II(iii)}, 
	we conclude that $w_Q < w_P$,  where $Q=(w_Q, F_2(w_Q)$)  represents the intersection of $\widehat{B^{'}A^{'}}$ and  the curve $y = F_2(w)$.
	Furthermore, $Q$ represents the right-most point on $\gamma_2$ since $y = F_2(w)$ is the vertical isocline of equation \eqref{f2}. Hence, $D$ cannot lie on $\gamma_2$ due to the inequality $w_Q < w_P$, which implies that $\gamma_2$ and $\gamma'$ are distinct.
	Thus, based on the uniqueness of  solutions, $A^*$ is not an intersection point of $\gamma_2$ and $w = w^*$. In other words, $A'$ and $C'$ do not coincide, which is a contradiction. Thus, the orbit arc $\widehat{ABC}$ is not closed.
\end{proof}

Secondly, we consider case {\bf (C2)}. According to Lemmas \ref{index}-\ref{index1}, we can conclude that system \eqref{lst} must have $2N-1$ equilibria  with index $-1$ and $2N$ equilibria with index 1, while the index of all other equilibria is $0$. 
We can label these $4N-1$ equilibria  with index $1$ or $-1$ as  $\bar{E}_1 =(\bar{x}_1,\bar{y}_1), \bar{E}_2 =(\bar{x}_2,\bar{y}_2), \cdots, \bar{E}_{4N-1} =(\bar{x}_{4N-1},\bar{y}_{4N-1})$, where $\bar{x}_1<\bar{x}_2<\cdots<\bar{x}_{4N-1}$
Therefore, without loss of generality,  we can move $\bar{E}_{2N}$  to the origin. 
Assume that $E^{*}=(x^*, y^*)$ is an equilibrium  of system \eqref{lst}.
By Lemma \ref{index},  we  can obtain that $q$ is odd and $a_q>0$ when the index of $E^{*}$ is $-1$, $q$ is odd and $a_q<0$ when the index of $E^{*}$ is $1$ and $q$ is even when the index of $E^{*}$ is $0$.  
Furthermore, by Lemma \ref{index1},  we can conclude that for any two equilibria with index $1$, there exists an equilibrium with index $-1$ whose abscissa lies between their abscissas. And for any two equilibria with index $-1$, there exists an equilibrium with index $1$ whose abscissa lies between their abscissas.
In this case, there will be $N$ equilibria with  index  $1$ and $N-1$ saddles on the left (resp., right) side of the $y$-axis. We can label these $2N-1$ equilibria on the left (resp., right) side of the $y$-axis as  $E_{l(2N-1)}=(x_{l(2N-1)}, y_{l(2N-1)}), E_{l(2N-2)}=(x_{l(2N-2)}, y_{l(2N-2)}), \cdots, E_{l1}=(x_{l1}, y_{l1})$, where $x_{l(2N-1)}<x_{l(2N-2)}<\cdots<x_{l1}$ (resp.,  $E_{r(2N-1)}=(x_{r(2N-1)}, y_{r(2N-1)})$, $E_{r(2N-2)}=(x_{r(2N-2)}, y_{r(2N-2)})$, $\cdots$, $E_{r1}=(x_{r1}, y_{r1})$, where $x_{r(2N-1)}>x_{r(2N-2)}>\cdots>x_{r1}$).
We  have that $g(x)\geq0$ when $$x\in\bigcup_{j=2}^{N}(x_{l(2j-1)}, x_{l(2j-2)})\cup(x_{l1}, 0)\cup\bigcup_{j=2}^{N}(x_{r(2j-3)}, x_{r(2j-2)})\cup(x_{r(2N-1)}, +\infty)$$ and  $g(x)\leq0$ when $$x\in(-\infty, x_{l(2N-1)})\cup\bigcup_{j=2}^{N}(x_{l(2j-2)}, x_{l(2j-3)})\cup(0, x_{r1})\cup\bigcup_{j=2}^{N}(x_{r(2j-2)}, x_{r(2j-1)}).$$

\begin{lm}
	In case {\bf (C2)}, the set $\mathcal{G}$ is a closed-orbit set if and only if 
\[
	F_1(w)\equiv F_2(w)~~{\text{for}}~~0\leq w\leq +\infty.
    \]
	\label{C2}
\end{lm}
\begin{proof}
	The proof is divided into  two steps as follows.
	\begin{description}
		\item[(I)] We first prove the sufficiency. And the proof will be divided  into the following two parts.
		\begin{description}
			\item[(i)] We first prove that  $w$ satisfies $x_1(w)=x_{lj}$ if and only if $x_2(w) = x_{rj}$ for any $j=1, 2, \cdots, 2N-1$.
			\item[(ii)] We can apply a similar proof of {\bf Step I} in Lemma \ref{C1} to prove that the orbit arc $\widehat{ABC}$ is closed.
		\end{description}
		\item[(II)] Next, we will use a proof by contradiction to demonstrate its necessity.  Due to the complexity of the proof, we have divided it into four parts.
		\begin{description}
			\item[(i)] If $F_1\not\equiv F_2$, we use the rightmost connected
			component of $\{F_1-F_2\neq0\}$ supplied by
			the semialgebraic observation above.
			\item[(ii)] On this component the difference has a fixed strict
			sign. After the symmetry $(y,t)\mapsto(-y,-t)$ if necessary, we
			take $F_1>F_2$. The bounded rightmost component is excluded by the
			same separation-and-uniqueness argument as in Lemma~\ref{C1};
			hence the detailed comparison only needs the unbounded case.
			\item[(iii)] Assume that  there exists an orbit arc $\widehat{ABC}$ in  $\mathcal{G}$.
			Let $\widehat{B^{'}C^{'}}$ cross $B^{**}=(w^{**}, y_{B^{**}})$ and $C^{**}=(w^{**}, y_{C^{**}})$  where $w^{**}=\max\{w^*,\max\Sigma,w(x_{l(2N-1)}),w(x_{r(2N-1)})\}$, $y_{B^{**}}$ is a number sufficiently small and tending to $-\infty$  and $y_{C^{**}}$ is a number  sufficiently large and tending to $+\infty$.
			And the intersection of $\widehat{B^{'}C^{'}}$ and $F_1(w)$ is $P^{*}=(w_{P^{*}},F_1(w_{P^{*}}))$.
			Clearly, there are three possible types of
			relationships between $y_{A^{**}}$ and $y_{C^{**}}$ where $A^{**}$ represents another intersection of $\widehat{B^{'}A^{'}}$ and the line $w=w^{**}$.
			Define $\gamma^{''}$ as the integral curve of equation \eqref{f2} starting from $A^{**}$. 
			Then, we will prove that $D^*$ cannot lie on $\gamma_2$ where    
			$D^{*}=(w_{P^{*}},y_{D^{*}})$ represents  the intersection of $\gamma^{''}$ and the line $w=w_{P^{*}}$ in each of these three cases.
			\item[(iv)]  
			Last, we will prove that  $w_{Q^*} < w_{P^*}$ where 
			$\widehat{B^{**}P^{*}}$ is the integral curve of \eqref{f1} and $Q^{*}=(w_{Q^{*}}, F_2(w_{Q^{*}})$)  is located on the curve $y = F_2(w)$, implying that $A'$ and $C'$ do not coincide.
		\end{description}
	\end{description}
	The comprehensive four-step contradiction argument establishes that if $F_1(w)\not\equiv F_2(w)$ no orbit arc  
	$\widehat{ABC}\in \mathcal{G}$ can be closed. Combined with the sufficiency proof, this completes the demonstration that $F_1(w)\equiv F_2(w)$ is both necessary and sufficient for $\mathcal{G}$ to be a closed-orbit set in case {\bf(C2)}.

	{\bf Step I: Sufficiency.}  
	
	{\bf Step I(i).} 
	A straight calculation shows that
	\begin{equation}
		\frac{{\rm d}F_{i}(w)}{{\rm d}w}=\frac{f(x_i(w))w^k}{\text{sgn}(x_i(w))|g(x_i(w))|}.
		\label{dF/dw}
	\end{equation}
	It follows from  $F_1(w)\equiv F_2(w)$ and \eqref{dF/dw} that $F_1^{'}(w)\equiv F_2^{'}(w)$, which means 
	\begin{equation}
		\frac{f(x_1(w))}{-|g(x_1(w))|}=\frac{f(x_2(w))}{|g(x_2(w))|}\qquad{\text{for}}\quad w\neq0.
		\label{eq}
	\end{equation}
	We claim that  $w$ satisfies $x_1(w)=x_{lj}$ if and only if $x_2(w) = x_{rj}$ for any $j=1, 2, \cdots, 2N-1$.
	For any $\bar{w}> 0$, we can expand both sides of equation \eqref{eq} around $x_1(\bar{w})$ and $x_2(\bar{w})$ respectively, resulting in:
	\begin{equation}
		\frac{f(x_1(w))}{-|g(x_1(w))|}=\frac{a_{k_1}(x_1(w)-\bar{x}_1)^{k_1}+o((x_1(w)-\bar{x}_1)^{k_1})}{-|a_{r_1}(x_1(w)-\bar{x}_1)^{r_1}+o((x_1(w)-\bar{x}_1)^{r_1})|},
		\label{eq1}
	\end{equation}
	\begin{equation}
		\frac{f(x_2(w))}{|g(x_2(w))|}=\frac{a_{k_2}(x_2(w)-\bar{x}_2)^{k_2}+o((x_2(w)-\bar{x}_2)^{k_2})}{|a_{r_2}(x_2(w)-\bar{x}_2)^{r_2}+o((x_2(w)-\bar{x}_2)^{r_2})|},
		\label{eq2}
	\end{equation}
	where $\bar{x}_1=x_1(\bar{w})$, $\bar{x}_2=x_2(\bar{w})$  and $k_1,k_2,r_1,r_2\in\mathbb{N}$.
	It is evident that there exists $\delta_1 > 0$ such that 
	\[
	\text{sgn}(a_{k_1}(x_1(w)-\bar{x}_1)^{k_1}+o((x_1(w)-\bar{x}_1)^{k_1}))= \text{sgn}(a_{k_1}(x_1(w)-\bar{x}_1)^{k_1})
	\]
	holds when $x_1 \in B_{\delta_1}(\bar{x}_1)$. Similarly, there exists $\delta_2 > 0$ such that
	\[
	\text{sgn}(a_{k_2}(x_2(w)-\bar{x}_2)^{k_2}+o((x_2(w)-\bar{x}_2)^{k_2}))= \text{sgn}(a_{k_2}(x_2(w)-\bar{x}_2)^{k_2})
	\]
	holds when $x_2 \in B_{\delta_2}(\bar{x}_2)$.
	Since $x_1$ and $x_2$ are continuous functions and $x_1$ is monotonically decreasing while $x_2$ is monotonically increasing with respect to $w$,  there exists $\epsilon > 0$ such that $x_1 \in B_{\delta_1}(\bar{x}_1)$ and $x_2 \in B_{\delta_2}(\bar{x}_2)$ when $w \in B_{\epsilon}(\bar{w})$.
	Then, combined with \eqref{eq}, we can conclude that
	\[
	\text{sgn}\left(-a_{k_1}(x_1(w)-\bar{x}_1)^{k_1}\right) = \text{sgn}\left(a_{k_2}(x_2(w)-\bar{x}_2)^{k_2}\right)
	\]
	holds when $w\in B_{\epsilon}(\bar{w})$.
	Therefore, we can infer that $k_1$ and $k_2$ have the same parity.
	Moreover, since equations (\ref{eq}-\ref{eq2}) hold, it follows that $k_1 - r_1 = k_2 - r_2$, which implies that $r_1$ and $r_2$ have the same parity.
	If $x_1(\bar{w})=x_{lj}$ where $1\leq j\leq 2N-1$, it implies that $r_1$ is odd. Consequently, we can deduce that $r_2$ is also odd.
	Thus, $x_2(\bar{w})$ needs to satisfy the condition that $g(x_2(\bar{w}))=0$ and $g(x_2(w))=a_{r_2}(x_2(w)-\bar{x}_2)^{r_2}+o((x_2(w)-\bar{x}_2)^{r_2})$ where $r_2$ is odd.
	In this case, the  value of $x_2(\bar{w})$ that satisfies this condition is $x_{ri}$ where $1\leq i\leq2N-1$ and $i$ can be equal to $j$ or different from $j$.
	Similarly, when $x_2(\bar{w})=x_{rj}$, the  value of $x_1(\bar{w})$ that satisfies the condition is $x_{li}$ where $1\leq i\leq2N-1$ and $i$ can be equal to $j$ or different from $j$.
	We will now prove that $i=j$.
	Assume that $x_1(\bar{w})=x_{l1}$ and $x_2(\bar{w})=x_{ri}$ where $i>1$. 
	Since $w(x)$ is monotonically decreasing in $(-\infty,0)$ and monotonically increasing in $(0,+\infty)$,  we have  $w(x_{l(2N-1)})>w(x_{l(2N-2)})>\cdots>w(x_{l1})$ and $w(x_{r1})<w(x_{r2})<\cdots<w(x_{r(2N-1)})$.
	Hence, we obtain $w(x_{r1})<w(x_{ri})=w(x_{l1})$.
	Therefore, when $x_2(\bar{w})=x_{r1}$, there is no value of $x_1(\bar{w})$ that satisfies \eqref{eq}, which is a contradiction. This implies that $i=1$.
	Similarly, we can prove that $i=j$ for any $1< j\leq 2N-1$.
	
	{\bf Step I(ii).} 
	Therefore,
$\operatorname{sgn}(g_1(w))=-\operatorname{sgn}(g_2(w))$ whenever both
values are nonzero, and equations \eqref{f1} and \eqref{f2} are identical
on every regular component. Ordinary uniqueness applies there, while
the continuation argument above gives the continuation through the finitely
many critical levels. Hence the orbit arc $\widehat{ABC}$ is closed.
	
	{\bf Step II: Necessity.} 
	
\noindent{\bf Step II(i).}
Let $I_*=(\alpha,\beta)$ be the rightmost component of
$\{F_1-F_2\neq0\}$ supplied by the semialgebraic observation above, and choose
$w^*\in I_*$. After the transformation $(y,t)\mapsto(-y,-t)$ if necessary,
assume $F_1>F_2$ on $I_*$.

If $\beta<+\infty$, then $F_1=F_2$ on $(\beta,+\infty)$. Applying the
local root-pairing argument of Step~I(i) successively to the critical levels
in this tail gives
\[
-\operatorname{sgn}g_1(w)=\operatorname{sgn}g_2(w)
\]
on every regular tail component. Hence the two folded equations coincide
for $w>\beta$. The strict separation produced by the comparison theorem on
$I_*$ cannot disappear under the common tail flow, and
the continuation argument above gives the continuation through the critical
levels. Thus the orbit cannot close.

\noindent{\bf Step II(ii).}
We may therefore assume $\beta=+\infty$. Increasing $w^*$ inside $I_*$
if necessary, we have
\[
F_1(w)>F_2(w),\qquad w>w^*,
\]
which is precisely the hypothesis needed in the remaining comparison.
	
	{\bf Step II(iii).} 
	We now assume, for the sake of contradiction, that there exists a closed orbit arc $\widehat{ABC}$ in  $\mathcal{G}$ even though $F_1(w)\not \equiv F_2(w)$.
	 Let
	 \[
	 w^{**}=\max\{w^*,\max\Sigma,w(x_{l(2N-1)}),w(x_{r(2N-1)})\}.
	 \]
	 Consider the intersection points of the integral curve  $\widehat{B^{'}C^{'}}$ with the vertical line $w=w^{**}$, denoted as  $B^{**}=(w^{**}, y_{B^{**}})$ and $C^{**}=(w^{**}, y_{C^{**}})$, where  $y_{B^{**}}$ is a sufficiently large negative number and $y_{C^{**}}$ is a sufficiently large positive number.
Let  $P^{*}=(w_{P^{*}},F_1(w_{P^{*}}))$ be the intersection point of $\widehat{B^{'}C^{'}}$ with the curve $y=F_1(w)$.
	Since $w^{**}$ is fixed, the same estimate used in case {\bf(C1)}, with
$W=w^{**}$, gives
	\[
	|y_{C'}-y_{C^{**}}|=O(|y_{C^{**}}|^{-1}),
	\qquad
	|y_{A'}-y_{A^{**}}|=O(|y_{A^{**}}|^{-1}),
	\]
	where $A^{**}$ represents another intersection of $\widehat{B^{'}A^{'}}$ and the line $w=w^{**}$. Thus the two errors can be made smaller than any fixed strict separation used in the comparison argument below.
For the orbit arc to be closed (i.e., for $A'$ and $C'$ to coincide in the $w$-$y$ plane), the vertical coordinates of $A^{**}$  and $C^{**}$ 
must satisfy one of the following three possible relationships:
		(a) $y_{A^{**}}>y_{C^{**}}$;~~
		(b) $y_{A^{**}}=y_{C^{**}}$;~~
		(c) $y_{A^{**}}<y_{C^{**}}$.
	
	\begin{figure}[!htb]
		\centering
		\subfloat[$y_{A^{**}}>y_{C^{**}}$]
		{\includegraphics[scale=0.55]{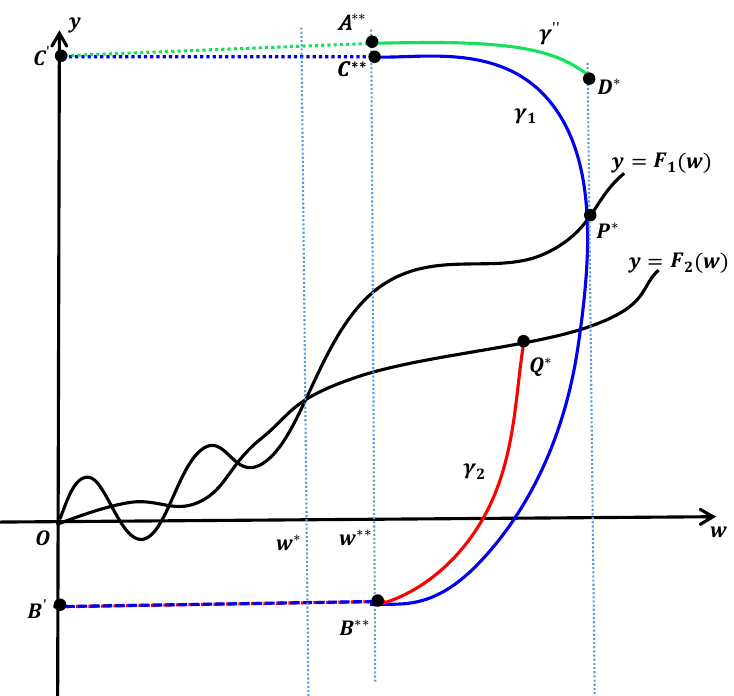}}\hspace{10pt}
		\subfloat[$y_{A^{**}}=y_{C^{**}}$]
		{\includegraphics[scale=0.55]{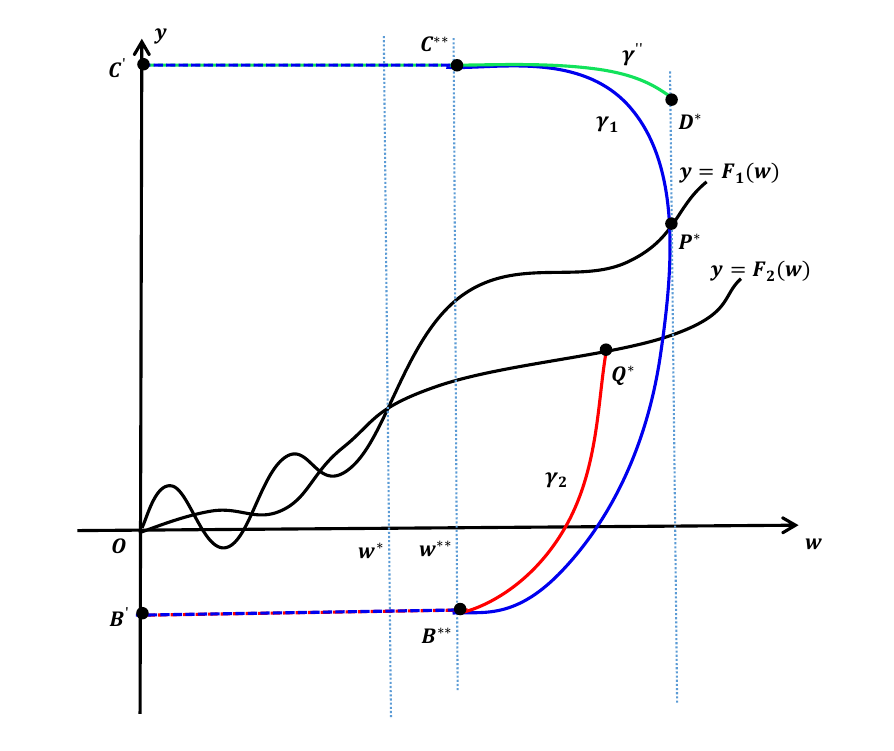}}\hspace{10pt}
		\subfloat[$y_{A^{**}}<y_{C^{**}}$ and $\bar{w}=w_{P^*}$]
		{\includegraphics[scale=0.55]{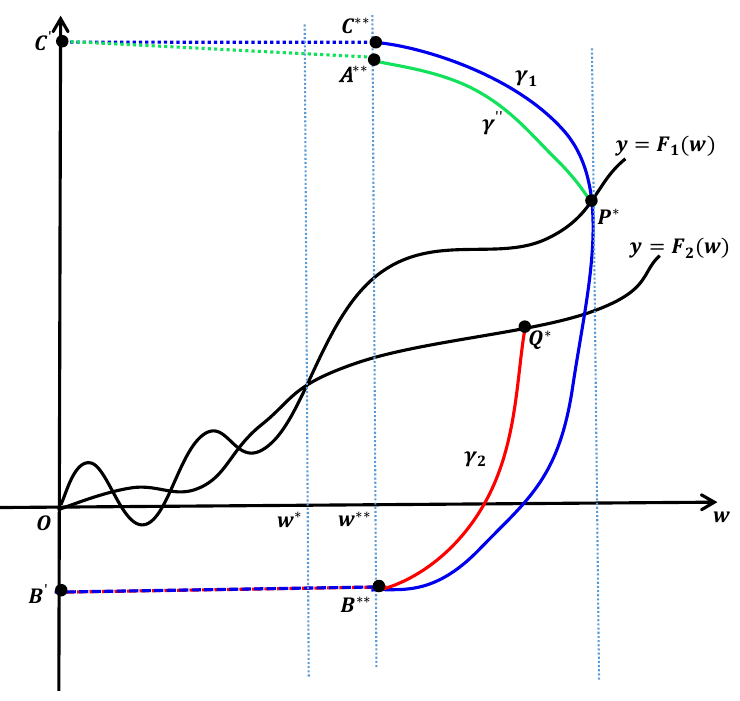}}\hspace{10pt}	
		\subfloat[$y_{A^{**}}<y_{C^{**}}$ and $\bar{w}<w_{P^*}$]
		{\includegraphics[scale=0.55]{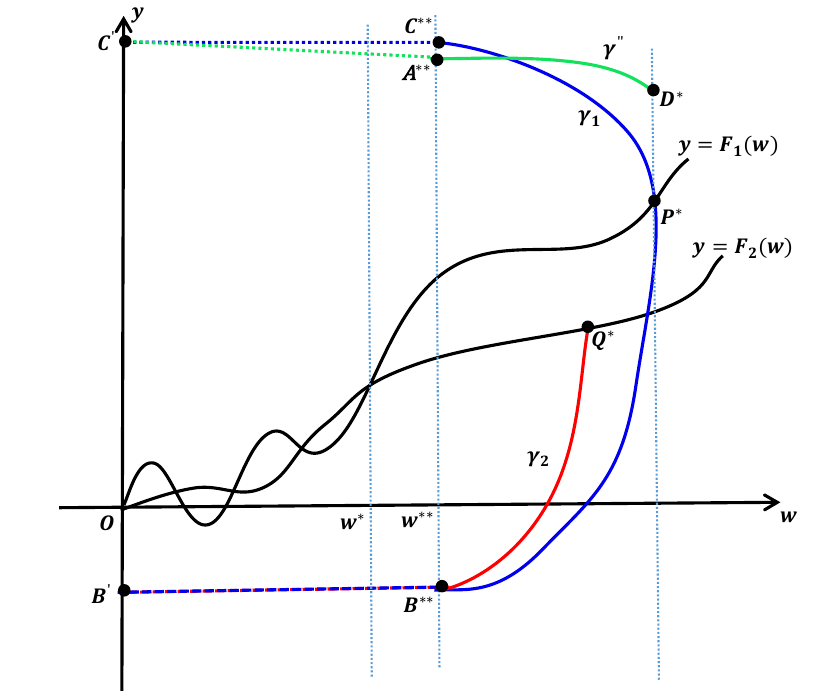}}\hspace{10pt}	
		\caption{{\footnotesize $\widehat{ABC}$ is not a closed orbit in case {\bf(C2-3)}.}}
		\label{c2}
	\end{figure}

We define an auxiliary integral curve $\gamma^{''}$ as the solution of equation \eqref{f2} starting from point $A^{**}$.  Applying the comparison theorem to the differential equations  \eqref{f1} and \eqref{f2} when  (a) holds, we determine that $\gamma^{''}\cap\{(w,y)|w\in(w^{**},w_{P^{*}})\}$ lies on the right-hand side  of the integral curve $\widehat{C^{**}P^*}$.
	Therefore, $\gamma^{''}$ must intersect the line $w=w_{P^{*}}$ at the point $D^{*}=(w_{P^{*}},y_{D^{*}})$, see Fig. \ref{c2}(a).

	When (b) holds,  a similar application of the comparison theorem confirms that $\gamma^{''}\cap\{(w,y)|w\in(w^{**},w_{P^{*}})\}$ of \eqref{f2} lies on the right-hand side of the integral curve $\widehat{C^{**}P^{*}}$ of \eqref{f1}.
	Then, $\gamma^{''}$ must intersect the line $w=w_{P^{*}}$ at the point $D^{*}=(w_{P^{*}},y_{D^{*}})$, see Fig. \ref{c2}(b).

Let $G^*=(w_{G^*}, y_{G^*})$ be the intersection point of $\gamma^{''}$ and $y=F_2(w)$ when (c) holds. The points on 
	$\widehat{C^{**}P^*}\cap\{(w,y)|w^{**}\leq w <\min\{w_{G^*},w_{P^*}\}\}$ and $\widehat{A^{**}G^*}\cap\{(w,y)|w^{**}\leq w <\min\{w_{G^*},w_{P^*}\}\}$ can be parameterized by $(w,y_1(w))$ and  $(w,y_2(w))$ respectively. Denote the vertical distance function
	$z(w): =y_2(w)-y_1(w)$. Using a method analogous to that in Step II(iii) of Lemma \ref{C1},
	with $w^{**}$ as the reference point, we obtain
	\begin{eqnarray*}
		\begin{aligned}
			z(w)=M(w)\exp\left\{\int_{w^{**}}^{w}\frac{\tau^k}{\left(F_2(\tau)-y_2(\tau)\right)\left(F_1(\tau)-y_1(\tau)\right)}{\rm d}\tau\right\},
		\end{aligned}
		\label{z2}
	\end{eqnarray*}
where the function $M(w)$ is  monotonically increasing. Similarly to the case {\bf(C1)}, if we assume $z(w) < 0$ for all $w \in \left[w^{**}, \min\{w_{G^{*}}, w_{P^{*}}\}\right)$, then $M(w) < 0$. However, this leads to a contradiction
\begin{itemize}
	\item[(\romannumeral1)] If $w_{G^{*}} < w_{P^{*}}$ (Fig. \ref{c2d}(a)), as $w$ approaches $w_{G^{*}}$ from the left, $y_2(w)$ approaches $F_2(w)$, causing $M(w)$ to tend to $+\infty$.
	\item [(\romannumeral2)] If $w_{G^{*}} \geq w_{P^{*}}$ (Fig. \ref{c2d}(b)), as $w$ approaches $w_{P^{*}}$ from the left, $y_1(w)$ approaches $F_1(w)$, again causing $M(w)$ to tend to $+\infty$.
\end{itemize}

Both scenarios contradict $M(w) < 0$. Therefore, there must exist a point $$\bar{w} \in (w^{**}, \min\{w_{G^{*}}, w_{P^{*}}\}),$$ where $z(\bar{w}) = 0$.
\begin{itemize}
	\item[(\romannumeral1)] If $\bar{w} = w_{P^{*}}$, then $\gamma^{\prime\prime}$ intersects the line $w = w_{P^{*}}$ at point $P^{*}$, as shown in Fig. \ref{c2}(c).
	\item[(\romannumeral2)] If $w^{**} < \bar{w} < w_{P^{*}}$, then for $w \in (\bar{w}, w_{P^{*}})$, the comparison theorem can be applied again, showing that $\gamma^{\prime\prime}$ lies to the right of $\widehat{C^{**}P^{*}}$. Thus, $\gamma^{\prime\prime}$ must intersect the line $w = w_{P^{*}}$ at a point $D^{*}$ above $P^{*}$, as shown in Fig. \ref{c2}(d).
\end{itemize}

In both subcases, the auxiliary curve $\gamma^{\prime\prime}$ fails to connect to the required point to close the orbit arc $\widehat{B'C'}$, completing the contradiction for Case (c).
		\begin{figure}[!htb]
	\centering
	\subfloat[$w_{G^*}<w_{P^*}$]
	{\includegraphics[scale=0.55]{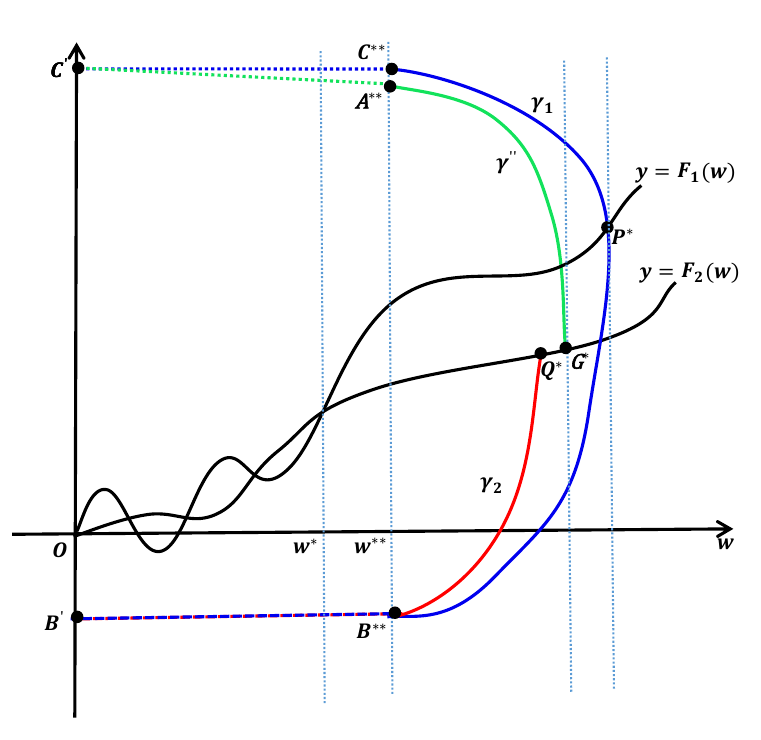}}\hspace{10pt}
	\subfloat[$w_{G^*}\geq w_{P^*}$]
	{\includegraphics[scale=0.55]{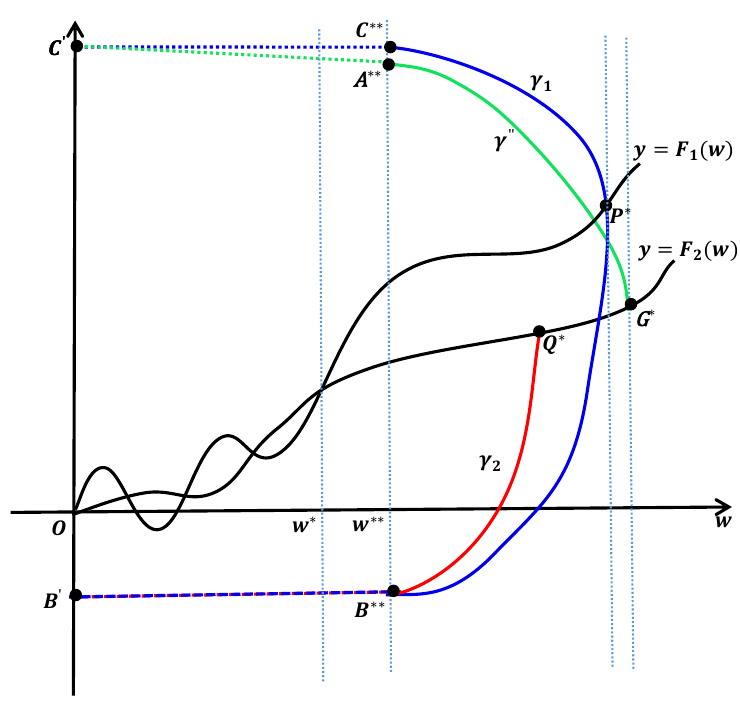}}\hspace{10pt}
	\caption{{\footnotesize $z(w)<0$ when $y_{A^{**}}<y_{C^{**}}$ in case {\bf(C2-3)}.}}
	\label{c2d}
\end{figure}

	{\bf Step II(iv).}
	Let $\widehat{B'A'}$ cross $\bar B^{**}=(w^{**}, y_{\bar B^{**}})$. It is evident that there are the following three types of relationships:
	\begin{center}
			(a) $y_{\bar B^{**}}>y_{B^{**}}$;~~
		(b) $y_{\bar B^{**}}=y_{B^{**}}$;~~
		(c) $y_{\bar B^{**}}<y_{B^{**}}$.
		\end{center}

		Repeating {\bf Step II(iii)} gives $w_{Q^*}<w_{P^*}$. Here
		$Q^*=(w_{Q^*},F_2(w_{Q^*}))$ is the intersection of
		$\widehat{B'A'}$ and the curve $y=F_2(w)$.
	Furthermore, $Q^{*}$ represents the right-most point on $\gamma_2$ since $y = F_2(w)$ is the vertical isocline of equation \eqref{f2}. As a result,  $D^*$ cannot lie on $\gamma_2$ due to the inequality $w_{Q^{*}} < w_{P^{*}}$, indicating that $\gamma_2$ and $\gamma^{''}$ are distinct.
	Thus, based on the uniqueness of solutions, the intersection point between $\gamma_2$ and $w = w^{**}$ cannot coincide at $C^{**}$. In other words, $A^{**}$ and $C^{**}$ do not coincide and the orbit arc $\widehat{ABC}$ is not closed.
\end{proof}

Thirdly, we consider case {\bf (C3)}. According to Lemmas \ref{index}-\ref{index1}, we can conclude that system \eqref{lst} must have $2N$ equilibria  with index -1 and $2N+1$ equilibria with index 1, while the index of all other equilibria is $0$.
We can label these $4N+1$ equilibria  with index 1 or -1 as  $\bar{E}_1 =(\bar{x}_1,\bar{y}_1), \bar{E}_2=(\bar{x}_2,\bar{y}_2), \cdots, \bar{E}_{4N+1}=(\bar{x}_{4N+1},\bar{y}_{4N+1})$, where $\bar{x}_1<\bar{x}_2<\cdots<\bar{x}_{4N+1}$
Therefore, without loss of generality, we can move $\bar{E}_{2N+1}$  to the origin. 
Assume that $E^{*}=(x^*, y^*)$ is an equilibrium  of system \eqref{lst}.
By Lemma \ref{index},  we  can obtain that $q$ is odd and $a_q>0$ when the index of $E^{*}$ is $-1$, $q$ is odd and $a_q<0$ when the index of $E^{*}$ is $1$ and $q$ is even when the index of $E^{*}$ is $0$.  
Furthermore, by Lemma \ref{index1},  we can conclude that for any two equilibria with index $1$, there exists an equilibrium with index $-1$ whose abscissa lies between their abscissas. And for any two equilibria with index $-1$, there exists an equilibrium with index $1$ whose abscissa lies between their abscissas.
In this case, there will be $N$ equilibria with  index  $1$ and $N$ saddles on the left (resp., right) side of the $y$-axis. We can label the $2N$ equilibria on the left (resp., right) side of the $y$-axis as  $E_{l(2N)}=(x_{l(2N)}, y_{l(2N)}), E_{l(2N-1)}=(x_{l(2N-1)}, y_{l(2N-1)}), \cdots, E_{l1}=(x_{l1}, y_{l1})$, where $x_{l(2N)}<x_{l(2N-1)}<\cdots<x_{l1}$ (resp.,  $E_{r(2N)}=(x_{r(2N)}, y_{r(2N)}), E_{r(2N-1)}=(x_{r(2N-1)}, y_{r(2N-1)}), \cdots, E_{r1}=(x_{r1}, y_{r1})$, where $x_{r(2N)}>x_{r(2N-1)}>\cdots>x_{r1}$).
We have that $g(x)\geq0$ when $$x\in\bigcup_{j=1}^{N}(x_{l(2j)}, x_{l(2j-1)})\cup(0, x_{r1})\cup\bigcup_{j=2}^{N}(x_{r(2j-2)}, x_{r(2j-1)})\cup(x_{r(2N)}, +\infty)$$ and  $g(x)\leq0$ when $$x\in(-\infty, x_{l(2N)})\cup\bigcup_{j=2}^{N}(x_{l(2j-1)}, x_{l(2j-2)})\cup(x_{l1}, 0)\cup\bigcup_{j=1}^{N}(x_{r(2j-1)}, x_{r(2j)}).$$

\begin{lm}
	In case {\bf (C3)}, the set $\mathcal{G}$ is a closed-orbit set if and only if 
\[
	F_1(w)\equiv F_2(w)~~{\text{for}}~~0\leq w\leq +\infty.
\]
	\label{C3}
\end{lm}
\begin{proof}
	The proof employs a structured two-step approach that builds upon the methodologies developed for cases {\bf(C1)} and {\bf(C2)}, while addressing the specific topological features of the even saddle configuration.
	\begin{description}
		\item[(I)] We first prove the sufficiency. 
		And the proof will be divided  into the following two parts.
		\begin{description}
			\item[(i)] We first prove that  $w$ satisfies $x_1(w)=x_{lj}$ if and only if $x_2(w) = x_{rj}$ for any $j=1, 2, \cdots, 2N$.
			\item[(ii)] We can apply a similar proof of {\bf Step I} in Lemma \ref{C1} to prove that the orbit arc $\widehat{ABC}$ is closed.
		\end{description}
		\item[(II)] Next, we will use a proof by contradiction to demonstrate its necessity.  Due to the complexity of the proof, we have divided it into two parts.
		\begin{description}
			\item[(i)] If $F_1\not\equiv F_2$, the semialgebraic observation above
			shows that $\{F_1-F_2\neq0\}$ has a rightmost connected component
			$I_*=(\alpha,\beta)$, on which $F_1-F_2$ has a fixed strict sign.
			As in {\bf Step II(i)-(ii)} of Lemma~\ref{C2}, a bounded right
			endpoint $\beta$ is incompatible with closure of every outer orbit:
			on the common tail the two folded equations have the same continuation,
			whereas the strict separation created on $I_*$ cannot disappear.
			Hence, in the only remaining case, $\beta=+\infty$.
			After interchanging the two folded branches if necessary, choose
			$w^*\in I_*$ so that
			$F_1(w)>F_2(w)$ for all $w>w^*$.
			\item[(ii)] Assume that  there exists an orbit arc $\widehat{ABC}$ in  $\mathcal{G}$.
			Let $\widehat{B^{'}C^{'}}$ cross $B^{**}=(w^{**}, y_{B^{**}})$ and $C^{**}=(w^{**}, y_{C^{**}})$  where $w^{**}=\max\{w^*,\max\Sigma,w(x_{l(2N)}),w(x_{r(2N)})\}$, $y_{B^{**}}$ is a number sufficiently small and tending to $-\infty$  and $y_{C^{**}}$ is a number  sufficiently large and tending to $+\infty$.
			Clearly, there are three possible types of
			relationships between $y_{A^{**}}$ and $y_{C^{**}}$ where $A^{**}$ represents another intersection of $\widehat{B^{'}A^{'}}$ and the line $w=w^{**}$.
			Define $\gamma^{''}$ as the integral curve of equation \eqref{f2} starting from $A^{**}$. 
			Like the proof  of {\bf Step II(iii-iv)} in Lemma \ref{C2}, we next will prove that $\widehat{ABC}$ is not closed
			in each of these three cases.
		\end{description}
	\end{description}

	{\bf Step I: Sufficiency.}  
	
	{\bf Step I(i).} 
	The foundation of the sufficiency argument lies in establishing a precise correspondence between equilibrium points on the left and right halves of the phase space. Consider equations (\ref{dF/dw}-\ref{eq2}) from the previous analysis.
	If $x_1(\bar{w})=x_{lj}$ where $1\leq j\leq 2N$, the odd multiplicity $r_1$ implies that the corresponding point $x_2(\bar{w})$ must also satisfy $g(x_2(\bar{w}))=0$ and $g(x_2(w))=a_{r_2}(x_2(w)-\bar{x}_2)^{r_2}+o((x_2(w)-\bar{x}_2)^{r_2})$ with odd multiplicity $r_2$.
	The possible candidates are  $x_{ri}$, where $1\leq i\leq2N$.
	We will now prove that $i=j$.
	Assume by contradiction that $x_1(\bar{w})=x_{l1}$ and $x_2(\bar{w})=x_{ri}$ with $i>1$.  The monotonicity properties of 
	$w(x)$ (monotonically decreasing on $(-\infty,0)$ and monotonically increasing on $(0,+\infty)$) yield  the ordering
	$$w(x_{l(2N)})>w(x_{l(2N-1)})>\cdots>w(x_{l1}),~~
	w(x_{r1})<w(x_{r2})<\cdots<w(x_{r(2N)}).$$
	This implies  $w(x_{r1})<w(x_{ri})=w(x_{l1})$.
	However, when $x_2(\bar{w})=x_{r1}$, no corresponding $x_1(\bar{w})$ satisfies equation \eqref{eq}, creating a contradiction. Thus, $i=1$.
	This reasoning extends to all  $1< j\leq 2N$, establishing the one-to-one correspondence.
	
	{\bf Step I(ii).} 
	The established correspondence implies that $\text{sgn}(g_1(w))=-\text{sgn}(g_2(w))$ whenever $g_1(w)\neq0$ and $g_2(w)\neq0$. Consequently, equations \eqref{f1} and \eqref{f2} are identical on every regular component of $[0,+\infty)\setminus\Sigma$.
	Ordinary uniqueness gives coincidence of the two folded traces on each
	regular component, and the continuation argument above carries this coincidence through the finitely many critical levels. Therefore the orbit
	arc $\widehat{ABC}$ is closed.
	
	{\bf Step II: Necessity.}
	
	{\bf Step II(i).}
	Set \(D=F_1-F_2\). By the semialgebraic observation above, if
	\(D\not\equiv0\), then \(\{D\neq0\}\) has a rightmost connected component
	\(I_*=(\alpha,\beta)\), on which \(D\) has a fixed strict sign.
	If \(\beta<+\infty\), then \(D=0\) for all \(w>\beta\).
	Applying the root-pairing argument of {\bf Step I(i)} successively on this
	common tail shows that
	\[
	-\operatorname{sgn}g_1(w)=\operatorname{sgn}g_2(w)
	\]
	on every regular component of the tail. Hence the two folded equations
	coincide there. On the other hand, the standard strict comparison theorem,
	applied on \(I_*\), produces a strict separation of the two traces issuing
	from the same intersection with the \(y\)-axis. Such a separation cannot
	disappear under the common tail flow by uniqueness, and
	the continuation argument above carries the separation through all critical levels. Therefore an outer orbit crossing \(I_*\) cannot close, contrary to
	the definition of \(\mathcal G\).

	Consequently \(\beta=+\infty\). Interchanging the two folded branches, or
	equivalently applying \((x,y,t)\mapsto(x,-y,-t)\), if necessary, we may
	choose \(w^*\in I_*\) such that
	\[
	F_1(w)>F_2(w),\qquad w>w^*.
	\]

	{\bf Step II(ii).}
	We demonstrate that no orbit arc \(\widehat{ABC}\in\mathcal{G}\) can be
	closed under this strict inequality. Let
	\(\widehat{B'C'}\) cross
	\(B^{**}=(w^{**},y_{B^{**}})\) and
	\(C^{**}=(w^{**},y_{C^{**}})\), where
	\[
	w^{**}=\max\{w^*,\max\Sigma,w(x_{l(2N)}),w(x_{r(2N)})\}.
	\]
	Here \(y_{B^{**}}\to-\infty\) and \(y_{C^{**}}\to+\infty\) as the chosen
	outer orbit tends to infinity. Let
	\(P^*=(w_{P^*},F_1(w_{P^*}))\) be the intersection of
	\(\widehat{B'C'}\) with \(y=F_1(w)\).
	Since $w^{**}$ is fixed, the same estimate used in case {\bf(C1)}, with
$W=w^{**}$, gives
	\[
	|y_{C'}-y_{C^{**}}|=O(|y_{C^{**}}|^{-1}),
	\qquad
	|y_{A'}-y_{A^{**}}|=O(|y_{A^{**}}|^{-1}),
	\]
	where \(A^{**}\) is the other intersection of \(\widehat{B'A'}\) with
	\(w=w^{**}\). These errors can therefore be made smaller than every fixed
	strict separation occurring in the comparison below.
	It is evident that there are the following three types of relationships:
		(a) $y_{A^{**}}>y_{C^{**}}$;~~
		(b) $y_{A^{**}}=y_{C^{**}}$;~~
		(c) $y_{A^{**}}<y_{C^{**}}$.
	
	For each case, by applying the detailed methodology developed in {\bf Step II(iii-iv)} of Lemma \ref{C2},  we can obtain that the orbit arc $\widehat{ABC}$ is not closed in the above three scenarios, as shown in Fig. \ref{c2}.
\end{proof}
\begin{proof}[{\bf Proof of Theorem \ref{cf-thm1}}]
	The proof of Theorem \ref{cf-thm1} employs a comprehensive two-step approach that establishes both the sufficiency and 
	necessity of the stated conditions for $P_{\infty}$ to be a center at infinity.
	
	{\bf Step 1: Sufficiency.} 
	The proof begins by invoking Lemma  \ref{inf}, which guarantees that when condition {\bf(i)} is satisfied, system \eqref{ls} exhibits no orbits connecting equilibria at infinity within the Poincar\'{e} disc. 
	This topological constraint forces the orbits near infinity to adopt specific configurations:
	\begin{enumerate}
		\item They must either form spirals approaching infinity.
		\item Or constitute closed orbits surrounding the infinite region.
	\end{enumerate}
	This fundamental observation ensures the existence of the set
	$\mathcal{G}$ at infinity, which serves as the domain for our subsequent analysis.
	
	We next analyze system \eqref{lst}. The proper Li\'{e}nard shear preserves the exterior closed-orbit property, which is the only equivalence at infinity used in this proof.
	As explained in Section \ref{sec:pre}, for every $w_0>0$ there is a unique pair
	\[
	x_1(w_0)<0<x_2(w_0)
	\]
	on the two inverse branches such that
	\[
	G(x_1(w_0))=G(x_2(w_0)),\qquad
	w(x_1(w_0))=w(x_2(w_0))=w_0.
	\]
	At $w_0=0$ the two branches meet at $x=0$. When statement {\bf(ii)} holds,
	$F(x_1(w_0))=F(x_2(w_0))$ for every $w_0>0$, and continuity gives the same equality at $w_0=0$. Hence
	\[
	F_1(w)\equiv F_2(w),\qquad w\geq0.
	\]
	
	By combining Lemma \ref{index} and Lemma \ref{index1}, we determine that the equilibria of system  \eqref{lst} can only possess indices of $1$, $-1$, or $0$. Under the condition that $a_m > 0$ and $m$ is odd (included in statement {\bf(i)}), the sum of the indices of the equilibria of system \eqref{lst} equals $1$. In this case, to determine whether all orbits of system \eqref{lst} in $\mathcal{G}$ are closed, we can divide our analysis into three parts: {\bf(C1)-(C3)}. Applying Lemmas \ref{C1}-\ref{C3} to each respective case confirms that under the identity
	$F_1(w)\equiv F_2(w),$
	all orbits of system \eqref{lst} within  $\mathcal{G}$ are closed. Applying the inverse proper Li\'{e}nard shear transfers this exterior period annulus to system \eqref{ls}, thereby establishing $P_{\infty}$ as a center at infinity.
	\\
	{\bf Step 2: Necessity.}
	We begin with the premise that $P_{\infty}$ is a center at infinity  of  system \eqref{ls}. Thus, there is an exterior period annulus and, in particular, no orbit connecting the two equilibria at infinity. Lemma~\ref{inf} therefore yields statement {\bf(i)}. The proper Li\'{e}nard shear maps a sufficiently far-out subannulus into an exterior closed-orbit set $\mathcal G$ of system \eqref{lst}.

	Lemma \ref{index} implies that the equilibria of system \eqref{lst} must have indices of 1, -1, or 0. Moreover, by applying Lemma \ref{index1} and noting that $a_m > 0$ and $m$ is odd, we can conclude that the sum of the indices of these equilibria is 1. Therefore, we only need to consider the following three cases: {\bf(C1)-(C3)}.
	Since all orbits in $\mathcal{G}$ are closed, we can apply Lemma \ref{C1}-\ref{C3} to show that $F_1(w)\equiv F_2(w)$ for all $0\leq w\leq+\infty$. Returning to \eqref{w}, we can obtain the form of $w$ and conclude that if $G(x_1) = G(x_2)$ for all $x_1<0<x_2$, then $w(x_1)=w(x_2)$ holds. Thus, $F_1(w(x_1))=F_2(w(x_2))$, i.e., $F(x_1)=F(x_2)$. Statement {\bf(ii)} holds, and the proof of Theorem \ref{cf-thm1} is complete.
\end{proof}
{
\begin{proof}[Proof of Corollary~\ref{cf-cor1}]
Assume first that $f$ is odd.  Then $F$ is even.  Since $g$ is odd, $G$ is
even and strictly increasing as a function of $|x|$ on the relevant branches.
Hence $G(x_1)=G(x_2)$ with $x_1<0<x_2$ implies $x_2=-x_1$, and therefore
$F(x_1)=F(x_2)$.  Thus condition {\bf(ii)} of Theorem~\ref{cf-thm1} holds.

Conversely, assume that $P_\infty$ is a center.  Condition {\bf(ii)} of
Theorem~\ref{cf-thm1} holds.  Since $G(-x)=G(x)$, it follows that
$F(-x)=F(x)$; hence $F$ is even and $f=F'$ is odd.  The conclusion follows
from Theorem~\ref{cf-thm1}.
\end{proof}
}

\subsection[Proofs of the second center--focus theorem]{Proofs of Theorem \ref{cf-thm2} and Corollaries \ref{cf-cor2}--\ref{cf-cor3}}
\begin{proof}[{\bf Proof of Theorem \ref{cf-thm2}}]
The proof will be divided into the following two steps.\\
{\bf Step 1: Sufficiency.} 
Assume the system is in the specified rescaled form:  
$\dot{x}=y,~~ \dot{y}=-g(x)-bxy$, with $g(x)$ being an odd function. Under this condition, the function $f(x)=bx$ is also odd. Furthermore, the parameter condition $m=3, a_m>b^2/8$ or $m>3, a_m>0$ ensures that condition {\bf(i)}  of Theorem \ref{cf-thm1} is satisfied. Since both $f(x)$ and $g(x)$ are odd, the simplified criterion of Corollary  \ref{cf-cor1} applies directly. Therefore, by Corollary  \ref{cf-cor1}, we conclude that $P_{\infty}$ is a center at infinity.
\\
{\bf Step 2: Necessity.}
Assume that \(P_{\infty}\) is a center at infinity for system \eqref{ls}
with \(n=1\). Then
\[
f(x)=b_0+b_1x,\qquad
F(x)=b_0x+\frac{b_1}{2}x^2,
\]
where \(b_1\neq0\). By statement {\bf(ii)} of
Theorem~\ref{cf-thm1}, for every pair \(x_1<0<x_2\) satisfying
\(G(x_1)=G(x_2)\), one has \(F(x_1)=F(x_2)\). Since \(x_1\neq x_2\),
\[
0=F(x_1)-F(x_2)
=(x_1-x_2)\left(b_0+\frac{b_1}{2}(x_1+x_2)\right),
\]
and hence
\[
x_1+x_2=-\frac{2b_0}{b_1}.
\]
Now let \(x_1\to0^-\). Since \(G(x_1)\to0\) and the restriction of \(G\)
to \([0,+\infty)\) is continuous and strictly increasing with \(G(0)=0\),
the equality \(G(x_1)=G(x_2)\) implies \(x_2\to0^+\). Passing to the
limit in the preceding identity gives \(b_0=0\).

Consequently \(F(x)=b_1x^2/2\), so condition {\bf(ii)} gives
\(x_2=-x_1\) for every paired pair. Therefore
\[
G(-x)=G(x),\qquad x>0.
\]
Differentiating this identity for \(x>0\) away from the finitely many
zeros of \(g\), and then using continuity, yields
\[
|g(-x)|=|g(x)|,\qquad x\in\mathbb R.
\]
Thus the polynomial identity
\[
g(-x)^2\equiv g(x)^2
\]
holds. Since \(\mathbb R[x]\) is an integral domain,
either \(g(-x)\equiv g(x)\) or \(g(-x)\equiv-g(x)\).
Condition {\bf(i)} of Theorem~\ref{cf-thm1} gives that \(m=\deg g\) is odd,
so the first alternative is impossible. Hence \(g\) is odd.
This proves the necessity.
\end{proof}
To give the global phase portraits in the Poincar\'{e} disc of system \eqref{m3n1}, some lemmas are presented. We can define the functions $\bar{g}(x)=ax+x^3$ and $\bar{f}(x)=bx$ to aid our analysis.
\begin{lm}
 System \eqref{m3n1} has a global center at the origin when $a\geq0$. 
 \label{lm:m3n1gc}
\end{lm}
\begin{proof}
   We use \cite[Corollary 3]{CLZ} to show that system \eqref{m3n1} has a global center at the origin. We next verify its conditions one by one.
    It is worth noting that $x\bar{g}(x)>0$ for $x\neq0$ when $a\geq0$. 
    Then, statement {\bf(i)} holds.
    When $a>0$, it is apparent that statement {\bf(ii)} holds, and when $a=0$, statement {\bf(ii*)} holds because $b^2<8$.
 Additionally, $b^2<8$, which also  implies that statement {\bf(iii)} holds. Furthermore, $\bar{f}(x)$ and $\bar{g}(x)$ are both odd. Therefore,  system \eqref{m3n1} satisfies statements {\bf(i)-(iii)} of \cite[Corollary 3]{CLZ}. Consequently,  system \eqref{m3n1} has a global center at the origin when $a\geq0$, as shown in Fig. \ref{fig:m3n1}(a).
\end{proof}
\begin{lm}
System \eqref{m3n1} exhibits three equilibria $E_0=(0,0)$, $E_l=(-\sqrt{-a},0)$ and $E_r=(\sqrt{-a},0)$ when $a<0$. Moreover, $E_0$ is a saddle, $E_l$ is an unstable focus and $E_r$ is a stable focus.
\label{lm:m3n1e}
\end{lm}
\begin{proof}
It is clear that system \eqref{m3n1} has three equilibria $E_0=(0,0)$, $E_l= (-\sqrt{-a},0)$ and $E_r=(\sqrt{-a},0)$.
Due to the Jacobian matrix 
\begin{equation*}
	J_0=
	\left[\begin{array}{cc}
		0 & 1 \\
		-a & 0	
	\end{array}\right]
\end{equation*} 
at $E_0$ of system \eqref{m3n1}, obviously, $E_0$ is a saddle since $D_0=a<0$.

The Jacobian matrix at $E_l$ is 
\begin{equation*}
	J_l=
	\left[\begin{array}{cc}
		0 & 1 \\
		2a & b\sqrt{-a}	
	\end{array}\right]
\end{equation*}
implying that $T_l: ={\rm tr}J_l=b\sqrt{-a}>0$ and $D_l: ={\rm det}J_l=-2a>0$. It is clear that $E_l$ is an anti-saddle.  $E_l$ is an unstable focus since $\Delta_l: =T_l^2-4D_l=(8-b^2)a<0$.

The Jacobian matrix at $E_r$ is 
\begin{equation*}
	J_r=
	\left[\begin{array}{cc}
		0 & 1 \\
		2a & -b\sqrt{-a}	
	\end{array}\right]
\end{equation*}
implying that $T_r: ={\rm tr}J_r=-b\sqrt{-a}<0	$ and $D_r: ={\rm det}J_r=-2a>0$. It is clear that $E_r$ is an anti-saddle.  $E_r$ is a stable focus since $\Delta_r: =T_r^2-4D_r=(8-b^2)a<0$.
\end{proof}

\begin{lm}
\label{lm:halfplane-reversible}
For either system \eqref{m3n1} or \eqref{m5n1}, let \(\Gamma\) be a
periodic orbit or a Jordan homoclinic loop whose interior contains an
equilibrium \(E\) with positive \(x\)-coordinate but contains no equilibrium
with negative \(x\)-coordinate. Then the interior of \(\Gamma\) is contained
in \(x>0\). The analogous statement holds with left and right interchanged.
\end{lm}
\begin{proof}
	Both systems \eqref{m3n1} and \eqref{m5n1} are odd in $x$, so the flow is reversible under the involution
	$\mathcal R(x,y)=(-x,y).$
	A regular intersection $z=(0,y_0)$ of $\Gamma$ with the $y$-axis is fixed by $\mathcal R$.
	By uniqueness $\mathcal R\phi_t(z)=\phi_{-t}(z),$ hence $\Gamma$ and its Jordan interior are $\mathcal R$-invariant.
	The interior would then contain $\mathcal RE=(-x_E,0)$ together with $E=(x_E,0)$, contradicting the hypothesis.
	Thus no regular point of $\Gamma$ lies on $x=0$. A homoclinic loop minus its saddle is connected and confined to $x>0$ (resp., $x<0$) depending on which side contains the equilibrium.
\end{proof}

\begin{lm}
System \eqref{m3n1} admits neither a limit cycle nor a homoclinic loop that encloses only the equilibrium $E_r$ $($resp., $E_l$ $)$.
\label{lm:m3n1lc}
\end{lm}

\begin{proof}
By Lemma~\ref{lm:halfplane-reversible}, a periodic orbit or Jordan
homoclinic loop enclosing only \(E_r\) has its interior in \(x>0\), while
the left-hand counterpart has its interior in \(x<0\). The divergence is
\[
{\rm div}(y,-\bar g(x)-\bar f(x)y)=-bx,
\]
which is strictly negative in \(x>0\) and strictly positive in \(x<0\).
The Bendixson criterion therefore excludes periodic orbits in either
half-plane.

For a Jordan homoclinic loop, let $D$ denote its bounded interior.
On $\partial D$ the vector field $X=(y,-\bar g-\bar f y)$ is tangent to the orbit,
so $X\cdot n=0$ on $\partial D$.
By Green's formula, applied to $D$ punctured by a small disk around the saddle
and passing to the limit, we obtain
\[
0=\int_{\partial D}X\cdot n\,{\rm d}s
=\iint_D \operatorname{div}X\,{\rm d}x\,{\rm d}y,
\]
which is impossible because $\operatorname{div}X$ is of constant sign and
nonvanishing in $D$. Hence no such homoclinic loop exists.
\end{proof}

\begin{proof}[{\bf Proof of Corollary  \ref{cf-cor2}}]
The proof will be divided into the following two steps.\\
{\bf Step 1: Sufficiency.} 
It is easy to check that $P_{\infty}$ is a center at infinity of system \eqref{m3n1} by applying Theorem \ref{cf-thm2}.
\\
{\bf Step 2: Necessity.}
Consider $(m,n)=(3,1)$. According to Theorem \ref{cf-thm2}, we can obtain that system \eqref{ls} where $P_{\infty}$ is a center at infinity can be written as
\begin{equation}
\dot{x}=y,\qquad
\dot{y}=-(a_1x+a_2x^3)-b_1xy,
\label{mn31}
\end{equation}  
where $a_2>b_1^2/8$, $b_1\neq0$ and $a_1,b_1\in\mathbb{R}$.
With a scaling 
\[
(x,y,t)\to\left(\left(\frac{a_2}{b_1^2}\right)^{\frac{1}{4}}b_1x, b_1y, \left(\frac{a_2}{b_1^2}\right)^{\frac{1}{4}}t\right),
\]
system \eqref{mn31} can be changed into \eqref{m3n1} where $a=a_1(b_1^2/a_2)^{1/2}$,  $b=(b_1^2/a_2)^{1/2}$ and $0<b<2\sqrt{2}$.
For simplicity, we will next divide our analysis of the global phase portraits in the Poincar\'{e} disc  into two parts.

First, consider $a\geq0$.
By Lemma \ref{lm:m3n1gc}, we can obtain that  system \eqref{m3n1} has a global center at the origin, as shown in Fig. \ref{fig:m3n1}(a).

Second, consider $a<0$.
 Lemma \ref{lm:m3n1e} tells us that there exist three equilibria, where $E_l$ is an unstable focus, $E_0$ is a saddle, and $E_r$ is a stable focus. Since the index of all equilibria surrounded by a closed orbit is 1, 
 our focus will be on investigating the existence and number of limit cycles surrounding only $E_l$, $E_r$, or all three equilibria. 
According to Lemma \ref{lm:m3n1lc}, there are no limit cycles or homoclinic loops surrounding  only $E_l$, or $E_r$. 
Moreover, since $P_\infty$ is a center at infinity, there must  exist closed orbits surrounding the three equilibria. 
Furthermore,  we can conclude that one of the unstable manifolds of $E_0$ connects $E_r$, and one of the stable manifolds  of $E_0$  connects $E_l$. 
By Lemma \ref{C3}, we can obtain that the orbit  of system \eqref{m3n1} that intersects both the positive and negative halves of the $y$-axis must be closed.
Then, the remaining unstable manifold must connect the remaining stable manifold.
Therefore,  we can obtain the global phase portraits in the Poincar\'{e} disc, as shown in Fig.\ref{fig:m3n1}(b).
\end{proof}
To give the global phase portraits in the Poincar\'{e} disc of system \eqref{m5n1}, some lemmas are presented.
%Then, we consider the local dynamics of system \eqref{m5n1}.
 We can define the functions $\hat{g}(x)=ax+cx^3+x^5$ and $\hat{f}(x)=bx$ to aid our analysis.
\begin{lm}
 In $\mathcal{G}_1$, system \eqref{m5n1} has a global center at the origin. 
 \label{lm:m5n1gc}
\end{lm}
\begin{proof} 
  We will demonstrate that system \eqref{m5n1} has a global center at the origin using \cite[Corollary 3]{CLZ}. Firstly, we establish that system \eqref{m5n1} possesses a global center at the origin by verifying the sufficient conditions outlined in \cite[Corollary 3]{CLZ}. 
  
  Firstly, define the auxiliary function $\hat{g}_1(x)$ such that  $x\hat{g}(x): =x^2\hat{g}_1(x)$. Its derivative is given by ${{\rm d}\hat{g}_1(x)}/{{\rm d}x}=4x(x^2+{c}/{2})$. When $c\geq0$ and $a\geq0$, $\hat{g}_1(x)$ is monotonically decreasing on $(-\infty,0)$ and monotonically increasing on $(0,+\infty)$. Since $\hat{g}_1(0)=a\geq0$, it follows that $\hat{g}_1(x)>0$ for all $x\neq0$. Consequently, $x\hat{g}(x)>0$ for all $x\neq0$. 
  
  On the other hand, when $c<0$ and $c^2-4a<0$, $\hat{g}_1(x)$ exhibits more complex monotonicity: it decreases on $(-\infty,-\sqrt{-2c}/2)$, increases on $(-\sqrt{-2c}/2,0)$, decreases on $(0,\sqrt{-2c}/2)$, and  increases on $(\sqrt{-2c}/2,+\infty)$. Additionally, $\hat{g}_1(-\sqrt{-2c}/2)=a-c^2/4>0$, $\hat{g}_1(0)=a>0$, and $\hat{g}_1(\sqrt{-2c}/2)=a-c^2/4>0$. Hence, $\hat{g}_1(x)>0$ for all $x\in\mathbb{R}$, and again $x\hat{g}(x)>0$ for all $x\neq0$. In both cases, condition {\bf(i)} in \cite[Corollary 3]{CLZ} is satisfied. We now verify the remaining conditions.
  If $a>0$, condition {\bf(ii)} holds directly. If $a=0$, condition {\bf(ii*)} is satisfied since  $c>b^2/8$. Condition {\bf(iii)} is readily verified by inspection. Furthermore, both $\hat f(x)$ and $\hat g(x)$ are odd. Since all hypotheses of \cite[Corollary 3]{CLZ} are met, system \eqref{m5n1} admits a global center at the origin when the parameters lie in $\mathcal{G}_1$. The corresponding global phase portrait is illustrated in Fig. \ref{fig:m5n1}(a).      
\end{proof}
\begin{lm}
 In $\mathcal{G}_2$, system \eqref{m5n1} has a  unique equilibrium  at the origin  and the neighborhood of the origin consists of a hyperbolic sector and an elliptical sector. 
 \label{lm:m5n1g2}
\end{lm}
\begin{proof}
It is easy to check that there is a  unique equilibrium at the origin and we can obtain that the neighborhood of the origin consists of a hyperbolic sector and an elliptical sector by applying \cite[Theorem 7.2 of Chapter 2]{ZDHD}. 
\end{proof}
\begin{lm}
System \eqref{m5n1} exhibits  three  equilibria $E_0=(0,0)$, $E_{l1}=(-\sqrt{(-c+\sqrt{c^2-4a})/2},0)$ and  $E_{r1}=(\sqrt{(-c+\sqrt{c^2-4a})/2},0)$ in $\mathcal{G}_3\cup\mathcal{G}_4\cup\mathcal{G}_5$.
Moreover, in $\mathcal{G}_3\cup\mathcal{G}_4$, $E_0$ is a saddle, $E_{l1}$ is a source and $E_{r1}$ is a sink.
In $\mathcal{G}_5$, $E_0$ is a center, $E_{l1}$ is a saddle-node with an unstable nodal part, and $E_{r1}$ is a saddle-node with a stable nodal part.
\label{lm:m5n1g345}
\end{lm}
\begin{proof}
 It is clear that system \eqref{m5n1} has  three  equilibria $E_0=(0,0)$, $E_{l1}= (-\sqrt{(-c+\sqrt{c^2-4a})/2},0)$ and $E_{r1}= (\sqrt{(-c+\sqrt{c^2-4a})/2},0)$ in  $\mathcal{G}_3\cup\mathcal{G}_4\cup\mathcal{G}_5$.
Due to the Jacobian matrix 
 \begin{equation*}
 J_0=
 \left[\begin{array}{cc}
	0 & 1 \\
	-a & 0	
	\end{array}\right]
 \end{equation*} 
 at $E_0$ of system \eqref{m5n1}, we obtain that  $T_{0}: ={\rm tr}J_{0}=0$ and $D_{0}: ={\rm det}J_{0}=a$.
In $\mathcal{G}_3$, it is evident that $E_{0}$ is a saddle due to the fact that $D_{0}=a<0$. 
In $\mathcal{G}_4$, $E_0$ is  a saddle according to \cite[Theorem 7.2 of Chapter 2]{ZDHD}. 
In $\mathcal{G}_5$, $E_0$ can either be a center or a focus because $T_0=0$ and $D_0>0$. 
Since  system \eqref{m5n1} is symmetric with respect to the $y$-axis by the transformation $(x,y, t)\to(-x,y,-t)$, it follows that $E_0$ must be a center.
  
 The Jacobian matrix at $E_{l1}$ is 
  \begin{equation*}
 J_{l1}=
 \left[\begin{array}{cc}
	0 & 1 \\
	4a-c^2+c\sqrt{c^2-4a} & b\sqrt{(-c+\sqrt{c^2-4a})/2}	
	\end{array}\right]
 \end{equation*}
  implying that $T_{l1}: ={\rm tr}J_{l1}=b\sqrt{(-c+\sqrt{c^2-4a})/2}$ and $D_{l1}: ={\rm det}J_{l1}=-4a+c^2-c\sqrt{c^2-4a}$. It is clear that $E_{l1}$ is an anti-saddle because $D_{l1}>0$ and  $E_{l1}$ is a source since $T_{l1}>0$ in  $\mathcal{G}_3\cup\mathcal{G}_4$. In $\mathcal{G}_5$, $D_{l1}=0$.
  By the transformation $(x,y,t)\to(x-\sqrt{-c/2}+y/(b\sqrt{-c/2}), y, t/(b\sqrt{-c/2}))$, system \eqref{m5n1} becomes
  {\scriptsize
  \begin{equation}
  \begin{aligned}
  &\dot{x}=(\frac{\sqrt{-2c}}{-bc})^2\left[(x+\frac{\sqrt{-2c}}{-bc}y)^2(x+\frac{\sqrt{-2c}}{-bc}y-\frac{\sqrt{-2c}}{2})(x+\frac{\sqrt{-2c}}{-bc}y-\sqrt{-2c})^2+b(x+\frac{\sqrt{-2c}}{-bc}y)y\right]: =P(x,y),\\
  &\dot{y}=y-\frac{\sqrt{-2c}}{-bc}\left[(x+\frac{\sqrt{-2c}}{-bc}y)^2(x+\frac{\sqrt{-2c}}{-bc}y-\frac{\sqrt{-2c}}{2})(x+\frac{\sqrt{-2c}}{-bc}y-\sqrt{-2c})^2+b(x+\frac{\sqrt{-2c}}{-bc}y)y\right]: =y+Q(x,y).
  \end{aligned}
  \label{t1}
  \end{equation}}
 By the Implicit Function Theorem there is a unique function $y=\psi(x)$ such that $\psi(0)=0$ and $\psi(x)+Q(x, \psi(x))=0$,
where	 $\psi(x)=-2cx^{2}/b+o(x^{2})$.
	Substituting $y=\psi(x)$ into $P(x,y)$, we obtain that $P(x,y)=-2\sqrt{-2c}x^{2}/{b}^2+o(x^{2})$.
	By \cite[ Theorem 7.1 of Chapter 2]{ZDHD},
	the origin of \eqref{t1} is a saddle-node, so is $E_{l1}$ of \eqref{m5n1}.
	Since $T_{l1}>0$, $E_{l1}$ is a saddle-node with an unstable  nodal part.
	
  The Jacobian matrix at $E_{r1}$ is 
  \begin{equation*}
 J_{r1}=
 \left[\begin{array}{cc}
	0 & 1 \\
	4a-c^2+c\sqrt{c^2-4a} & -b\sqrt{(-c+\sqrt{c^2-4a})/2}	
	\end{array}\right]
 \end{equation*}
  implying that $T_{r1}: ={\rm tr}J_{r1}=-b\sqrt{(-c+\sqrt{c^2-4a})/2}$ and $D_{r1}: ={\rm det}J_{r1}=-4a+c^2-c\sqrt{c^2-4a}$. It is clear that $E_{r1}$ is an anti-saddle because $D_{r1}>0$ and  $E_{r1}$ is a sink since $T_{r1}<0$ in  $\mathcal{G}_3\cup\mathcal{G}_4$.
  In $\mathcal{G}_5$, $D_{r1}=0$.
   By the transformation $(x,y,t)\to(x+\sqrt{-c/2}+y/(-b\sqrt{-c/2}), y, t/(-b\sqrt{-c/2}))$, system \eqref{m5n1} becomes
{\scriptsize
  \begin{equation}
  \begin{aligned}
  &\dot{x}=-(\frac{\sqrt{-2c}}{-bc})^2\left[(x-\frac{\sqrt{-2c}}{-bc}y)^2(x-\frac{\sqrt{-2c}}{-bc}y+\frac{\sqrt{-2c}}{2})(x-\frac{\sqrt{-2c}}{-bc}y+\sqrt{-2c})^2+b(x-\frac{\sqrt{-2c}}{-bc}y)y\right]: =P(x,y),\\
  &\dot{y}=y+\frac{\sqrt{-2c}}{-bc}\left[(x-\frac{\sqrt{-2c}}{-bc}y)^2(x-\frac{\sqrt{-2c}}{-bc}y+\frac{\sqrt{-2c}}{2})(x-\frac{\sqrt{-2c}}{-bc}y+\sqrt{-2c})^2+b(x-\frac{\sqrt{-2c}}{-bc}y)y\right]: =y+Q(x,y).
  \end{aligned}
  \label{t2}
  \end{equation}}
   By the Implicit Function Theorem there is a unique function $y=\psi(x)$ such that $\psi(0)=0$ and $\psi(x)+Q(x, \psi(x))=0$,
where	 $\psi(x)=2cx^{2}/b+o(x^{2})$.
	Substituting $y=\psi(x)$ into $P(x,y)$, we obtain that $P(x,y)=2\sqrt{-2c}x^{2}/{b}^2+o(x^{2})$.
	By \cite[ Theorem 7.1 of Chapter 2]{ZDHD},
	the origin of \eqref{t2} is a saddle-node, so is $E_{r1}$ of \eqref{m5n1}.
	Since $T_{r1}<0$, $E_{r1}$ is a saddle-node with a stable nodal part.
\end{proof}
\begin{lm}
System \eqref{m5n1} exhibits   five equilibria $E_0=(0,0)$, $E_{l1}=(-\sqrt{(-c+\sqrt{c^2-4a})/2},0)$, $E_{l2}=(-\sqrt{(-c-\sqrt{c^2-4a})/2},0)$, $E_{r2}=(\sqrt{(-c-\sqrt{c^2-4a})/2},0)$ and $E_{r1}=(\sqrt{(-c+\sqrt{c^2-4a})/2},0)$ in $\mathcal{G}_6$.
Moreover, $E_0$ is a center, $E_{l2}$ and $E_{r2}$ are  saddles, $E_{l1}$ is a source and $E_{r1}$ is a sink.
\label{lm:m5n1g6}
\end{lm}
\begin{proof}
The equilibrium equation \(x(a+cx^2+x^4)=0\) gives the five
stated equilibria. At any nonzero equilibrium \((x_*,0)\),
the Jacobian satisfies
\[
\operatorname{tr}J(x_*,0)=-bx_*,
\qquad
\det J(x_*,0)=2x_*^2(c+2x_*^2).
\]
The determinant is negative at \(E_{l2},E_{r2}\) and positive
at \(E_{l1},E_{r1}\). Thus the inner equilibria are saddles;
since \(b>0\), the outer left and right equilibria are a source
and a sink, respectively.

At \(E_0\), the eigenvalues are \(\pm i\sqrt a\).
The reversibility under \((x,y,t)\mapsto(-x,y,-t)\)
excludes a focus, so \(E_0\) is a center.
\end{proof}

\begin{lm}
In $\mathcal{G}_3\cup\mathcal{G}_4\cup\mathcal{G}_6$, system \eqref{m5n1} has neither limit cycle  nor homoclinic loop only surrounding  $E_{l1}$ (resp., $E_{r1}$).
\label{lm:m5n1lc}
\end{lm}
\begin{proof}
By Lemma~\ref{lm:halfplane-reversible}, any periodic orbit or Jordan
homoclinic loop enclosing only \(E_{r1}\) has its interior in \(x>0\);
the left-hand counterpart has its interior in \(x<0\). For every parameter
region under consideration,
\[
{\rm div}(y,-\hat g(x)-\hat f(x)y)=-bx,
\]
so the divergence is strictly negative in \(x>0\) and strictly positive in
\(x<0\). The Bendixson criterion excludes periodic orbits. The same Green
formula argument as in Lemma~\ref{lm:m3n1lc} excludes Jordan homoclinic
loops. Hence no such orbit can enclose only \(E_{r1}\) or only \(E_{l1}\).
\end{proof}

\begin{proof}[{\bf Proof of Corollary \ref{cf-cor3}}]
The proof will be divided into the following two steps.\\
{\bf Step 1: Sufficiency.} 
It is easy to check that $P_{\infty}$ is a center at infinity of system \eqref{m5n1} by applying Theorem \ref{cf-thm2}.
\\
{\bf Step 2: Necessity.}
Consider $(m,n)=(5,1)$. According to Theorem \ref{cf-thm2}, we can obtain that system \eqref{ls} where $P_{\infty}$ is a center at infinity can be written as
\begin{equation}
\dot{x}=y,\qquad
\dot{y}=-(a_1x+a_2x^3+a_3x^5)-b_1xy,
\label{mn51}
\end{equation} 
where $a_3>0$, $b_1\neq0$ and $a_1, a_2\in\mathbb{R}$.
With a scaling 
\[
(x,y,t)\to\left(\left(\frac{a_3}{b_1^4}\right)^{\frac{1}{6}}b_1x, b_1y, \left(\frac{a_3}{b_1^4}\right)^{\frac{1}{6}}t\right),
\]
system \eqref{mn51} can be changed into \eqref{m5n1} where 
\[
a=a_1\left(\frac{b_1^4}{a_3}\right)^{1/3},
\qquad
c=a_2\left(\frac{b_1^2}{a_3^2}\right)^{1/3},
\qquad
b=\left(\frac{b_1^4}{a_3}\right)^{1/3}.
\]
For simplicity, we will divide our analysis into five parts.

In $\mathcal{G}_1$, we can obtain that  system \eqref{m5n1} has a global center at the origin 
by Lemma \ref{lm:m5n1gc}, as shown in Fig. \ref{fig:m5n1}(a).
In $\mathcal{G}_2$, Lemma \ref{lm:m5n1g2} tells us that system \eqref{m5n1} has a unique equilibrium  located at the origin. The neighborhood of the origin consists of a hyperbolic sector and an elliptical sector. 
Since  $P_{\infty}$ is a center at infinity, there must exist a closed orbit surrounding the origin. 
By  Lemma \ref{C1},  we can obtain that the orbit  of system \eqref{m5n1} that intersects both the positive and negative halves of the $y$-axis must be closed. Consequently, the elliptical sector  must be bounded, as shown in   Fig. \ref{fig:m5n1}(b).

In $\mathcal{G}_3\cup\mathcal{G}_4$, we can obtain that there are three finite equilibria, where $E_{l1}$ is a source, $E_{0}$ is a saddle and $E_{r1}$ is a sink by Lemma  \ref{lm:m5n1g345}. 
Due to the fact that  the index of all equilibria surrounded by a closed orbit is $1$, 
our focus will be on investigating the existence and number of limit cycles surrounding only $E_{l1}$, $E_{r1}$, or all three equilibria. 
According to Lemma \ref{lm:m5n1lc}, there are no limit cycles or homoclinic loops surrounding  only $E_{l1}$, or $E_{r1}$. 
Moreover, since $P_\infty$ is a center at infinity, there must  exist closed orbits surrounding the three equilibria. 
Furthermore,  we can conclude that one of the unstable manifolds of $E_0$ connects $E_{r1}$, and one of the stable manifolds  of $E_0$  connects $E_{l1}$. 
By Lemma \ref{C3}, we can obtain that the orbit  of system \eqref{m5n1} that intersects both the positive and negative halves of the $y$-axis must be closed.
Then, the remaining unstable manifold must connect  the remaining stable manifold.
Therefore,  we can obtain the global phase portraits in the Poincar\'{e} disc, as shown in Fig. \ref{fig:m5n1}(c).

In $\mathcal{G}_5$, we can obtain that there are three finite equilibria, where $E_0$ is  a center, $E_{r1}$ is a saddle-node with stable nodal part, and $E_{l1}$ is a saddle-node with unstable nodal part  by Lemma  \ref{lm:m5n1g345}. 
Moreover, since $P_\infty$ is a center at infinity, there must  exist closed orbits surrounding the three equilibria.     
By Lemma \ref{C3}, we can obtain that the orbit  of system \eqref{m5n1} that intersects both the positive and negative halves of the $y$-axis must be closed.      
Then, we can obtain the global phase portraits in the Poincar\'{e} disc, as shown in Fig. \ref{fig:m5n1}(d).

In $\mathcal{G}_6$, we can obtain that there are five finite equilibria, 
where $E_0$ is a center, $E_{l2}$ and $E_{r2}$ are saddles, $E_{l1}$ is a source and  $E_{r1}$ is a sink by Lemma  \ref{lm:m5n1g6}. 
Due to the fact that the index of all equilibria surrounded by a closed orbit is 1, we first investigate the existence and number of limit cycles surrounding only $E_{l1}$ and $E_{r1}$.
According to Lemma \ref{lm:m5n1lc}, there are no limit cycles or homoclinic loops surrounding  only $E_{l1}$ or $E_{r1}$. 
We next investigate the existence and number of limit cycles  surrounding  $E_{l1}$, $E_{l2}$ and $E_0$ (resp., $E_{r1}$, $E_{r2}$ and $E_0$). 
By Lemma \ref{C3}, we can obtain that the orbit  of system \eqref{m5n1} that intersects both the positive and negative halves of the $y$-axis must be closed.
Then, two stable (unstable) manifolds of $E_{r2}$ (resp., $E_{l2}$) would have to form a closed orbit, which is impossible. Therefore, there cannot exist a closed orbit containing three equilibria.
Since $P_{\infty}$ is a center at infinity, we can obtain the global phase portraits in the Poincar\'{e} disc, as shown in Fig. \ref{fig:m5n1}(e). 
      \end{proof}
     \subsection{Concluding remarks}\label{sec:conclusion}
       In \cite{DH}, the transformation $x={\rm Cs}\theta/s, y={\rm Sn}\theta/s^{l}$ (\({\rm Cs}\) and \({\rm Sn}\) are functions with a period of $2\pi$ and \({\rm Cs}^2\theta+{\rm Sn}^2\theta=1\) ) is used to convert the center--focus problem at infinity for system \eqref{ls} into the center--focus problem at the origin of system \eqref{DHS}, where
        \begin{equation}
        \begin{aligned}
        &\dot{s}={\rm Sn}\theta\sum_{k=0}^{2l-2}\bar{a}_k{\rm Cs}^k\theta s^{2l-k}+{\rm Sn}^2\theta {\rm Cs}^n\theta s^{l-n}+{\rm Sn}^2\theta\sum_{k=0}^{n-1}\bar{b}_k{\rm Cs}^k\theta s^{l-k},\\
        &\dot{\theta}=-\left(1+{\rm Cs}\theta \sum_{k=0}^{2l-2}\bar{a}_k{\rm Cs}^k\theta s^{2l-k-1}+{\rm Sn}\theta {\rm Cs}^{n+1}\theta s^{l-n-1}+{\rm Cs}\theta {\rm Sn}\theta\sum_{k=0}^{n-1}\bar{b}_k{\rm Cs}^k\theta s^{l-k-1}\right).
        \end{aligned}
        \label{DHS}
        \end{equation}
         As it was emphasized by Poincar\'{e}, it is the matter of great importance to find typical situations for which all equations of the system determining centers are satisfied.
          Therefore, while this method might be  useful in determining specific systems, it still poses significant challenges for studying system \eqref{ls}.

We provide a necessary and sufficient condition for a generalized polynomial Li\'{e}nard system to have a center at infinity, thereby solving its center--focus problem at infinity.
However, verifying this condition remains challenging.

%%%%%%%%%%%%%%%%%%%%%%%%%%%%%%%%%%%
 %%%%%%%%%%%%%%%%%%%%%%%%%%%%%%%%%%%
\section{The cyclicity of Ilyashenko-Kotova lips}
In this section, we investigate the cyclicity of system \eqref{or}. 
The analytical strategy employs multiple techniques to specific regions of the phase space. Near the origin ($H\rightarrow0$), we utilize polar coordinate transformations and asymptotic expansions of the first return map. For the regime near infinity ($H\rightarrow\infty$), we employ compactification techniques to transform the system into a form amenable to standard perturbative analysis. Near the lips ($H\rightarrow1/6$), we generalize the Melnikov function expansions from the one-cuspidal loop to the two-cuspidal loop configuration.
%%%%%%%%%%%%%%%%%%

\subsection{Limit cycles near lips}\label{subsec:limit-cycles-lips}
Based on the topological structure of Ilyashenko-Kotova lips, this subsection investigates limit cycle bifurcations in two-cuspidal loop. By establishing asymptotic expansions of Melnikov functions, we analyze the generation mechanism of limit cycles under small parameter perturbations. 

From \eqref{H}, near $H=1/6$,
there are two families of periodic orbits denoted by $\left\{L_h^{ \pm}\right\}$
\begin{eqnarray}
	\begin{aligned}
		& L_h^{-}: \frac{1}{2} y^2+\frac{1}{2} x^2-\frac{1}{2} x^4+\frac{1}{6}x^6=h \quad{\text{if}}~~ 0<h-\frac{1}{6}\ll1, \\
		& L_h^{+}: \frac{1}{2} y^2+\frac{1}{2} x^2-\frac{1}{2} x^4+\frac{1}{6}x^6=h \quad {\text{if}}~~0<\frac{1}{6}-h\ll1.
	\end{aligned}
	\label{3.1}
\end{eqnarray}
By symmetry considerations, the following integrals vanish
$$\oint_{L_h^{ \pm}} x^{2i+1}y{\rm{d}}x=0\quad \text{for all } i \in \mathbb{Z}.$$
All periodic orbits \(L_h^\pm\) are oriented according to the Hamiltonian
flow of the unperturbed system.
Consequently, the Melnikov functions for system \eqref{or} can be simplified to
\begin{eqnarray}
	M^{\pm}(h)=\oint_{L_h^{ \pm}} \sum_{i=0}^n a_i x^{2i} y {\rm d}x=\sum_{i=0}^n a_i I_i^{ \pm}(h),
	\label{M}
\end{eqnarray}
where
$$
I_i^{ \pm}(h)=\oint_{L_h^{ \pm}} x^{2i} y {\rm d}x.
$$
To precisely count the number of limit cycles that can bifurcate from the lips, we now derive a specific structure for the Melnikov functions associated with our perturbed system.

\begin{lm}
\label{MM}
The Melnikov functions \(M^\pm(h)\) in \eqref{M} can be rewritten as
\[
M^\pm(h)=
\begin{dcases}
a_0I_0^\pm(h)+a_1I_1^\pm(h), & n=1,\\
a_0I_0^\pm(h)+a_1I_1^\pm(h)+a_2I_2^\pm(h), & n=2,\\
\displaystyle
\sum_{j=0}^{[ n/3]}A_{0j}^n h^j I_0^\pm(h)
+
\sum_{j=0}^{[ (n-1)/3]}A_{1j}^n h^j I_1^\pm(h)
+
\sum_{j=0}^{[ (n-2)/3]}A_{2j}^n h^j I_2^\pm(h),
& n\geq3.
\end{dcases}
\]
Here \(A_{\nu j}^n\), \(\nu=0,1,2\), are linear functions of
\(a_0,\ldots,a_n\). Moreover, all the coefficients \(A_{\nu j}^n\) appearing
above are independent linear combinations of \(a_0,\ldots,a_n\).
\end{lm}

\begin{proof}
The cases \(n=1,2\) are immediate. For \(i\geq3\),  integration of \({\rm{d}}x^{2i-5}y\) around $L^{\pm}$, using
\(y^2=2h-x^2+x^4-x^6/3\), gives
\[
0=(2i-5)\left(2hI_{i-3}^\pm-I_{i-2}^\pm
+I_{i-1}^\pm-\frac13I_i^\pm\right)
-3\left(I_{i-2}^\pm-2I_{i-1}^\pm+I_i^\pm\right).
\]
Thus,
\[
I_i^\pm(h)
=
\alpha_i h I_{i-3}^\pm(h)
+
\beta_i I_{i-2}^\pm(h)
+
\gamma_i I_{i-1}^\pm(h),
\quad \alpha_i=\frac{12i-30}{2i+4},\quad
\beta_i=-\frac{6i-6}{2i+4},\quad
\gamma_i=\frac{6i+3}{2i+4}.
\]
Notice that
\(\alpha_i\neq0\) and \(i\geq3.\)  Set
$
\lambda_0=\lambda_1=\lambda_2=1
$ and
$\lambda_i=\alpha_i\lambda_{i-3}\neq0$ for $i\geq3.$
Induction in the recurrence shows that, for \(i=3m+r\), \(r\in\{0,1,2\}\),
\[
I_i^\pm(h)
=
\lambda_i h^m I_r^\pm(h)
+
\sum_{\substack{0\leq \mu\leq2,\ s\geq0\\ 3s+\mu<i}}
\lambda_{i,\mu s}h^s I_\mu^\pm(h),
\]
where the constants \(\lambda_{i,\mu s}\) are independent of the perturbation
parameters. Indeed, the assertion is obvious for \(i=0,1,2\). Assume it holds for all
indices smaller than \(i\). Since \(i=3m+r\), the term
\(\alpha_i hI_{i-3}^\pm\) contributes
\[
\alpha_i h\lambda_{i-3}h^{m-1}I_r^\pm
=
\lambda_i h^m I_r^\pm.
\]
All other terms produced by
\(
hI_{i-3}^\pm\), \(I_{i-2}^\pm\),\(I_{i-1}^\pm
\)
have lower reduced index \(3s+\mu<i\). Hence the claimed formula follows by
induction. Substitution into 
$
M^\pm(h)=\sum_{i=0}^n a_iI_i^\pm(h).
$
Collecting the coefficients of \(h^jI_\nu^\pm(h)\), we obtain
\[
M^\pm(h)
=
A_0^n(h)I_0^\pm(h)
+
A_1^n(h)I_1^\pm(h)
+
A_2^n(h)I_2^\pm(h),
\]
where
\[
A_\nu^n(h)=\sum_{j\geq0}A_{\nu j}^n h^j,
\qquad \nu=0,1,2.
\]
By the leading-term formula above, the degrees satisfy
\[
\deg A_0^n\leq \left[\frac n3\right],\qquad
\deg A_1^n\leq \left[\frac{n-1}{3}\right],\qquad
\deg A_2^n\leq \left[\frac{n-2}{3}\right].
\]
Thus the representation in the statement follows.

It remains to prove the independence of the coefficients \(A_{\nu j}^n\).
Order the coefficients according to the single index
\(
k=3j+\nu\), \(\nu=0,1,2.
\)
That is, we use the order
\[
A_{00}^n,\ A_{10}^n,\ A_{20}^n,\ A_{01}^n,\ A_{11}^n,\ A_{21}^n,\ldots .
\]
For \(k=3j+\nu\leq n\), the coefficient \(A_{\nu j}^n\) has the triangular
form
\[
A_{\nu j}^n
=
\lambda_{3j+\nu}a_{3j+\nu}
+
\mathcal L(a_{3j+\nu+1},a_{3j+\nu+2},\ldots,a_n),
\]
where \(\mathcal L\) denotes a linear combination of the indicated variables.
Indeed, the term \(a_{3j+\nu}I_{3j+\nu}^\pm\) contributes the non-zero leading
term
$
\lambda_{3j+\nu}a_{3j+\nu}h^jI_\nu^\pm(h),
$
whereas no \(a_iI_i^\pm\) with \(i<3j+\nu\) can contribute to
\(h^jI_\nu^\pm(h)\). Therefore the Jacobian matrix
\[
\frac{
\partial(A_{00}^n,A_{10}^n,A_{20}^n,A_{01}^n,A_{11}^n,A_{21}^n,\ldots)
}{
\partial(a_0,a_1,a_2,a_3,a_4,a_5,\ldots,a_n)
}
\]
is upper triangular, with diagonal entries
\(
\lambda_0,\lambda_1,\ldots,\lambda_n
\)
with determinant
$
\prod_{i=0}^n\lambda_i\neq0.
$
Hence the coefficients \(A_{\nu j}^n\) are independent linear combinations of
the original parameters \(a_0,\ldots,a_n\).
\end{proof}

Having established the structure of the Melnikov functions, we now proceed to
the core result of this subsection: the number of zeros that can be generated
near the two-cuspidal loop.

Let
\[
h_0=\frac16,\qquad
\tau_-=h-h_0>0,\qquad
\tau_+=h_0-h>0.
\]
Thus
$
L_h^-:\ h>h_0$ and
$L_h^+:\ h<h_0.$
We use the  expansions
\[
M^-(h)
=
\sum_{j\ge0}C_j^-\tau_-^j
+
\sum_{j\ge0}D_j^-\tau_-^{j+5/6}
+
\sum_{j\ge0}E_j^-\tau_-^{j+7/6},
\]
and
\[
M^+(h)
=
\sum_{j\ge0}C_j^+\tau_+^j
+
\sum_{j\ge0}D_j^+\tau_+^{j+5/6}
+
\sum_{j\ge0}E_j^+\tau_+^{j+7/6}.
\]

For both sides, define the selected  coefficients by
\(
\Lambda_0^\pm=C_0^\pm,
\)
and, for \(j\ge0\),
\[
\Lambda_{3j+1}^\pm=D_j^\pm,\qquad
\Lambda_{3j+2}^\pm=C_{j+1}^\pm,\qquad
\Lambda_{3j+3}^\pm=E_j^\pm.
\]
Thus the selected sequence is
$
C_0^\pm,\ D_0^\pm,\ C_1^\pm,\ E_0^\pm,\ D_1^\pm,\ C_2^\pm,\ldots .
$

For $i=0,1,2$, it is clear, and, for \(i\ge3\),
\[
P_{i\nu}(h)
=
\alpha_i hP_{i-3,\nu}(h)
+
\beta_iP_{i-2,\nu}(h)
+
\gamma_iP_{i-1,\nu}(h),
\qquad \nu=0,1,2,
\]
where
\[
\alpha_i=\frac{12i-30}{2i+4},\qquad
\beta_i=-\frac{6i-6}{2i+4},\qquad
\gamma_i=\frac{6i+3}{2i+4}.
\]

We write
\(
h=h_0-\tau_+\),\(\tau_+>0.
\)
The plus-side expansions of the three basic Abelian integrals are written as
\[
I_\nu^+(h_0-\tau_+)
=
\sum_{k\ge0}u_{\nu k}^+\tau_+^k
+
\sum_{k\ge0}v_{\nu k}^+\tau_+^{k+5/6}
+
\sum_{k\ge0}w_{\nu k}^+\tau_+^{k+7/6},
\qquad \nu=0,1,2.
\]
The coefficients
\(u_{\nu k}^+\), \(v_{\nu k}^+\), \(w_{\nu k}^+\)
depend only on the unperturbed Hamiltonian and on the chosen orientation
of the ovals. They can be given by Corollary \ref{mel-lips}.
Moreover,  we have \(C_1^+\) can be replaced by
\(
(-3\sqrt3\pi,\,-\sqrt3\pi,\,0).
\)

Let
$
P_{i\nu}^{(m)}(h_0)
$ denote the $m$-th derivative at $h_0$.
Expanding
\[
P_{i\nu}(h_0-\tau_+)
=
\sum_{m\ge0}
\frac{(-1)^m}{m!}P_{i\nu}^{(m)}(h_0)\tau_+^m,
\]
the selected coefficients \(\Lambda_r^+\) are linear functions of
\(a_0,\ldots,a_n\). 

Let
\[
\boldsymbol a
=
(a_0,\ldots,a_n)^{\mathsf T},
\qquad
\boldsymbol\Lambda_n^+(a)
=
\bigl(
\Lambda_0^+(a),\ldots,\Lambda_n^+(a)
\bigr)^{\mathsf T}.
\]
Since each \(\Lambda_r^+\) is linear in \(a_0,\ldots,a_n\), write
\[
\Lambda_r^+(a)
=
\sum_{i=0}^n\ell_{ri}^+a_i,\qquad 
L_n^+
=
(\ell_{ri}^+)_{0\leq r,i\leq n},\qquad 
(\ell_{ri}^+)_{0\leq r,i\leq n}=
\left(
\frac{\partial\Lambda_r^+}{\partial a_i}
\right)_{0\leq r,i\leq n}.
\]
Then
$
\boldsymbol\Lambda_n^+(a)
=
L_n^+\boldsymbol a.
$
Consequently, the selected coefficient map is a linear isomorphism if and
only if
\(
\det L_n^+\neq0.
\)

The entries of \(L_n^+\) are as follows. For \(i=0,\ldots,n\),
\[
\ell_{0i}^+
=
\sum_{\nu=0}^2u_{\nu0}^+P_{i\nu}(h_0).
\]
For \(j\ge0\), whenever the corresponding row index is not larger than \(n\),
\[
\ell_{3j+1,i}^+
=
\sum_{\nu=0}^2
\sum_{m=0}^{j}
\frac{(-1)^m}{m!}
P_{i\nu}^{(m)}(h_0)v_{\nu,j-m}^+,
\]
\[
\ell_{3j+2,i}^+
=
\sum_{\nu=0}^2
\sum_{m=0}^{j+1}
\frac{(-1)^m}{m!}
P_{i\nu}^{(m)}(h_0)u_{\nu,j+1-m}^+,
\]
and
\[
\ell_{3j+3,i}^+
=
\sum_{\nu=0}^2
\sum_{m=0}^{j}
\frac{(-1)^m}{m!}
P_{i\nu}^{(m)}(h_0)w_{\nu,j-m}^+.
\]
These formulas define all rows of \(L_n^+\).

For later determinant computations, the first four rows may be written, up to
adding a multiple of a preceding row, as
\[
\begin{aligned}
&\ell_{0i}^+
=
-\sqrt3\pi
\left(
\frac14P_{i0}(h_0)
+
\frac1{24}P_{i1}(h_0)
+
\frac1{64}P_{i2}(h_0)
\right),
\\
&\ell_{1i}^+
=
V
\left(
P_{i0}(h_0)+P_{i1}(h_0)+P_{i2}(h_0)
\right),\\
&\ell_{2i}^+
=
-3\sqrt3\pi\,P_{i0}(h_0)
-\sqrt3\pi\,P_{i1}(h_0)+
\frac{\sqrt3\pi}{4}P_{i0}'(h_0)
+
\frac{\sqrt3\pi}{24}P_{i1}'(h_0)
+
\frac{\sqrt3\pi}{64}P_{i2}'(h_0),
\\
&\ell_{3i}^+
=
W\left(P_{i1}(h_0)+3P_{i2}(h_0)\right).
\end{aligned}
\]

\begin{lm}
\label{lm:Ln-plus-nondegenerate}
 We have
\(
\det L_n^+\neq0.
\)
\end{lm}

\begin{proof}
By Lemma \ref{MM}, the map
$
a\longmapsto
(A_{00}^n,A_{10}^n,A_{20}^n,A_{01}^n,A_{11}^n,A_{21}^n,\ldots)
$
is a linear isomorphism.  Therefore it is enough to prove that the selected
 coefficients are independent with respect to these reduced
coordinates.

Write
\[
A_\nu^n(h_0-\tau_+)
=
\sum_{j=0}^{d_\nu}B_{\nu j}^+\tau_+^j,
\qquad \nu=0,1,2,
\]
where
\[
d_0=\left[\frac n3\right],\qquad
d_1=\left[\frac{n-1}{3}\right],\qquad
d_2=\left[\frac{n-2}{3}\right].
\]
This identity uniquely defines the coefficients \(B_{\nu j}^+\).
The change of coordinates
$
(A_{\nu j}^n)\longmapsto (B_{\nu j}^+)
$
is triangular with non-zero diagonal entries.  Hence it is invertible.

From
\[
M^+(h)
=
\sum_{\nu=0}^2
A_\nu^n(h_0-\tau_+)I_\nu^+(h_0-\tau_+),
\]
we have
\[
C_j^+
=
\sum_{\nu=0}^2\sum_{s=0}^{j}B_{\nu s}^+u_{\nu,j-s}^+,
\quad
D_j^+
=
\sum_{\nu=0}^2\sum_{s=0}^{j}B_{\nu s}^+v_{\nu,j-s}^+,
\quad 
E_j^+
=
\sum_{\nu=0}^2\sum_{s=0}^{j}B_{\nu s}^+w_{\nu,j-s}^+.
\]

For every complete level \(j\), the principal block of the map
\[
(B_{0j}^+,B_{1j}^+,B_{2j}^+)
\longmapsto
(C_j^+,D_j^+,E_j^+)
\]
is
\[
\mathcal M^+
=
\begin{pmatrix}
u_{00}^+&u_{10}^+&u_{20}^+\\
v_{00}^+&v_{10}^+&v_{20}^+\\
w_{00}^+&w_{10}^+&w_{20}^+
\end{pmatrix}
=
\begin{pmatrix}
-\dfrac{\sqrt3\pi}{4}
&
-\dfrac{\sqrt3\pi}{24}
&
-\dfrac{\sqrt3\pi}{64}
\\[2mm]
V&V&V
\\
0&W&3W
\end{pmatrix}.
\]
A direct calculation gives
\[
\begin{aligned}
\det\mathcal M^+
=
VW\left(2u_{00}^+-3u_{10}^+ +u_{20}^+\right)=
-\frac{25\sqrt3\pi}{64}VW.
\end{aligned}
\]
Since $\sigma_0,\sigma_1\neq0$ given in the Corollary \ref{mel-lips} and
\[
V=2^{11/6}3^{1/3}\sigma_0\neq0,\qquad
W=2^{7/6}3^{2/3}\sigma_1\neq0,
\]
we have
\(
\det\mathcal M^+\neq0.
\) We now compute the Jacobian determinant in the three residue classes of \(n\).

\medskip
\noindent
\textbf{Case 1. \(n=3q\).}

The independent Taylor coefficients are
\[
B_{00}^+,B_{10}^+,B_{20}^+;\
B_{01}^+,B_{11}^+,B_{21}^+;\ \ldots ;\
B_{0,q-1}^+,B_{1,q-1}^+,B_{2,q-1}^+;\
B_{0q}^+ .
\]
After a finite permutation of rows, the selected coefficients are ordered as
\[
C_0^+,D_0^+,E_0^+;\
C_1^+,D_1^+,E_1^+;\ \ldots ;\
C_{q-1}^+,D_{q-1}^+,E_{q-1}^+;\
C_q^+ .
\]
Thus the corresponding Jacobian is block triangular with diagonal blocks
$
\mathcal M^+,\ldots,\mathcal M^+,\ u_{00}^+.
$
Thus, up to a non-zero sign coming from the row permutation,
$
\det L_{3q}^+
=
(\det\mathcal M^+)^q u_{00}^+
\neq0.
$

\medskip
\noindent
\textbf{Case 2. \(n=3q+1\).}

The independent Taylor coefficients are
\[
B_{00}^+,B_{10}^+,B_{20}^+;\
B_{01}^+,B_{11}^+,B_{21}^+;\ \ldots ;\
B_{0,q-1}^+,B_{1,q-1}^+,B_{2,q-1}^+;\
B_{0q}^+,B_{1q}^+ .
\]
After a finite permutation of rows, the selected coefficients are ordered as
\[
C_0^+,D_0^+,E_0^+;\
C_1^+,D_1^+,E_1^+;\ \ldots ;\
C_{q-1}^+,D_{q-1}^+,E_{q-1}^+;\
C_q^+,D_q^+ .
\]
The final block is
\[
\mathcal M_1^+
=
\begin{pmatrix}
u_{00}^+&u_{10}^+\\
V&V
\end{pmatrix}.
\]
Hence
$
\det\mathcal M_1^+
=
-\frac{5\sqrt3\pi}{24}V
\neq0.
$
Thus, up to a non-zero sign,
$
\det L_{3q+1}^+
=
(\det\mathcal M^+)^q\det\mathcal M_1^+
\neq0.
$

\medskip
\noindent
\textbf{Case 3. \(n=3q+2\).}

The independent Taylor coefficients are
\[
B_{00}^+,B_{10}^+,B_{20}^+;\
B_{01}^+,B_{11}^+,B_{21}^+;\ \ldots ;\
B_{0,q-1}^+,B_{1,q-1}^+,B_{2,q-1}^+;\
B_{0q}^+,B_{1q}^+,B_{2q}^+ .
\]
After a finite permutation of rows, the selected coefficients are ordered as
\[
C_0^+,D_0^+,E_0^+;\
C_1^+,D_1^+,E_1^+;\ \ldots ;\
C_{q-1}^+,D_{q-1}^+,E_{q-1}^+;\
C_q^+,D_q^+,C_{q+1}^+ .
\]
For the last block, the \(C_{q+1}^+\)-row may be replaced, without changing
the determinant, by the normalized row obtained from Corollary
\ref{mel-lips}.  Hence the final block is
\[
\mathcal M_2^+
=
\begin{pmatrix}
-\dfrac{\sqrt3\pi}{4}
&
-\dfrac{\sqrt3\pi}{24}
&
-\dfrac{\sqrt3\pi}{64}
\\
V&V&V
\\
-3\sqrt3\pi&-\sqrt3\pi&0
\end{pmatrix}.
\]
A direct calculation gives
$
\det\mathcal M_2^+
=
-\frac{15}{32}\pi^2V.
$
Since \(V\neq0\), we have
\(
\det\mathcal M_2^+\neq0.
\)
Consequently, up to a non-zero sign,
$
\det L_{3q+2}^+
=
(\det\mathcal M^+)^q\det\mathcal M_2^+
\neq0.
$

In all three cases, the map
\(
(B_{\nu j}^+)
\longmapsto
(\Lambda_0^+,\Lambda_1^+,\ldots,\Lambda_n^+)
\)
has non-zero Jacobian.  Since the coordinate changes
\(
a\longmapsto(A_{\nu j}^n),
\)
\((A_{\nu j}^n)\longmapsto(B_{\nu j}^+)
\)
are also invertible, the selected coefficient map
\(
a\mapsto(\Lambda_0^+(a),\ldots,\Lambda_n^+(a))
\)
is non-degenerate.  Hence
\(
\det L_n^+\neq0.
\)
\end{proof}

Define
\[
{
\mathcal C_{\rm lip}^a
=
\left\{
a\in\mathbb R^{n+1}:
\begin{array}{l}
\Lambda_j^+(a)\Lambda_{j+1}^+(a)<0,
\qquad j=0,\ldots,n-1,\\[1mm]
0<|\Lambda_0^+(a)|\ll|\Lambda_1^+(a)|\ll\cdots\ll|\Lambda_n^+(a)|
\end{array}
\right\}.
}
\]
By Lemma \ref{lm:Ln-plus-nondegenerate}, the map
$
a\longmapsto
(\Lambda_0^+(a),\ldots,\Lambda_n^+(a))
$
is a linear isomorphism.  Hence $\mathcal C_{\rm lip}^a$ is a nonempty
open cone.

\begin{lm}
\label{hn5l}
{If $a\in\mathcal C_{\rm lip}^a$,}
$M^+$ and $M^-$
have altogether at least
\[
\eta_{\rm lip}(n)
=n+\left[\dfrac{n}{3}\right]+\left[\dfrac {n+1}3\right]
\]
simple zeros near the two-cuspidal loop.
\end{lm}

\begin{proof}
For $a\in\mathcal C_{\rm lip}^a$, the standard dominance argument for
 series gives at least \(n\) simple zeros of \(M^+(h)\).

It remains to count the zeros of \(M^-(h)\).  The minus-side coefficients are
not independent of the plus-side ones, since both are produced by the same
parameter vector \(a\).  Comparing the two  expansions in Corollary
\ref{mel-lips}, and using
\(
\sigma_0>0,\) \(\sigma_1>0,\)
\(\kappa_0<0\), \(\kappa_1<0,
\)
we obtain the leading sign relations
\[
\operatorname{sgn}C_j^-=(-1)^j\operatorname{sgn}C_j^+,
\quad
\operatorname{sgn}D_j^-=(-1)^{j+1}\operatorname{sgn}D_j^+,
\quad
\operatorname{sgn}E_j^-=(-1)^{j+1}\operatorname{sgn}E_j^+.
\]
More precisely, ordering the coefficients by the increasing Puiseux
exponents used in the definition of \(\Lambda_r^\pm\), the comparison of
the two cusp expansions is triangular:
\[
\Lambda_r^-
=
\eta_r\Lambda_r^+
+
\sum_{q=0}^{r-1}t_{rq}\Lambda_q^+,
\qquad r=0,\ldots,n,
\]
where the constants \(t_{rq}\) depend only on the unperturbed Hamiltonian
and \(\eta_r\neq0\). Their signs are
\(
\operatorname{sgn}\eta_0=1,
\)
and, for \(j\ge0\),
$
\operatorname{sgn}\eta_{3j+1}
=
\operatorname{sgn}\eta_{3j+2}
=
\operatorname{sgn}\eta_{3j+3}
=
(-1)^{j+1},
$
whenever the corresponding index does not exceed \(n\). Indeed, these are
exactly the diagonal signs in the three relations for \(C_j^\pm\),
\(D_j^\pm\), and \(E_j^\pm\) above; all terms with smaller Puiseux exponent
have already occurred among \(\Lambda_0^+,\ldots,\Lambda_{r-1}^+\).
Under the strong dominance condition in
\(\mathcal C_{\rm lip}^a\), the lower selected coefficients are negligible
relative to the diagonal term and therefore do not change its sign.

Since the plus-side selected sequence
\(
\Lambda_0^+,\Lambda_1^+,\ldots,\Lambda_n^+
\)
is alternating, the induced minus-side selected sequence
\(
\Lambda_0^-,\Lambda_1^-,\ldots,\Lambda_n^-
\)
has, up to the initial sign, the repeated pattern
\[
+,+,-,\ +,+,-,\ +,+,-,\ldots .
\]
Therefore the number of sign changes among the first \(n+1\) minus-side
selected coefficients is
\[
\left[\dfrac{n}{3}\right]+\left[\dfrac{n+1}{3}\right]
\]
Applying the same  dominance argument to \(M^-(h)\), we get the stated
number of simple zeros of \(M^-(h)\). Adding the \(n\) zeros of \(M^+(h)\) gives
\[
\eta_{\rm lip}(n)
=n+\left[\dfrac{n}{3}\right]+\left[\dfrac{n+1}{3}\right]
\]
This proves the lemma.
\end{proof}

\subsection{Limit cycles near the origin}
\label{subsec:origin-analysis}

We now investigate limit cycles bifurcating from the period annulus near the
origin.  For convenience, we work with the equivalent system \eqref{uy}.
Introduce the polar coordinates
\(
u=r\cos\theta\),\(y=r\sin\theta\).
Then system \eqref{uy} is transformed into
\begin{equation}
\frac{{\rm d}r}{{\rm d}\theta}
=
\frac{
r^2\sin^2\theta\,Q(r,\theta)
}{
1+r\sin\theta\cos\theta\,Q(r,\theta)
},
\label{polar}
\end{equation}
where
\[
Q(r,\theta)
=
\frac{
\varepsilon_a|\cos\theta|
\sum_{i=0}^{n}a_i\omega^{i-\frac12}
+
\varepsilon_b\cos\theta
\sum_{i=0}^{n}b_i\omega^i
}{
(\omega-1)^2
}.
\]
Here \(\omega=\omega(r,\theta)\) is determined by
\begin{equation}
r^2\cos^2\theta
=
\frac{(\omega-1)^3+1}{3}.
\label{om}
\end{equation}

For \(0<r\ll1\), the branch satisfying \(\omega\to0\) as \(r\to0\) has the
expansion
\[
\omega
=
\sum_{m\ge1}e_m r^{2m}\cos^{2m}\theta,
\qquad
e_1=e_2=1.
\]
Moreover,
\[
\frac{1}{(\omega-1)^2}
=
\sum_{m\ge0}q_m r^{2m}\cos^{2m}\theta,
\qquad
\omega^i
=
\sum_{m\ge i}r_{i,m}r^{2m}\cos^{2m}\theta,
\]
and
\[
|\cos\theta|\,\omega^{i-\frac12}
=
\sum_{m\ge0}
l_{i,m}r^{2(i+m)-1}\cos^{2(i+m)}\theta .
\]
The factor \(|\cos\theta|\) in the exact equation cancels the half-power
coming from \(\omega^{i-\frac12}\); hence the resulting local expansion is
analytic in \(r\) and contains ordinary even powers of \(\cos\theta\).
Substituting these expansions into \eqref{polar}, we obtain an analytic
equation of the form
\begin{equation}
\frac{{\rm d}r}{{\rm d}\theta}
=
\sum_{k\ge0}\mathcal A_k(\theta)r^{k+1},
\label{origin-radial-expanded}
\end{equation}
where, up to terms at least quadratic in the perturbation coefficients,
$
\mathcal A_k(\theta)
=
A_k\sin^2\theta\cos^k\theta .
$
Let
\[
\delta=
\max_{0\le i\le n}
\left\{
|\varepsilon_a a_i|,
|\varepsilon_b b_i|
\right\}.
\]
All terms which are at least quadratic in the perturbation coefficients will
be included in \(O(\delta^2)\) below.

The leading coefficients \(A_k\) have the triangular form
\begin{equation}
A_{2j}
=
\varepsilon_a
\left(
\lambda_j a_j+\mathcal L_j^a(a_0,\ldots,a_{j-1})
\right),
\qquad
\lambda_j\neq0,
\label{A-even-origin}
\end{equation}
and
\begin{equation}
A_{2j+1}
=
\varepsilon_b
\left(
\mu_j b_j+\mathcal L_j^b(b_0,\ldots,b_{j-1})
\right),
\qquad
\mu_j\neq0,
\label{A-odd-origin}
\end{equation}
where \(\mathcal L_j^a\) and \(\mathcal L_j^b\) are linear forms in the
indicated preceding parameters.
In particular, we obtain
\[
\begin{aligned}
A_0 &=\varepsilon_a a_0, \quad
A_1 =\varepsilon_b b_0, \quad
A_2 =\varepsilon_a\left(\tfrac{3}{2} a_0+a_1\right), \quad
A_3 =\varepsilon_b\left(2 b_0+b_1\right), \\
A_4 &=\varepsilon_a\left(\tfrac{85}{24} a_0+\tfrac{5}{2} a_1+a_2\right), \quad
A_5 =\varepsilon_b\left(5 b_0+3 b_1+b_2\right),\\
A_6 &=\varepsilon_a\left(\tfrac{147}{16} a_0+\tfrac{161}{24} a_1+\tfrac{7}{2} a_2+a_3\right).
\end{aligned}
\]

Let \(P_0(r)\) denote the first return map near the origin and write
\[
P_0(r)-r=\sum_{m\ge0}C_m^0r^{m+1}.
\]
The selected coefficients near the origin are the even coefficients
$
C_0^0,\ C_2^0,\ C_4^0,\ldots,C_{2n}^0.
$
Their first-order parts are triangular in the \(a\)-parameters:
\[
C_{2j}^0
=
\varepsilon_a
\left(
\gamma_j a_j+\mathcal L_j(a_0,\ldots,a_{j-1})
\right)
+O(\delta^2),
\qquad
j=0,\ldots,n,
\]
where \(\gamma_j\neq0\).  Hence \(a_j\) appears for the first time in
\(C_{2j}^0\), with a nonzero coefficient.

The odd coefficients are passive
$
C_{2j+1}^0=O(\delta^2),\qquad j\ge0.
$
Indeed, the one-fold contribution to \(C_{2j+1}^0\) contains
\[
\int_0^{2\pi}\sin^2\theta\cos^{2j+1}\theta\,d\theta=0,
\]
and all remaining Picard--Brudnyi terms are at least quadratic in the
perturbation coefficients.

We define the origin admissible set by
\[
\mathcal C_{\rm org}
=
\left\{
(a,b,\varepsilon_a,\varepsilon_b):
\begin{array}{l}
0<|C_0^0|
\ll
|C_2^0|
\ll
\cdots
\ll
|C_{2n}^0|\ll1,\\[1mm]
C_0^0C_2^0>0,\quad
C_{2j}^0C_{2j+2}^0<0,\quad j=1,\ldots,n-1,\\[1mm]
C_{2j+1}^0=O(\delta^2)\ \text{is dominated by the selected even terms.}
\end{array}
\right\}.
\]
Equivalently, using the triangular formula above, this set is nonempty:
one chooses \(a_0,a_1,\ldots,a_n\) recursively so that the selected even
coefficients satisfy the required sign and dominance conditions, and then
chooses the effective perturbation size sufficiently small so that the
\(O(\delta^2)\)-terms are dominated.  The \(b\)-parameters do not enter the
first-order origin construction.

\begin{lm}
\label{ln5o}
If
\(
(a,b,\varepsilon_a,\varepsilon_b)\in\mathcal C_{\rm org},
\)
then the first return map \(P_0\) has at least \(n-1\) positive fixed points
near the origin.  Consequently, system \eqref{or} has at least \(n-1\) limit
cycles near the origin.
\end{lm}

\begin{proof}
By the definition of \(\mathcal C_{\rm org}\),
\[
0<|C_0^0|
\ll
|C_2^0|
\ll
\cdots
\ll
|C_{2n}^0|\ll1,
\]
and
\[
C_0^0C_2^0>0,
\qquad
C_{2j}^0C_{2j+2}^0<0,
\quad j=1,\ldots,n-1.
\]

Choose test radii
\(
0<r_0<r_1<\cdots<r_n\ll1
\)
so that at \(r=r_j\), the monomial
\(
C_{2j}^0r^{2j+1}
\)
dominates all other selected even monomials.  This is possible because the
selected even coefficients are strongly separated.

The odd coefficients satisfy
\(
C_{2j+1}^0=O(\delta^2),
\)
and are dominated by the selected even monomials at the same test radii.
The analytic tail is also dominated by taking \(r_n\) sufficiently small.
Therefore
\[
\operatorname{sgn}\bigl(P_0(r_j)-r_j\bigr)
=
\operatorname{sgn}C_{2j}^0,
\qquad
j=0,\ldots,n.
\]

Since \(C_0^0\) and \(C_2^0\) have the same sign, while
\(
C_2^0,\ C_4^0,\ldots,C_{2n}^0
\)
alternate in sign, the values
\[
P_0(r_1)-r_1,\quad
P_0(r_2)-r_2,\quad \ldots,\quad
P_0(r_n)-r_n
\]
alternate in sign.  Hence \(P_0(r)-r\) has at least one zero in each interval
\(
(r_j,r_{j+1}), \)
\(j=1,\ldots,n-1.
\)
Thus \(P_0\) has at least \(n-1\) positive fixed points near the origin.
After an arbitrarily small perturbation inside the same admissible set, these
fixed points may be assumed simple, and the corresponding periodic orbits are
limit cycles.
\end{proof}

%%%%%%%%%%%%%%%%%%%%%%%%%%%%%%%%%%
%%%%%%%%%%%%%%%%%%%%%%%%%%%%%%%%%%
%%%%%%%%%%%%%%%%%%%%%%%%%%%%%%%%%%
%%%%%%%%%%%%%%%%%%%%%%%%%%%%%%%%%%

\subsection{Limit cycles near infinity}
\label{subsec:limit-cycles-infinity}

We study limit cycles bifurcating from sufficiently large energy levels
of the exterior period annulus.  Here ``near infinity'' means that the
energy level \(h\) is chosen sufficiently large, or equivalently that
\(R=(2h)^{-1/6}>0\) is sufficiently small.  In particular, these periodic
orbits belong to the exterior period annulus \(h>1/6\) and are separated
from the two-cuspidal boundary \(h=1/6\).  Thus the limit cycles constructed
in this subsection are spatially separated from those bifurcating from the
lips.

For \(\varepsilon_a=\varepsilon_b=0\), the origin of system \eqref{uy} is
the linear center
\[
\frac{{\rm d}u}{{\rm d}\tau}=y,\qquad
\frac{{\rm d}y}{{\rm d}\tau}=-u .
\] Introduce polar coordinates at infinity in the \((u,y)\)-plane:
\[
u=\frac{\cos\theta}{r},\qquad
y=\frac{\sin\theta}{r},
\qquad
R=r^{1/3}.
\]
Thus \(R\to0^+\) corresponds to infinity.  Put
\[
\beta_i=\varepsilon_b b_i,\qquad i=0,\ldots,n,
\qquad
|\beta|=\max_{0\le i\le n}|\beta_i|.
\]
We write
\[
S_a(\omega)=\sum_{i=0}^n a_i\omega^{i-\frac12},
\qquad
S_\beta(\omega)=\sum_{i=0}^n\beta_i\omega^i,
\qquad
\Omega=\omega-1 .
\]
Here
\[
\omega
=
\left(3\frac{\cos^2\theta}{R^6}-1\right)^{1/3}+1,
\qquad
\Omega
=
\left(3\frac{\cos^2\theta}{R^6}-1\right)^{1/3}.
\]

On every outer arc where \(\Omega\neq0\), the denominator is nonvanishing,
and hence it can be expanded with respect to the perturbation parameters.
Since \(r=R^3\), the cubic parameter jet is
\[
\frac{{\rm d}R}{{\rm d}\theta}
=
-\frac13R^{-2}\frac{\sin^2\theta\,K}{\Omega^2}
+\frac13R^{-5}\frac{\cos\theta\sin^3\theta\,K^2}{\Omega^4}
-\frac13R^{-8}\frac{\cos^2\theta\sin^4\theta\,K^3}{\Omega^6}
+O(K^4).
\]

The points \(\Omega=0\), corresponding to \(x=\pm1\), are regular
crossings of the exterior periodic orbits.  Near these points we work
directly in the original coordinates, so no expansion involving negative
powers of \(\Omega\) is used across \(\Omega=0\). The regular local transitions at \(x=\pm1\) contribute only higher-order
terms and therefore do not affect the coefficients considered below.

For \(m=2,4,6\), define the functions
\(
\chi_{\alpha,\ell}^{(m)}(\theta)
\)
by
\[
\frac{\omega^\alpha}{\Omega^m}
=
\sum_{\ell\ge0}
\chi_{\alpha,\ell}^{(m)}(\theta)
R^{2m-2\alpha+2\ell}.
\]
The leading coefficient satisfies
\[
\chi_{\alpha,0}^{(m)}(\theta)
=
3^{\frac{\alpha-m}{3}}
|\cos\theta|^{\frac{2(\alpha-m)}{3}},
\]
and hence
\[
\chi_{\alpha,0}^{(m)}(\theta)>0
\qquad
\text{for }0<\theta<\frac{\pi}{2}.
\]

We now record the coefficient functions which will be used below.  The
first-order \(a\)-block is
\[
\varepsilon_a
\sum_{i=0}^n
\sum_{\ell\ge0}
a_iA_{i,\ell}(\theta)R^{3-2i+2\ell},
\]
where
\[
A_{i,\ell}(\theta)
=
-\frac13|\cos\theta|\sin^2\theta\,
\chi_{i-\frac12,\ell}^{(2)}(\theta).
\]
The first-order \(b\)-block is
\[
\sum_{j=0}^n
\sum_{\ell\ge0}
\beta_jB_{j,\ell}(\theta)R^{2-2j+2\ell},
\]
where
\[
B_{j,\ell}(\theta)
=
-\frac13\cos\theta\sin^2\theta\,
\chi_{j,\ell}^{(2)}(\theta).
\]
The quadratic \(b^2\)-block is
\[
\sum_{j,k=0}^n
\sum_{\ell\ge0}
\beta_j\beta_kD_{j,k,\ell}(\theta)
R^{3-2j-2k+2\ell},
\]
where
\[
D_{j,k,\ell}(\theta)
=
\frac13\cos^3\theta\sin^3\theta\,
\chi_{j+k,\ell}^{(4)}(\theta).
\]
The mixed \(\varepsilon_a\beta\)-block is
\[
\varepsilon_a
\sum_{i,j=0}^n
\sum_{\ell\ge0}
a_i\beta_jE_{i,j,\ell}(\theta)
R^{4-2i-2j+2\ell},
\]
where
\[
E_{i,j,\ell}(\theta)
=
\frac23|\cos\theta|\cos^2\theta\sin^3\theta\,
\chi_{i+j-\frac12,\ell}^{(4)}(\theta).
\]
Finally, the mixed \(\varepsilon_a\beta^2\)-block is
\[
\varepsilon_a
\sum_{i,j,k=0}^n
\sum_{\ell\ge0}
a_i\beta_j\beta_k
G_{i,j,k,\ell}(\theta)
R^{5-2i-2j-2k+2\ell},
\]
where
\[
G_{i,j,k,\ell}(\theta)
=
-|\cos\theta|\cos^4\theta\sin^4\theta\,
\chi_{i+j+k-\frac12,\ell}^{(6)}(\theta).
\]
When \(\ell=0\), we write simply
\[
A_i=A_{i,0},\quad
B_j=B_{j,0},\quad
D_{j,k}=D_{j,k,0},\quad
E_{i,j}=E_{i,j,0},\quad
G_{i,j,k}=G_{i,j,k,0}.
\]

The coefficient integrals below are understood as limits of the corresponding
truncated outer integrals.  From the explicit leading terms, for
\(q=0,\ldots,n-2\),
\(A_n=O(|\cos\theta|^{(2n-2)/3})\),
\(B_{n-q}=O(|\cos\theta|^{(2n-2q-1)/3})\),
\(D_{n,n-q}=O(|\cos\theta|^{(4n-2q+1)/3})\),
\(E_{n,n-q}=O(|\cos\theta|^{(4n-2q)/3})\), and
\(G_{n,n,n-q}=O(|\cos\theta|^{(6n-2q+2)/3})\) as
\(\cos\theta\to0\).  All these exponents are positive, so the truncated
integrals converge to the quadrant integrals used below.

\begin{lm}
\label{lm:weighted-picard-brudnyi}
Let
\(M=6n-3\), \(W=\frac1M R^M .\)
Assume that, after the formal expansion of the right-hand side with respect
to the effective perturbation parameters, the weighted equation can be written as
\[
\frac{\rm dW}{\rm d\theta}
=
\sum_{m\ge0}\mathscr A_m(\theta)R^m .
\]
{
Let $P_{\infty}(R)$ be the first return map in the $R$-variable.  Then the
coefficient $C_j^\infty$ of $R^j$ in the formal weighted displacement
\[
\Delta_\infty(R)
:=\frac{P_{\infty}(R)^M-R^M}{M}
=\sum_{j\ge0}C_j^\infty R^j
\]
}
is given by
\[
C_j^\infty
=
\sum_{s\ge1}
\sum_{\substack{m_1,\ldots,m_s\ge0\\
m_1+\cdots+m_s-(s-1)M=j}}
\kappa^M_{m_1,\ldots,m_s}
\int_{0\le t_1\le\cdots\le t_s\le2\pi}
\mathscr A_{m_1}(t_1)\cdots \mathscr A_{m_s}(t_s)
\,dt_1\cdots dt_s ,
\]
where
\[
\kappa^M_{m_1,\ldots,m_s}
=
\prod_{\nu=2}^{s}
\left(
m_\nu+\cdots+m_s-(s-\nu)M
\right).
\]
\end{lm}

\begin{proof}
The proof is the same Picard-iteration argument as in the proof of
\cite[Theorem 2.1]{Br1}.  We only record the modification caused by the
weighted variable.

Since
$
W=\frac1M R^M,
$
one has
\[
\frac{\rm d}{\rm dW}R^m=mR^{m-M}.
\]
The one-fold Picard contribution of
$
\mathscr A_m(\theta)R^m
$
has weighted degree \(m\).  For a two-fold ordered contribution
\[
\mathscr A_{m_1}(t_1)\mathscr A_{m_2}(t_2),
\qquad
0\le t_1\le t_2\le2\pi,
\]
the second variation differentiates \(R^{m_2}\) with respect to \(W\).
Thus it contributes the factor \(m_2\) and has weighted degree
\(
m_1+m_2-M .
\)
For a three-fold contribution the factor is
\(
m_3\bigl(m_2+m_3-M\bigr),
\)
and the weighted degree is
\(
m_1+m_2+m_3-2M .
\)
Continuing inductively, a word
\(
(m_1,\ldots,m_s)
\)
contributes to weighted degree
\(
m_1+\cdots+m_s-(s-1)M
\)
with coefficient
\[
\prod_{\nu=2}^{s}
\left(
m_\nu+\cdots+m_s-(s-\nu)M
\right).
\]
This gives the stated formula.
\end{proof}

{
We have
\[
 \Delta_\infty(R)=\frac{P_{\infty}(R)^M-R^M}{M},
\]
and the positive zeros of $\Delta_\infty$ are exactly the positive fixed
points of $P_{\infty}$.
}  By the regular-transition estimate above, the cubic homogeneous terms of
weighted degree at most \(2n-3\) have the same coefficients as the truncated
outer maps.  Hence Lemma~\ref{lm:weighted-picard-brudnyi} gives
\begin{equation}
\Delta_\infty(R)
\sim
\sum_{j\ge0}C_j^\infty R^j .
\label{Cinf-def}
\end{equation}
Here the symbol \(\sim\) is used only for coefficient extraction; the
fourth- and higher-order parameter remainder is always taken from the exact
first return map, not from termwise integration of the singular outer
series.

\begin{lm}
\label{lm:low-coefficients-infty}
For \(q=0,\ldots,n-2\), the \(\varepsilon_a\beta^2\)-homogeneous part of
\(C_{2q+1}^\infty\) has the form
\begin{equation}
\bigl[C_{2q+1}^\infty\bigr]_{\varepsilon_a\beta^2}
=
\varepsilon_a
\left(
\Gamma_q^\infty a_n\beta_n\beta_{n-q}
+
\mathcal L_q^\infty(a,\beta_n,\ldots,\beta_{n-q+1})
\right),
\label{Codd-triangular-infty}
\end{equation}
where
\(
\Gamma_q^\infty<0,
\)
and where \(\mathcal L_q^\infty\) is independent of \(\beta_{n-q}\).
Moreover, no earlier coefficient
$
C_1^\infty,\ldots,C_{2q-1}^\infty
$
contains \(\beta_{n-q}\) in its \(\varepsilon_a\beta^2\)-homogeneous part.
\end{lm}

\begin{proof}
The mixed \(\varepsilon_a\beta^2\)-block of the \(R\)-equation is
\[
\varepsilon_a
\sum_{i,j,k=0}^n
\sum_{\ell\ge0}
a_i\beta_j\beta_k
G_{i,j,k,\ell}(\theta)
R^{5-2i-2j-2k+2\ell}.
\]
After passing to the weighted variable
\(
W= M^{-1}R^M,
\)
a monomial
\[
a_i\beta_j\beta_kG_{i,j,k,\ell}(\theta)
R^{5-2i-2j-2k+2\ell}
\]
contributes to the weighted displacement with degree
\[
M-1+5-2i-2j-2k+2\ell
=
6n+1-2(i+j+k)+2\ell .
\]
Hence it contributes to \(C_{2q+1}^\infty\) precisely when
\(
i+j+k=3n-q+\ell .
\)
Taking
\(
i=n,\) \(j=n\), \(k=n-q\), \(\ell=0,\)
we obtain the monomial
\(
a_n\beta_n\beta_{n-q}
\)
in \(C_{2q+1}^\infty\).

The same degree relation shows that if \(p<q\), then \(C_{2p+1}^\infty\)
cannot contain \(\beta_{n-q}\) in its \(\varepsilon_a\beta^2\)-homogeneous
part.  This gives the triangular dependence.

It remains to compute the sign of the coefficient of
\(a_n\beta_n\beta_{n-q}\).  Set
\(
L=\frac{\pi}{2}.
\)
On \(0<t<L\), define
\[
a(t)=-A_n(t),
\qquad
b_0(t)=-B_n(t),
\qquad
b_q(t)=-B_{n-q}(t),
\]
\[
d_q(t)=D_{n,n-q}(t),
\qquad
e_0(t)=E_{n,n}(t),
\qquad
e_q(t)=E_{n,n-q}(t),
\qquad
g_q(t)=-G_{n,n,n-q}(t).
\]
By the explicit formulas for \(A,B,D,E,G\), and since
\[
\sin t>0,\qquad
\cos t>0,\qquad
\chi_{\alpha,0}^{(m)}(t)>0
\quad
\text{for }0<t<L,
\]
one has
\[
a(t)>0,\qquad b_0(t)>0,\qquad b_q(t)>0,
\]
\[
d_q(t)>0,\qquad e_0(t)>0,\qquad e_q(t)>0,\qquad g_q(t)>0.
\]
Set
\[
B_0^*(t)=\int_0^t b_0(u)\,du,
\qquad
B_q^*(t)=\int_0^t b_q(u)\,du,
\qquad
D_q^*(t)=\int_0^t d_q(u)\,du .
\]
Then
\[
B_0^*(t)>0,\qquad B_q^*(t)>0,\qquad D_q^*(t)>0
\qquad
\text{for }0<t<L.
\]

We compute the ordered Picard integrals by splitting each ordered simplex
according to the quadrants containing its ordered variables.  Thus, for a
\(k\)-fold ordered integral over
\(
0\le t_1\le\cdots\le t_k\le2\pi,
\)
we split the domain into the subdomains where each \(t_\nu\) belongs to one
of the four intervals
\[
\left[0,\frac{\pi}{2}\right],\quad
\left[\frac{\pi}{2},\pi\right],\quad
\left[\pi,\frac{3\pi}{2}\right],\quad
\left[\frac{3\pi}{2},2\pi\right].
\]
Only after this splitting do we use the symmetry of the integrands.

The parity relations are
\[
A_n(\theta+\pi)=A_n(\theta),
\qquad
B_j(\theta+\pi)=-B_j(\theta),
\]
\[
D_{j,k}(\theta+\pi)=D_{j,k}(\theta),
\qquad
E_{i,j}(\theta+\pi)=-E_{i,j}(\theta),
\qquad
G_{i,j,k}(\theta+\pi)=G_{i,j,k}(\theta).
\]
The coefficient of \(a_n\beta_n\beta_{n-q}\) is
\(
\Gamma_q^\infty=-\mathcal K_{n,q},
\)
where
\[
\begin{aligned}
\mathcal K_{n,q}
={}&
4\nu_q\int_0^L g_q(t)\,dt
+
8\nu_q(n-q)\int_0^L a(t)D_q^*(t)\,dt
\\
&+
8(n-q-1)
\int_0^L
\left[
B_0^*(t)e_q(t)
+
b_q(t)\int_0^t a(u)B_0^*(u)\,du
\right]dt
\\
&+
8(1-\delta_{q0})(n+q-1)
\int_0^L
\left[
B_q^*(t)e_0(t)
+
b_0(t)\int_0^t a(u)B_q^*(u)\,du
\right]dt
\\
&+
8(n-q)
\int_0^L
a(t)
\left[
D_q^*(t)
+
\int_0^t B_0^*(u)b_q(u)\,du
\right]dt .
\end{aligned}
\]
Here
\(
\nu_0=1,\) and \(\nu_q=2\quad(q\ge1).
\)

Every term in \(\mathcal K_{n,q}\) is nonnegative, and at least the first,
second, and last lines are strictly positive.  Therefore
\(
\mathcal K_{n,q}>0.
\)
Consequently,
\(
\Gamma_q^\infty=-\mathcal K_{n,q}<0.
\)
This proves the lemma.
\end{proof}

We define the selected infinity coefficients by
\[
C_{2q+1}^*(a,\beta)
=
\Gamma_q^\infty a_n\beta_n\beta_{n-q}
+
\mathcal L_q^\infty(a,\beta_n,\ldots,\beta_{n-q+1}),
\qquad q=0,\ldots,n-2,
\]
and for convenience, set
\(
z=R^2,
\)
such that
\[
Q_\infty^*(z,a,\beta)
=
C_1^*(a,\beta)
+
C_3^*(a,\beta)z
+\cdots+
C_{2n-3}^*(a,\beta)z^{n-2}.
\]

\begin{lm}
\label{lm:C-infty-nonempty}
Fix \(a\in\mathbb R^{n+1}\) with \(a_n\neq0\).  Then one can choose
\(
\beta_n,\beta_{n-1},\ldots,\beta_2
\)
so that the polynomial
\(
Q_\infty^*(z,a,\beta)
\)
has \(n-2\) simple positive zeros.
\end{lm}

\begin{proof}
Choose numbers
\(
0<z_1<z_2<\cdots<z_{n-2}
\)
and set
\[
D(z)=d_0+d_1z+\cdots+d_{n-2}z^{n-2}
=
\lambda\prod_{j=1}^{n-2}(z-z_j),
\]
where \(\lambda\neq0\) will be chosen below.

Choose \(\lambda\) so that \(d_0\) has the same sign as
\(
\Gamma_0^\infty a_n .
\)
Then the equation
\(
\Gamma_0^\infty a_n\beta_n^2=d_0
\)
has a real solution \(\beta_n\neq0\).

Suppose that
\(
\beta_n,\beta_{n-1},\ldots,\beta_{n-q+1}
\)
have already been chosen.  Since
\(
\Gamma_q^\infty a_n\beta_n\neq0,
\)
the equation
\[
\Gamma_q^\infty a_n\beta_n\beta_{n-q}
+
\mathcal L_q^\infty(a,\beta_n,\ldots,\beta_{n-q+1})
=
d_q
\]
has a unique real solution for \(\beta_{n-q}\).  Iterating this for
\(
q=1,\ldots,n-2,
\)
we obtain
\[
C_{2q+1}^*(a,\beta)=d_q,
\qquad q=0,\ldots,n-2.
\]
Hence
\[
Q_\infty^*(z,a,\beta)=D(z),
\]
and \(Q_\infty^*(z,a,\beta)\) has the prescribed \(n-2\) simple positive zeros.
\end{proof}

The infinity admissible set is defined by
\[
\mathcal C_{\rm inf}
=
\left\{
(a,b,\varepsilon_a,\varepsilon_b):
\begin{array}{l}
	a_n\neq0,\quad
	\beta_i=\varepsilon_b b_i,\quad i=0,\ldots,n,\\
	\text{there exist }
	0<z_1^-<z_1^+<\cdots<z_{n-2}^-<z_{n-2}^+\ll1
	\text{ such that}\\
	Q_\infty^*(z_j^-,a,\beta)
	Q_\infty^*(z_j^+,a,\beta)<0,\\
	\operatorname{sgn}
	\Delta_\infty\!\left(\sqrt{z_j^\pm}\right)
	=
	\operatorname{sgn}
	Q_\infty^*(z_j^\pm,a,\beta),
	\qquad j=1,\ldots,n-2.
\end{array}
\right\}.
\]

\begin{lm}
\label{lm:infinity-cyclicity}
If
\(
(a,b,\varepsilon_a,\varepsilon_b)\in\mathcal C_{\rm inf},
\)
then system \eqref{or} has at least \(n-2\) limit cycles bifurcating from
infinity.
\end{lm}

\begin{proof}
By Lemma \ref{lm:C-infty-nonempty}, the polynomial
\[
\bar Q_\infty^*(z)=Q_\infty^*(z,a,\bar\beta)
\]
has \(n-2\) simple positive zeros.  Choose test points
\[
0<z_1^-<z_1^+<\cdots<z_{n-2}^-<z_{n-2}^+\ll1
\]
around these zeros so that
\[
\bar Q_\infty^*(z_j^-)\bar Q_\infty^*(z_j^+)<0,
\qquad j=1,\ldots,n-2.
\]

The definition of \(\mathcal C_{\rm inf}\) gives
\[
\operatorname{sgn}\Delta_\infty(R_j^\pm)
=
\operatorname{sgn}Q_\infty^*(z_j^\pm,a,\beta).
\]
Hence
\[
\Delta_\infty(R_j^-)\Delta_\infty(R_j^+)<0,
\qquad j=1,\ldots,n-2.
\]
By the intermediate value theorem, \(\Delta_\infty\) has at least one zero in
each interval
\(
(R_j^-,R_j^+).
\)
These intervals are pairwise disjoint.  Therefore \(\Delta_\infty\) has at
least \(n-2\) positive zeros.

Since
\[
\Delta_\infty(R)
=
{\frac1M\left(P_{\infty}(R)^M-R^M\right),}
\]
and \(R>0\), the zeros of \(\Delta_\infty\) are exactly the positive fixed
points of ${P_{\infty}}$.  Therefore system \eqref{or} has at least \(n-2\) limit
cycles near infinity.  Here ``near infinity'' refers to the sufficiently
large-energy part of the exterior period annulus.
\end{proof}

%%%%%%%%%%%%%%%%%%%%%%%%%%%%%%%%%%
%%%%%%%%%%%%%%%%%%%%%%%%%%%%%%%%%%
%%%%%%%%%%%%%%%%%%%%%%%%%%%%%%%%%%%%%
%%%%%%%%%%%%%%%%%%%%%%%%%%%%%%%%%%%%%%
\subsection[Proofs of the Hilbert-number results]{Proofs of Theorems \ref{t-hn5}--\ref{thm-hn} and Proposition \ref{prop:P3-not-origin-max}}
\begin{proof}[Proof of Theorem \ref{t-hn5}]
First, assume
\(
(a,b,\varepsilon_a,\varepsilon_b)\in\mathbb P_1.
\)
By Lemma \ref{ln5o}, the displacement map near the origin has at
least \(n-1\) positive simple zeros.  Hence the perturbed system has at least
\(n-1\) limit cycles near the origin.

Second, assume
\(
(a,b,\varepsilon_a,\varepsilon_b)\in\mathbb P_2.
\)
{
By Lemma~\ref{lm:infinity-cyclicity}, the exact weighted displacement
$\Delta_\infty$ has at least $n-2$ positive zeros.  These zeros are precisely
the positive fixed points of the first return map $P_{\infty}$.  Therefore the perturbed system has at least $n-2$ limit cycles bifurcating from infinity.
For $n=2$, this contribution is zero.
}

Third, assume
\(
(a,b,\varepsilon_a,\varepsilon_b)\in\mathbb P_3.
\)
Then ${a\in\mathcal C_{\rm lip}^a}$.  By Lemma \ref{hn5l}, the two Melnikov
functions \(M^-\) and \(M^+\) have altogether at least
\[
\eta_{\rm lip}(n)=n+\left[\dfrac n3\right]+\left[\dfrac {n+1}3\right]
\]
simple zeros near the two-cuspidal loop.  Hence the perturbed system has at
least \(\eta_{\rm lip}(n)\) limit cycles near the lips.
\end{proof}

{
\begin{proof}[Proof of Theorem~\ref{thm-p12}]
Choose a point of the nonempty origin admissible set with $a_n\neq0$; this is
possible because the triangular origin construction leaves an open choice of
the final nonzero coefficient.  For this fixed $a$ and $b\in \mathcal C_{\rm inf}$, Lemma~\ref{lm:C-infty-nonempty}
chooses $\beta_n,\ldots,\beta_2$ realizes the required infinity signs after a
common small scaling of the $\beta$-parameters.  Taking
\[
 0<|\beta|^2\ll\varepsilon_a\ll|\beta|\ll1
\]
keeps all quadratic $b$-contributions dominated in the origin construction
while preserving the infinity sign conditions.  Hence
$\mathbb P_1\cap\mathbb P_2\neq\emptyset$.
\end{proof}

\begin{proof}[Proof of Theorem~\ref{thm-hn}]
The set $\mathcal C_{\rm lip}^a$ is a nonempty open cone.  Since
$\{a_n=0\}$ is a proper hyperplane, choose
$a\in\mathcal C_{\rm lip}^a$ with $a_n\neq0$.  For this fixed $a$, choose the
$b\in \mathcal C_{\rm inf}$  and
\[
 0<\varepsilon_b^2\ll\varepsilon_a\ll\varepsilon_b\ll1.
\]
Both local constructions are then valid, so
$\mathbb P_2\cap\mathbb P_3\neq\emptyset$.  Combining Theorem~\ref{t-hn5}(ii)
and (iii) gives
\[
 H_0(2n+1,5)\ge(n-2)+\eta_{\rm lip}(n)=B(n).
\]
\end{proof}
}

\begin{proof}[Proof of Proposition \ref{prop:P3-not-origin-max}]
Let
\[
\mathcal O(a)=
\bigl(\mathcal O_0(a),\mathcal O_1(a),\ldots,\mathcal O_n(a)\bigr)
\]
be the selected first-order coefficient map used in the origin construction.
By the triangular structure of the origin coefficients,
\[
\mathcal O_j(a)
=
\gamma_j a_j+\mathcal L_j(a_0,\ldots,a_{j-1}),
\qquad
\gamma_j\neq0,
\qquad j=0,\ldots,n.
\]

Suppose, to the contrary, that the maximal lips condition and the maximal
origin condition can hold simultaneously for arbitrarily strong dominance.
{Then there exist $\rho_k\to0^+$ and normalized parameter vectors}
\[
a^{(k)}=(a_0^{(k)},\ldots,a_n^{(k)}),
\qquad |a^{(k)}|=1,
\]
such that both dominance conditions hold and, in particular,
\[
0<|\mathcal O_0(a^{(k)})|
\le
\rho_k|\mathcal O_1(a^{(k)})|
\le
\cdots
\le
\rho_k^n|\mathcal O_n(a^{(k)})|.
\]
Passing to a subsequence, we may assume that
\[
a^{(k)}\to a^*,
\qquad |a^*|=1.
\]
Since \(\mathcal O_n\) is linear and the sequence \(a^{(k)}\) is bounded,
\(\mathcal O_n(a^{(k)})\) is bounded.  Hence, for \(j=0,\ldots,n-1\),
\[
|\mathcal O_j(a^{(k)})|
\le
\rho_k^{\,n-j}|\mathcal O_n(a^{(k)})|
\longrightarrow0.
\]
Therefore
\[
\mathcal O_0(a^*)=\cdots=\mathcal O_{n-1}(a^*)=0.
\]
Using the triangular form of \(\mathcal O\), we obtain recursively
\(
a_0^*=a_1^*=\cdots=a_{n-1}^*=0.
\)
Thus
\(a^*=\alpha e_n\), \( \alpha\neq0,\)
where
\[
e_n=(0,\ldots,0,1)\in\mathbb R^{n+1}.
\]

On the other hand, the maximal lips dominance gives
\[
0<|\Lambda_0^+(a^{(k)})|
\le
\rho_k|\Lambda_1^+(a^{(k)})|
\le
\cdots
\le
\rho_k^n|\Lambda_n^+(a^{(k)})|.
\]
As before, since \(\Lambda_n^+\) is linear and \(a^{(k)}\) is bounded, this
implies
\(
\Lambda_0^+(a^*)=0.
\)
, where  {$\Lambda_0^+$ is the first selected plus-side coefficient.}
But
\[
\Lambda_0^+(a^*)
=
\alpha\Lambda_0^+(e_n).
\]
By the reduction formula for the Abelian integrals,
\[
\Lambda_0^+(e_n)
=
\sum_{\nu=0}^2 u_{\nu0}^+P_{n\nu}(h_0)
=
I_n^+(h_0),
\qquad h_0=\frac16.
\]
{
Let $\Gamma_{h_0}^+$ be the boundary curve of the period annulus at
$h_0=1/6$, and let $D^+$ be the bounded domain enclosed by
$\Gamma_{h_0}^+$.  Then Green's formula gives
\[
 I_n^+(h_0)
 =\oint_{\Gamma_{h_0}^+}x^{2n}y\,dx
 =-\iint_{D^+}x^{2n}\,dx\,dy
 \neq0.
\]
}
Therefore
\(
\Lambda_0^+(e_n)\neq0.
\)
Since \(\alpha\neq0\), we get
\(
\Lambda_0^+(a^*)\neq0,
\)
which contradicts \(\Lambda_0^+(a^*)=0\).

Thus the maximal lips strong-dominance condition and the maximal origin
strong-dominance condition cannot hold simultaneously.
\end{proof}

\section*{Acknowledgements}
We would like to express our sincere gratitude to Professor Maoan Han for his insightful comments, which have significantly enhanced the professionalism of this paper. The first, second and fourth authors are partially supported by the National Natural Science Foundation of China (Nos. 12571190, 12322109) and the Science and Technology Innovation Program of Hunan Province (No. 2023RC3040). The third author is partially supported by ERC Grant \#885707.

	\clearpage
	\appendix
	
	{
		\section{Appendix}
		\label{app:poincare}
		
		\subsection{Definitions}
		Let us recall the concepts of the Poincar\'{e}  transformation and the Poincar\'{e} disc referenced in \cite{DH, ZDHD}.
		Consider a unit sphere $\Sigma$ tangent to the plane $\pi$, where  $\pi$ is at its south pole. 
		A tangent plane $\prod_x$ is constructed such that it is tangent to  $\Sigma$ at its east pole.		
		By projecting lines from the center of $\Sigma$ to any point on $\pi$ (excluding points on the $y$-axis), we obtain a one-to-one correspondence between points on $\pi $ and points on $\prod_x$. This mapping allows us to project trajectories from $\pi$ (excluding the $y$-axis) onto $\prod_x$, including points at infinity.
		Similarly, for the point at infinity on the y-axis of the phase plane $\pi$, we can also find a corresponding transformation that maps it to the origin of the tangent plane $\prod_y$, where $\prod_y$ is a plane that is tangent to $\Sigma$ and intersects the phase plane $\pi$ perpendicularly at $y=1$.
		\begin{defi}
			In order to study situations at infinity on the phase plane $\pi$, we usually perform the following Poincar\'{e} transformations:
			
			For points on the $x$-axis:
			\[
			x=\frac1z,\qquad y=\frac uz.
			\] 
			
			For points on the $y$-axis:  
			\[
			x=\frac{v}{z},~~y=\frac{1}{z},~~z\neq0.
			\]
			The first transformation is for points on the $x$-axis, and the second one is for points on the $y$-axis.
			These transformations map points at infinity on $\pi$ to finite points on the corresponding tangent planes.
		\end{defi}
		It is worth noting that a line connecting the center of the sphere $\Sigma$ to any point on the phase plane $\pi$ will have an intersection point on the southern hemisphere of $\Sigma$.		
		\begin{defi}
			To better understand the global phase portraits of the trajectory on the phase plane $\pi$, we project $\Sigma$  vertically into $\pi$, forming a unit disc centered at the south pole of  $\Sigma$,  which is called the Poincar\'{e} disc. In this way, $\pi$ can be projected into the interior of the disc. The points at infinity on $\pi$ correspond to pairs of  diametrically opposite points on the circumference of the disc.
		\end{defi}
		
		We now define the concepts of center at infinity and focus at infinity, following \cite{BR}.
		When the equator is invariant, trajectories spiral around the equator, and there are no singularities on the equator or no other trajectories passing through singularities on the equator, the polynomial system  behaves as if the equator were collapsed to  a point $P_{\infty}$.
		\begin{defi}
			$P_{\infty}$ is a center at infinity of a planar differential equation if there is a neighborhood $U$ of the unit circle in the Poincar\'{e} disc that is filled with closed orbits, as shown in {\rm Fig. \ref{cf}(a)}.
		\end{defi}

		\begin{defi}
			$P_{\infty}$ is a stable (resp., an unstable) focus at infinity of a planar differential equation if there exists a neighborhood $U$ around the unit circle of the Poincar\'{e} disc, such that the $\omega$-limit (resp., $\alpha$-limit) set of every orbit in $U$ is the unit circle of the Poincar\'{e} disc, as shown in {\rm Fig. \ref{cf}(b) (resp., (c))}.
		\end{defi}	
		\begin{figure}[!htb]
			\centering
			\subfloat[A center at infinity]
			{\includegraphics[scale=0.7]{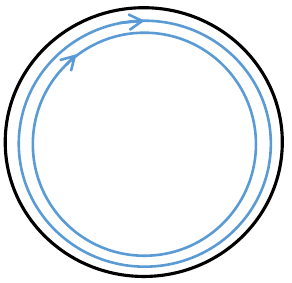}}\hspace{25pt}
			\subfloat[A  stable focus at infinity]
			{\includegraphics[scale=0.7]{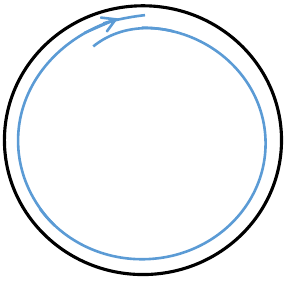}}\hspace{25pt}
			\subfloat[An  unstable focus at infinity]
			{\includegraphics[scale=0.7]{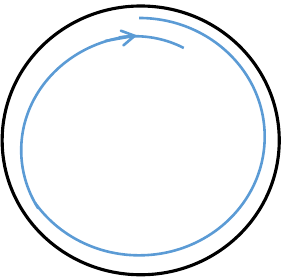}}\hspace{25pt}
			\caption{{\footnotesize The center--focus problem at infinity.}}
			\label{cf}
		\end{figure}

		\begin{defi}
			Consider a planar analytic differential system
			\[
			\dot{x}=P(x,y),\qquad \dot{y}=Q(x,y),
			\]
			with an isolated center at \(O=(0,0)\), and let
			\(J=D(P,Q)(O)\) be its Jacobian matrix. The center is called a
			{\bf linear-type center} if \(J\) has a nonzero pair of purely
			imaginary eigenvalues, equivalently
			\[
			\operatorname{tr}J=0,\qquad \det J>0.
			\]
			It is called a {\bf nilpotent center} if \(J\neq0\) is nilpotent.
			These definitions concern the linear part at the equilibrium; the
			full analytic vector field need not be linear.
		\end{defi}

\subsection{Auxiliary results}
We record a separate realization result on the regular outer period
annulus. Here the vector field \eqref{or} is considered with the alternative
parameter scale \(0<\varepsilon_b\ll\varepsilon_a\ll1\).

{
\renewcommand{\thethm}{A}
\begin{thm}\label{thm-mid}
For every integer \(n\geq1\), there exist real coefficients
\(a_0,\ldots,a_n\), with \(a_n\neq0\), and sufficiently small parameters
\(0<\varepsilon_b\ll\varepsilon_a\ll1\) such that system \eqref{or}
has at least \(n\) limit cycles bifurcating from the period annulus
\(\{H=h:h\in(1/6,\infty)\}\).
\end{thm}
}

Orient the outer ovals \(\Gamma(h)=\{H=h\}\) clockwise and put
\[
J_k(h)=\oint_{\Gamma(h)}x^{2k}y\,{\rm d}x,\qquad
M_n(h)=\sum_{k=0}^{n}a_kJ_k(h).
\]
The odd damping terms have zero Abelian integral by symmetry, so the
first-order displacement is \(-\varepsilon_a M_n(h)\).

\begin{lm}\label{lm-middle}
For every integer \(n\geq1\), the coefficients \(a_0,\ldots,a_n\) can be
chosen, with \(a_n\neq0\), so that \(M_n\) has at least \(n\) simple zeros
in \((1/6,\infty)\).
\end{lm}

\begin{proof}
Write \(V(x)=((x^2-1)^3+1)/6\) and
\(r(h)^2=1+(6h-1)^{1/3}\). Then
\[
J_k(h)=2\int_{-r(h)}^{r(h)}
x^{2k}\sqrt{2(h-V(x))}\,{\rm d}x>0.
\]
Dominated convergence and the substitution \(t=x^2\) give
\[
J_k(1/6+)=\frac{2}{\sqrt3}B\left(k+\frac12,\frac52\right),
\qquad
\frac{J_1(h)}{J_0(h)}\longrightarrow\frac16
\quad(h\downarrow1/6).
\]
Moreover, setting \(x=r(h)u\), using \(r(h)\sim(6h)^{1/6}\), and
noting that \(V(r(h)u)/h\to u^6\) uniformly on \([-1,1]\), we obtain
\[
J_k(h)\sim C_kh^{(k+2)/3},\qquad
C_k=2\sqrt2\,6^{(2k+1)/6}\int_{-1}^{1}u^{2k}\sqrt{1-u^6}\,{\rm d}u>0.
\]
In particular, \(J_j(h)/J_k(h)\to0\) as \(h\to\infty\) whenever \(j<k\).

For \(n=1\), choose \(c>1/6\). The function \(M_1=J_1-cJ_0\) is
negative near \(1/6\) and positive for large \(h\), so it has a
sign-changing zero. Suppose inductively that \(M_{n-1}\) has opposite
nonzero signs at the endpoints of \(n-1\) disjoint intervals
\([\alpha_j,\beta_j]\subset(1/6,\infty)\).
Choose \(H>\max_j\beta_j\) with \(M_{n-1}(H)\neq0\), and choose
\(a_n\neq0\) with sign opposite to \(M_{n-1}(H)\). Its magnitude can be
made small enough that
\[
|a_n|<
\frac12\min_{t\in K}\frac{|M_{n-1}(t)|}{J_n(t)},\qquad
K=\{\alpha_j,\beta_j:1\leq j<n\}\cup\{H\}.
\]
Thus \(M_n=M_{n-1}+a_nJ_n\) preserves all these endpoint signs.
Since \(M_n(h)/J_n(h)\to a_n\), it has another sign-changing zero
to the right of \(H\). This proves the induction.

Finally, take a sufficiently small regular value \(\delta\) of the
analytic function \(M_n/J_0\) on the union of the chosen intervals.
Replacing \(a_0\) by \(a_0-\delta\) preserves all endpoint signs and
makes the corresponding zeros simple.
\end{proof}

\begin{proof}[Proof of Theorem~\ref{thm-mid}]
Choose the coefficients given by Lemma~\ref{lm-middle} and fix the
\(b\)-coefficients. On compact regular annuli containing the selected
zeros, the displacement is
\(-\varepsilon_a M_n+o(\varepsilon_a)\) in \(C^1\) as
\(\varepsilon_a\to0\) and \(\varepsilon_b/\varepsilon_a\to0\).
Each simple zero therefore yields a limit cycle.
\end{proof}

}


{\small	\begin{thebibliography}{99}
		\bibitem{ALP}
		W. Aziz, J. Llibre, C. Pantazi, Centers of quasi-homogeneous polynomial differential equations of degree
		three, {\it Adv. Math.}, {\bf 254} (2014), 233-250.
		
		\bibitem{BR}
		T. R. Blows, C. Rousseau, Bifurcation at infinity in polynomial vector fields, {\it J. Differ. Equ.}, {\bf 104}
		(1993), 215-242.
		\bibitem{BL}
		T.R. Blows, N.G. Lloyd, The number of small-amplitude limit cycles of Li\'{e}nard equations, Math. Proc. Camb. Philos. Soc., {\bf95} (1984) 359–366.
		
		\bibitem{BPY}
		M. Briskin, F. Pakovich, Y. Yomdin, Algebraic geometry of the center focus problem for Abel differential equations, {\it Ergod. Theor. Dyn. Syst.}, {\bf 36} (2016), 714-744.
		
		\bibitem{BRY}
		M. Briskin, N. Roytvarf, Y. Yomdin, Center conditions at infinity for Abel differential
		equations, {\it Annals Math.}, {\bf 172} (2010), 437-483.
		\bibitem{Br}
		A. Brudnyi, On the center problem for ordinary differential equations, {\it Amer. J. Math.}, {\bf 128} (2006), 419-451.	
	
	\bibitem{Br1}
A. Brudnyi, An explicit expression for the first return map in the center problem, {\it J. Differ. Equ.}, {\bf 206} (2004), 306–314.
        
		\bibitem{CLZ}
		H. Chen, Z. Li, R. Zhang,  Establishing definitive conditions for global centers in generalized polynomial Li\'{e}nard systems,
		{\it Discrete Contin. Dyn. Syst. Ser. B}, {\bf30} (2025), 1314–1340.
		
		\bibitem{CFLZ}
		H. Chen, Z. Feng, Y. Lu, R. Zhang, Dynamics near infinity of polynomial Li\'{e}nard systems, submitted.	
		
		\bibitem{CLV}	
		I.E. Colak, J. Llibre, C. Valls, Hamiltonian nilpotent centers of linear plus cubic homogeneous polynomial
		vector fields, {\it Adv. Math.}, {\bf 259} (2014), 655-687.
		
		\bibitem{Conti98}	
		R. Conti, Centers of planar polynomial systems. {\it  A review, Matematiche} {\rm LIII}, (1998) 207-240.
		
		\bibitem{Dulac}	
		H. Dulac, D\'etermination et integration d’une certaine classe d\'equations diff\'erentielle ayant par point singulier un
		centre, {\it Bull. Sci. Math. S\'er. (2)}, {\bf 32} (1908), 230-252.
		
		
		\bibitem{DH}
		F. Dumortier, C. Herssens, Polynomial Li\'{e}nard equations near infinity, {\it J. Differ. Equ.}, {\bf 153}
		(1999), 1-29.
		
		\bibitem{Gine2017}
		J.  Gin\'e, Center conditions for polynomial Li\'enard systems, {\it Qual. Theory Dyn. Syst.}, {\bf 16} (2017), 119-26.


        
	
		\bibitem{Hale}
		J. K. Hale, {\it Ordinary Differential Equations}, Robert E. Krieger Publishing Company, New York, 1980.
		
		
		\bibitem{H1}
		M. Han, Liapunov constants and Hopf cyclicity of Li\'{e}nard systems, {\it Ann. Differ. Equ.}, {\bf15} (1999), 113–126.
		
		\bibitem{HR}	
		M. Han, V.G. Romanovski, On the number of limit cycles of polynomial Li\'{e}nard systems, {\it Nonlinear Anal. Real World Appl.}, {\bf14} (2013), 1655-1668.
		
		\bibitem{HSYC}
	M. Han, C. Shu, J. Yang, A. Chian, 	Polynomial Hamiltonian system with a nilpotent critical point,  {\it Adv. Space Res.}, {\bf46} (2010), 521–525.
		

		
		\bibitem{HZY}
		M. Han, H. Zang, J. Yang, Limit cycle bifurcations by perturbing a cuspidal loop in a Hamiltonian system, {\it J. Differ. Equ.}, {\bf246} (2009), 129–163.
		
		
		
		\bibitem{HLX}
		H. He, J. Llibre, D. Xiao, Hamiltonian polynomial differential systems with global centers in the plane, {\it Sci. China
			Math.},  {\bf 48} (2021), 2018.
		
		\bibitem{HXi}	
		H.	He, D. Xiao, On the global center of planar polynomial differential systems and the related problems, {\it J. Appl. Anal. Comput.}, {\bf 12} (2022),  1141-1157.        


        \bibitem{IL}
        Yu. Ilyashenko, Normal forms for local families and nonlocal bifurcations, {\it Astérisque}, {\bf222} (1994),  233-258.
		
		\bibitem{I}
		Yu. Ilyashenko, Centennial history of Hilbert's 16th problems, {\it Bull. Amer. Math. Soc.}, {\bf39} (2002), 254-301.
		
		\bibitem{IK}
		Yu. Ilyashenko,  V. Kaloshin, {\it Bifurcation of planar and spatial polycycles: Arnold's program and its development}, The Arnoldfest (Toronto, ON, 1997), Amer. Math. Soc., Providence, RI, 1999. 
		
		\bibitem{IY}
		Yu. Ilyashenko, S. Yakovenko, {\it Concerning the Hilbert 16th problems},   Amer. Math. Soc. Transl. Ser. 2, Amer. Math. Soc., Providence, RI, 1995.
		
		\bibitem{K}
		V. Kaloshin, The Existential Hilbert 16-th problem and an estimate for cyclicity of elementary polycycles, {\it Invent. Math.,} {\bf151} (2003), 451-512. 
		
		\bibitem{LY}
		L. Li and J. Yang, On the number of limit cycles for a quintic Li\'enard system under polynomial perturbations, {\it J. Appl. Anal. Comput.}, {\bf 9} (2019), 2464-2481.
		
		\bibitem{LMP}
		A. Lins, W. de Melo, C.C. Pugh, On Li\'enard equation, \textit{Lecture Notes in Math.}, \textbf{597} (1977), 335-357.
		
		\bibitem{LMT}
		J. Llibre, A.C. Mereu, M.A. Teixeira, Limit cycles of the generalized polynomial Li\'{e}nard differential equations, {\it Math. Proc. Cambridge Philos. Soc.}, {\bf148} (2010), 363–383.  
		
		
		\bibitem{LV}
		J. Llibre,  C. Valls, Global centers of the generalized polynomial Li\'enard differential systems, {\it J. Differ. Equ.}, {\bf 330} (2022), 66-80.
		
		\bibitem{M}
		R. Moussu, Sym\'{e}trie et forme normale des centres et foyers d\'{e}g\'{e}n\'{e}r\'{e}s, {\it Ergod. Theor. Dyn. Syst.}, {\bf2} (1982), 241-251.
		
		\bibitem{MY}
		Deyue Ma, Junmin Yang, Some properties of Melnikov function and the number of limit cycles near a heteroclinic loop with two nilpotent cusps, {\it Qual. Theor. Dyn. Syst.,} {\bf 24} (2025), 190.
		
		\bibitem{P}
		H. Poincar\'{e}, M\'{e}moire sur les courbes d\'{e}finies par les \'{e}quations diff\'{e}rentielles, {\it J. Math.}, {\bf37} (1881) 375-422; Oeuvres de Henri Poincar\'{e}, Gauthier-Villars, Paris, 1951, 3-84.
		
		
		\bibitem{Smale91}
		S. Smale, Dynamics retrospective: great problems, attempts that failed, \textit{Phys. D}, \textbf{51} (1991), 267-273.
		
		\bibitem{Smale98}
		S. Smale, Mathematical problems for the next century, \textit{Math. Intell.}, \textbf{20} (1998), 7-15.
		
			\bibitem{TY}
			M.A. Teixeira, J. Yang, The center focus problem and reversibility, {\it J. Differ. Equ.}, {\bf 174} (2001), 237-251.
			
		\bibitem{X1}
		Y. Xiong, Bifurcation of Limit cycles by perturbing a class of hyper-elliptic Hamiltonian systems of degree five, {\it J. Math. Anal. Appl.}, {\bf 411} (2014), 559-573.
		
		\bibitem{X2}
		Y. Xiong, On the number of limit cycles near a homoclinic loop with a nilpotent cusp of order $m$, {\it J. Differ. Equ.}, {\bf380}  (2024), 146-180.
		
		\bibitem{XH}
		Y. Xiong, M. Han, New lower bounds for the Hilbert number of polynomial systems of Liénard type, {\it J. Differ. Equ.}, {\bf257} (2014), 2565-2590.
		
				\bibitem{YH}
		J. Yang, X. Hu, Limit cycle bifurcations near a heteroclinic loop with two nilpotent cusps of general order, {\it Internat. J. Bifur. Chaos Appl. Sci. Engrg.}, {\bf 32} (2022), 2250083.
		
		
		\bibitem{ZDHD}
		Z. Zhang, T. Ding, W. Huang, Z. Dong,
		{\it Qualitative Theory of Differential Equations},
		Transl. Math. Monogr. 101,
		Amer. Math. Soc., Providence, RI, 1992.
		
	\end{thebibliography}

}
\end{document}